\documentclass[11pt]{article}

\usepackage{epsfig}
\usepackage{amssymb, latexsym}
\usepackage{amscd}
\usepackage[all]{xy}
\usepackage{epsf}
\usepackage[mathscr]{eucal}

\usepackage{tikz}
\usepackage{tikz-cd}
\usepackage{verbatim}
\usetikzlibrary{matrix,arrows}
\usepackage{mathtools}
\usepackage{xcolor}
\usepackage{stmaryrd}
\usetikzlibrary{positioning,fit,calc}

\DeclareFontFamily{U}{solomos}{}
\DeclareErrorFont{U}{solomos}{m}{n}{10}
\DeclareFontShape{U}{solomos}{m}{n}{
  <-> s*[1.1]  gsolomos8r
}{}

\usepackage{amsmath,amsfonts,amscd,amssymb,amsthm}

\usepackage[titletoc]{appendix}
\usepackage[sc]{titlesec}

\long\def\comment#1\endcomment{}

\theoremstyle{plain}

\newtheorem{theorem}{\sc Theorem}[section]
\newtheorem{lemma}[theorem]{\sc Lemma}

\newtheorem{prop}[theorem]{\sc Proposition}
\newtheorem{coroll}[theorem]{\sc Corollary}

\makeatletter
\newcommand*{\doublerightarrow}[2]{\mathrel{
  \settowidth{\@tempdima}{$\scriptstyle#1$}
  \settowidth{\@tempdimb}{$\scriptstyle#2$}
  \ifdim\@tempdimb>\@tempdima \@tempdima=\@tempdimb\fi
  \mathop{\vcenter{
    \offinterlineskip\ialign{\hbox to\dimexpr\@tempdima+1em{##}\cr
    \rightarrowfill\cr\noalign{\kern.5ex}
    \rightarrowfill\cr}}}\limits^{\!#1}_{\!#2}}}
\makeatother

\theoremstyle{plain}

\newtheorem{defn}[theorem]{\sc Definition}

\theoremstyle{exercise}
\newtheorem{remark}[theorem]{\sc Remark}

\newtheorem{example}[theorem]{\sc Example}

\makeatletter \@addtoreset{equation}{section} \makeatother

\def\eqref#1{\thetag{\ref{#1}}}

\let\latexref=\ref
\def\ref#1{{\normalfont{\latexref{#1}}}}

\newcommand{\ldot}{{\:\raisebox{2,3pt}{\text{\circle*{1.5}}}}}
\newcommand{\udot}{{\:\raisebox{3pt}{\text{\circle*{1.5}}}}}
\def\dlim_#1{{\displaystyle\lim_{#1}}^\hdot}

\newcommand{\End}{\operatorname{End}}

\newcommand{\id}{\operatorname{\rm id}}

\newcommand{\Mor}{\mathrm{Mor}}
\newcommand{\Ob}{\mathrm{Ob}}

\newcommand{\Hom}{\mathrm{Hom}}

\newcommand{\Hoch}{\mathrm{Hoch}}
\newcommand{\Lie}{\mathrm{Lie}}

\newcommand{\op}{\mathrm{op}}

\newcommand{\dg}{\mathrm{dg}}

\newcommand{\Sets}{\mathscr{S}ets}

\newcommand{\Vect}{\mathscr{V}ect}

\newcommand{\Cat}{{\mathscr{C}at}}

\renewcommand{\top}{\mathrm{top}}

\newcommand{\Fun}{{\mathrm{Fun}}}

\renewcommand{\k}{\Bbbk}

\newcommand{\pprime}{{\prime\prime}}

\newcommand{\tot}{\mathrm{tot}}

\newcommand{\coh}{\mathrm{coh}}
\newcommand{\Coh}{{\mathscr{C}{oh}}}

\newcommand{\sotimes}{{\overset{\sim}{\otimes}}}

\newcommand{\Hot}{{\mathrm{Hot}}}
\newcommand{\Ord}{\mathrm{Ord}}
\newcommand{\Grph}{{\mathscr{G}rph}}
\newcommand{\HHom}{{\mathscr{H}om}}
\newcommand{\EEnd}{{\mathscr{E}nd}}
\newcommand{\FFun}{{\mathscr{F}un}}
\newcommand{\Tot}{\mathrm{Tot}}
\newcommand{\RRHom}{{\mathbb{R}\underline{\mathrm{Hom}}}}
\newcommand{\coll}{\mathrm{coll}}

\newcommand{\Span}{{\mathscr{S}pan}}
\newcommand{\Sym}{\mathrm{Sym}}
\newcommand{\Disk}{\mathrm{Disk}}

\renewcommand{\min}{\mathrm{min}}
\renewcommand{\max}{\mathrm{max}}
\newcommand{\Br}{\mathrm{Br}}
\newcommand{\Op}{\mathrm{Op}}
\newcommand{\circD}{D^{\circ}}

\newcommand{\Tree}{\mathbf{Tree}}
\newcommand{\Out}{\mathrm{Out}}
\newcommand{\Symm}{\mathrm{Symm}}
\newcommand{\Des}{\mathrm{Des}}
\newcommand{\Trees}{\mathbf{Tree}}
\newcommand{\dom}{\mathrm{dom}}
\newcommand{\Glob}{\mathbf{Glob}}

\newcommand{\sevafigc}[4]{\begin{figure}[h]\centerline{
 \epsfig{file=#1,width=#2,angle=#3}}
\bigskip\caption{#4}\end{figure}}

\title{\sc{The twisted tensor product of dg categories\\and a contractible 2-operad}}

\author{\sc{Boris Shoikhet}}
\date{}

\begin{document}\maketitle
{\footnotesize
\begin{center}{\parbox{4,5in}{{\sc Abstract.}
It is well-known that the ``pre-2-category'' $\Cat_\dg^\coh(\k)$ of small dg categories over a field $\k$, with 1-morphisms defined as dg functors, and 2-morphisms defined as the complexes of coherent natural transformations, fails to be a strict 2-category. 
The question ``What do dg categories form'',  raised by V.Drinfeld in [Dr], is interpreted in this context as a question of finding a weak 2-category structure on $\Cat_\dg^\coh(\k)$.
In [T2], D.Tamarkin proposed an answer to this question, by constructing a contractible 2-operad in the sense of M.Batanin [Ba3], acting on $\Cat_\dg^\coh(\k)$. 
 
In this paper, we construct {\it another} contractible 2-operad, acting on $\Cat_\dg^\coh(\k)$. 
Our main tool is the {\it twisted tensor product} of small dg categories, introduced in [Sh3]. We establish a one-side associativity for the twisted tensor product, making $(\Cat_\dg^\coh(\k),\sotimes)$ a skew monoidal category in the sense of [S],  and construct a {\it twisted composition} $\Coh_\dg(D,E)\sotimes\Coh_\dg(C,D)\to\Coh_\dg(C,E)$, and prove some compatibility between these two structures. Taken together,  the two structures give rise to a 2-operad $\mathcal{O}$, acting on $\Cat_\dg^\coh(\k)$. Its contractibility is a consequence of a general result of [Sh3].

}}
\end{center}
}

\section{\sc Introduction}\label{section0}
\subsection{\sc }
In this paper, we further investigate the {\it twisted tensor product} of small differential graded (dg) categories over a field $\k$, which was recently introduced in [Sh3]. Recall that the twisted tensor product $C\sotimes D$ fulfils the adjunction 
\begin{equation}\label{eqx01}
\Fun_\dg(C\sotimes D, E)\simeq \Fun_\dg(C,\Coh_\dg(D,E))
\end{equation}
where $\Fun_\dg(-,-)$ is the set of dg functors, and $\Coh_\dg(D,E)$ is the dg category whose objects are the dg functors $f\colon D\to E$, and whose morphisms $f\Rightarrow g$ are given by the reduced Hochschild cochains on $D$ with coefficients in the $D^\op\otimes D$-module $E(f(-),g(-))$. The closed elements in  $\Coh_\dg(D,E)(f,g)$ are thought of as {\it derived} natural transformations from $f$ to $g$. Such derived complexes were introduced, for simplicial enrichment, by Cordier and Porter (see [CP] and the references for earlier papers therein), and were studied for a general enrichment in [Ba1,2], [St]. Unlike for the simplicial enrichment, for the dg enrichment there is an associative (vertical) composition, making $\Coh_\dg(-,-)$ a dg category. It follows from [Fa, Th.1.7] that, for $D$ cofibrant for the Tabuada closed model structure [Tab], the dg category $\Coh_\dg(D,E)$ is isomorphic in the homotopy category to the dg category $\RRHom(D,E)$ introduced by To\"{e}n in [To]. In particular, for $D,D^\prime$ cofibrant, and $w_1\colon D\to D^\prime, w_2\colon E\to E^\prime$ quasi-equivalences, the dg functors $w_1^*\colon \Coh_\dg(D^\prime,E)\to \Coh_\dg(D,E)$ and $w_{2*}\colon \Coh_\dg(D,E)\to \Coh_\dg(D,E^\prime)$ are quasi-equivalences.

It is worthy to compare adjunction \eqref{eqx01} with the adjunction proven in [To, Sect. 6]\footnote{To\"{e}n proved in [To, Cor.6.4] a much stronger statement than \eqref{eqx02}.}:
\begin{equation}\label{eqx02}
\Hot(C\otimes D,E)\simeq\Hot(C,\RRHom(D,E))
\end{equation}
where $\Hot$ stands for (the set valued external $\Hom$ in) the homotopy category of the category $\Cat_\dg(\k)$ of small dg categories over $\k$, with formally inverted quasi-equivalences.

We stress that, unlike \eqref{eqx02}, the adjunction \eqref{eqx01} holds in the category $\Cat_\dg(\k)$ itself, not in its homotopy category. It makes our $C\sotimes D$ non-symmetric in $C$ and $D$. 

One always has a dg functor $p\colon C\sotimes D\to C\otimes D$. Recall our main result in [Sh3]:
\begin{theorem}\label{theoremx1}
For $C,D$ cofibrant, the dg functor $p\colon C\sotimes D\to C\otimes D$ is a quasi-equivalence.
\end{theorem}

We recall the construction of $C\sotimes D$ in Section \ref{section11}.

\subsection{\sc }
In this paper, we construct, by means of the twisted tensor product, a homotopically final 2-operad, in the sense of Batanin [Ba3-5], which acts on the ``pre-2-category'' $\Cat_\dg^\coh(\k)$ of small dg categories over $\k$, whose objects are small dg categories over $\k$, whose morphisms are dg functors, and whose complex of 2-morphisms $f\Rightarrow g\colon C\to D$ is defined as $\Coh_\dg(C,D)(f,g)$. That is, our 2-operad solves the same problem as the 2-operad constructed by Tamarkin in [T2]. However, the 2-operads are distinct; we prove in [Sh5] that our dg 2-operad is isomorphic to the normalized 2-operadic analogue of the brace operad of [BBM], and is quasi-isomorphic to (the dg condensation of) the Tamarkin 2-operad. 

The pre-2-category $\Coh_\dg^\coh(\k)$ fails to be a strict 2-category. In particular, we can not define the horizontal composition $\Coh_\dg(D,E)(g_1,g_2)\otimes\Coh_\dg(C,D)(f_1,f_2)\to \Coh_\dg(C,E)(g_1f_1,g_2f_2)$. There are 2 candidates for such horizontal composition, and there is a homotopy between them. Therefore, the ``minimal'' space of ``all possible horizontal compositions'' as above is the complex 
\begin{equation}\label{eq02}
0\to\underset{\deg=-1}{\k}\xrightarrow{d}\underset{\deg=0}{\k\oplus\k}\to 0
\end{equation}
$$
d(x_{12})=x_2-x_1
$$
where $x_1,x_2,x_{12}$ are generators in the corresponding vector spaces.
Note that the cohomology of this complex is $\k[0]$; therefore, ``homotopically the operation is unique''. The reader is referred to Section \ref{sectionassoc} and to [Sh6], Sect.4.2 for more detail on the two horizontal compositions and the homotopy between them, as well as for graphical illustrations for the corresponding cochains.

\subsection{\sc}
One of our main new observations is existence of a canonical dg functor
\begin{equation}\label{eqx03}
M\colon\Coh_\dg(D,E)\sotimes\Coh_\dg(C,D)\to\Coh_\dg(C,E)
\end{equation}
called {\it the twisted composition}, with nice properties.

It is associative in the sense specified in Theorem \ref{theorem3} below. Let $\Psi\in \Coh_\dg(C,D)(f_1,f_2)$, $\Theta\in \Coh_\dg(D,E)(g_1,g_2)$. Then the two horizontal compositions are
$M((\Theta\otimes \id_{f_2})*(\id_{g_1}\otimes \Psi))$ and $M((\id_{g_2}\otimes\Psi)*(\Theta\otimes \id_{f_1}))$, correspondingly, and the homotopy between them is $M(\varepsilon(\Theta;\Psi))$ (see Section \ref{section11} for the notations for $\sotimes$ used here). The  morphism $\varepsilon(\Theta;\ \Psi_1,\dots,\Psi_n)$ is sent by $M$ to the {\it brace operation} $\Theta\{\Psi_1,\dots,\Psi_n\}$.

In fact, existence of such twisted composition is a formal consequence of a {\it closed skew monoidal} category structure on $\Cat_\dg(\k)$ with the skew monoidal product given by the twisted tensor product $\sotimes$, and with the right adjoint to $-\sotimes X$ given by $\Coh_\dg(X,=)$, see Section \ref{closedskewmonoidalintro}. Namely, any closed skew monoidal category is enriched over itself, which gives the twisted composition map, see Proposition \ref{propsm}. The skew-monoidality of the twisted tensor product means that there is an associator 
\begin{equation}\label{xassoc}
\alpha_{C,D,E}\colon (C\sotimes D)\sotimes E\to C\sotimes (D\sotimes E)
\end{equation}
{\it which is in general not an isomorphism} nor a quasi-isomorphism, which fulfils the hexagon and unit axioms.

One can iterate the twisted composition, and get a canonical dg functor 
\begin{equation}\label{eqx03bis}
M\colon \Coh_\dg(C_{k-1},C_k)\sotimes (\Coh_\dg(C_{k-2},C_{k-1})\sotimes(\dots\sotimes \Coh_\dg(C_0,C_1)\dots))\to\Coh_\dg(C_0,C_k)
\end{equation}
Let $n_1,\dots,n_k\ge 1, k\ge 1$. It defines a 2-globular pasting diagram $D=(n_1,\dots,n_k)$, see Figure \ref{fig1}.
Assume we are given dg categories $C_0,\dots,C_k$, dg functors $\{F_{i,j}\colon C_{i-1}\to C_i\}_{\substack{{i=1\dots k}\\{j=0,\dots, n_i}}}$, and elements $\{\Psi_{ij}\colon F_{ij}\Rightarrow F_{i,j+1}\in\Coh_\dg(C_{i-1},C_i)\}_{\substack{{i=1\dots k}\\{j=0,\dots, n_i-1}}}$.

One imagines this data by the following 2-globular pasting diagram, augmented by the 2-arrows $\Psi_{ij}\colon F_{i,j-1}\to F_{ij}$
\sevafigc{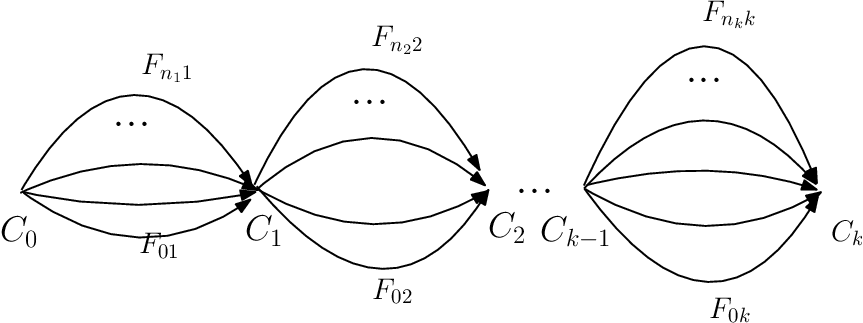}{100mm}{0}{The 2-globular pasting diagram $D=(n_1,\dots,n_k)$\label{fig1}}

We want to find the most general composition of all $\{\Psi_{ij}\}$ which should be an element in 
$\Coh_\dg(C_0,C_k)(F_\min,F_\max)$ where $F_\min=F_{k,0}\circ\dots\circ F_{1,0}$, $F_\max=F_{k,n_k}\circ\dots\circ F_{1,n_1}$. 
In virtue of the map \eqref{eqx03bis}, we have to ``cook up" an element in the l.h.s. of it, by the coherent 2-arrows $\{\Psi_{ij}\}$, and apply $M$ to this element. The question is than how to produce such elements.

Let $i_n$ be the ordinary category, having objects $0,1,\dots,n$, and a unique morphism $i_n(i,j)$ for any $i\le j$. Denote by $I_n$ the $\k$-linear category $I_n=\k[i_n]$, by $\{e_1,\dots,e_n\}$ its generators, $e_j\in I_n(j-1,j)$. 
Define
\begin{equation}\label{eqx15}
I_{n_1,\dots,n_k}=I_{n_k}\sotimes(I_{n_{k-1}}\sotimes(\dots \sotimes(I_{n_2}\sotimes I_{n_1})\dots))
\end{equation}
\begin{equation}\label{eqx16}
\mathcal{O}(n_1,\dots,n_k)=I_{n_1,\dots,n_k}(\min,\max)
\end{equation}
where 
$$
\min=(0,0,\dots,0),\ \ \max=(n_k,n_{k-1},\dots,n_1)
$$
Note that $\mathcal{O}(n)=\k[0]$, and $\mathcal{O}(1,1)$ is exactly the complex \eqref{eq02}.

Having an element in $\mathcal{O}(n_1,\dots,n_k)$, we can plug $\Psi_{ij}$ for the place of the  corresponding generator $e_j^i$ of $I_{n_i}$, and get in an element in the l.h.s. of \eqref{eqx03bis}. Moreover, it is the most general ``canonical'' way of doing it. 

One can similarly plug elements of $\mathcal{O}(D)$ for the place of generators of some other $\mathcal{O}(D^\prime)$, which makes the 2-collection $\{\mathcal{O}(D)\}$ a 2-operad.

 A 2-operad in $\Vect_\dg(\k)$ \footnote{more precisely, a 1-terminal 2-operad [Ba4]} is given by a complex of $\k$-vector spaces $\mathcal{O}(D)$, for any 2-globular pasting diagram $D=(n_1,\dots,n_k)$. These complexes are subject to the 2-operadic associativity, see Section \ref{sectionbcomp}.

We can state our main result:
\begin{theorem}\label{theorem1}
The complexes of $\k$-vector spaces $\mathcal{O}(D)$, $D=(n_1,\dots,n_k)$ are the components of a 2-operad acting on $\Cat_\dg^\coh(\k)$. The 2-operad $\mathcal{O}$ is homotopically final, that is, there is a map of complexes $p\colon \mathcal{O}(D)\to\k[0]$ which is a quasi-isomorphism of complexes, for any $D$, which is compatible with the 2-operadic composition. 
\end{theorem}
We briefly recall the main definitions related to 2-operads in Appendix B. According to Batanin [Ba3] (see also Theorem \ref{theoremm}) an action of a homotopy trivial $n$-operad encodes a weak $n$-category. 

Note that the dg categories $I_n$ are cofibrant. As well, the twisted tensor product $C\sotimes D$ of two cofibrant dg categories is cofibrant, by [Sh3, Lemma 4.5]. Therefore, Theorem \ref{theoremx1} is applied to $I_{n_1,\dots,n_k}$ and it gives {\it a quasi-isomorphism}
$$
I_{n_1,\dots,n_k}(\min,\max)\xrightarrow{{quis}}(I_{n_k}\otimes\dots\otimes I_{n_1})(\min,\max)=\k[0]
$$
(The statement that this map is a quasi-isomorphism seems to be quite non-trivial, we do not know any way to prove it directly).
In fact, this application was our main motivation for developing of a more general theory in [Sh3].

\subsection{\sc }\label{sectionintroxskew}
Let us outline the constructions and results which lead to a proof of Theorem \ref{theorem1}.

First of all, the ``one-sided'' associativity map \eqref{xassoc}, together with ``one-sided'' unit maps, make $\Cat_\dg^\coh(\k)$ a skew-monoidal category, [S], [BL], [LS]. This skew-monoidal category is closed, with the inner $\Hom(C,D)$ given by $\Coh_\dg(C,D)$.

There is an analogue of the Mac Lane coherence theorem for skew monoidal categories, proven in loc.cit. The situation for the skew case is more complicated than its classical counterpart. Luckily, our example belongs to a special class of a skew monoidal categories, called {\it perfect} skew monoidal categories (see Definition \ref{defskewperf}). For this case, the coherence theorem is essentially simplified, and is exactly as simple as its classical pattern. 

One has:
\begin{theorem}\label{theorem2}
The triple $(\Cat_\dg(\k),\sotimes,\alpha)$ (augmented by the unit $\underline{\k}$ and the unit maps) forms a perfect skew monoidal category.
\end{theorem}

The skew monoidal structure on $(\Cat_\dg(\k),\sotimes)$ is essentially used for the 2-operad structure on the complexes $\mathcal{O}(D)$, defined in \eqref{eqx15}, \eqref{eqx16}. More precisely, the corresponding coherence theorem is employed to establish the 2-operadic associativity and unit identities. 

After that, we establish a compatibility between the skew monoidal structure given by $\alpha$ with the twisted composition dg functor 
$$
M\colon \Coh_\dg(D,E)\sotimes\Coh_\dg(C,D)\to\Coh_\dg(C,E)
$$
where
$C,D,E\in\Cat_\dg(\k)$. Consider the diagram
\begin{equation}\label{eqx17}
\xymatrix{
(\Coh_\dg(C,D)\sotimes \Coh_\dg(B,C))\sotimes\Coh_\dg(A,B)\ar[rr]^{\alpha}\ar[d]_{M\sotimes\id}&&\Coh_\dg(C,D)\sotimes(\Coh_\dg(B,C)\sotimes\Coh_\dg(A,B))\ar[d]^{\id\sotimes M}\\
\Coh_\dg(B,D)\sotimes \Coh_\dg(A,B)\ar[d]_{M}&&\Coh_\dg(C,D)\sotimes\Coh_\dg(A,C)\ar[d]^{M}\\
\Coh_\dg(A,D)\ar[rr]^{=}&&\Coh_\dg(A,D)
}
\end{equation}
\begin{theorem}\label{theorem3}
There exists a twisted composition \eqref{eqx03} which is associative in the sense that diagram \eqref{eqx17} commutes.
\end{theorem}

This theorem is translated to an action of the 2-operad $\mathcal{O}$ on $\Cat_\dg^\coh(\k)$. 

\subsection{\sc }\label{closedskewmonoidalintro}
Existence of the twisted composition $M$ which fulfils \eqref{eqx17} follows from a general categorical consideration, as we now explain.

Assume $\mathscr{C}$ is a skew monoidal category, and there is an {\it inner Hom} $[-,-]\colon \mathscr{C}^\op\times\mathscr{C}\to\mathscr{C}$, such there there is an adjunction
\begin{equation}\label{adjskew}
\mathscr{C}(X\otimes Y,Z)\simeq \mathscr{C}(X,[Y,Z])
\end{equation}
Then one can translate the commutative diagrams in the definition  of a skew monoidal category to commutative diagrams build up entirely in terms of the inner product $[-,-]$ and the corresponding maps. The resulting concept is called a {\it skew closed category}, and was introduced in [S]. Moreover, in presence of adjunction \eqref{adjskew} the two structures are equivalent, as it is shown in [loc.cit., Sect. 8].

An advantage of this approach is that the compatibility such as \eqref{eqx17} is a consequence of the above-mentioned general set-up. 

In this way, the twisted composition is not a new construction, but a consequence of the adjunction \eqref{eqx01} on $\Cat_\dg(\k)$ and of the skew-monoidal structure $\sotimes$ on $\Cat_\dg(\k)$, and \eqref{eqx17} and Theorem \ref{theorem3} hold automatically. More precisely, the adjunction \eqref{adjskew} gives morphism
$$
[X,Y]\otimes X\to Y
$$
adjoint to the identity morphism, and
$$
([Y,Z]\otimes [X,Y])\otimes X\to [Y,Z]\otimes([X,Y]\otimes X)\to [Y,Z]\otimes Y\to Z
$$
gives, by the adjunction, the twisted composition. The associativity \eqref{eqx17} and the unit axiom are consequences of the corresponding properties for $\sotimes$, likewise for the classical case of closed monoidal category [Ke, Sect. 1.6]. 

\comment

Note that our methods allow us to prove the following more general result. Assume we are given a 2-quiver $\mathscr{X}$ whose underlying 1-quiver $\mathscr{X}_1$ is a strict ordinary category, the 2-morphisms (whose compositions are not defined yet) are elements of a symmetric monoidal category $V$. Assume that for any objects $C,D\in \mathscr{X}$, $\mathscr{X}(C,D)\in \Cat(V)$ is a $V$-category (that is, the vertical composition of 2-morphisms is strict and is already defined). Denote by $I_n$ the $V$-enriched length $n$ interval category. 
Then any closed skew monoidal structure on $\mathscr{X}_1$ whose internal hom $[C,D]=\mathscr{X}(C,D)$ gives rise to a 2-operad in $V$ acting on $\mathscr{X}$ (this 2-operad needn't be homotopically final, in any sense). 

\endcomment

\subsection{\sc }
When the paper had been completed the author learned about the papers [BW], [BCW] devoted to closed questions for the Gray tensor product of bicategories. The papers loc.cit. shed some light on our constructions and indicate how they can be performed in general. Below we provide few remarks on these methods in their application to our question. 

First of all, we can consider consider $\Cat_\dg^\coh(\k)$ as a 1-category enriched in $(\Cat_\dg(\k),\sotimes)$. So the question this paper is devoted to is to translate a strict 1-category enriched in $(\Cat_\dg(\k),\sotimes)$ to a weak 1-category enriched in $(\Cat_\dg(\k),\otimes)$ where $\otimes\colon \Cat_\dg(\k)\times\Cat_\dg(\k)\to \Cat_\dg(\k)$ is the conventional tensor product of dg categories. Here ``weak 1-category'' is understood as an algebra over a homotopy final 1-operad in $C^\udot(\k)$.

Secondly, the same methods are applied to {\it any} strict 1-category enriched in $(\Cat_\dg(\k),\sotimes)$, not only to $\Cat_\dg^\coh(\k)$. Such categories are analogues of ``Gray categories'' in loc.cit. In particular, our results provide a proof of the latter statement, without any changes.

Finally, the methods of loc.cit. allow us to rephrase our constructions to make them valid for any closed skew monoidal category $(\mathscr{V},\sotimes)$ and any strict 1-category enriched in $\mathscr{V},\sotimes)$ (provided $(\mathscr{V},\sotimes)$ can be ``lifted'' to a multitensor on dg 1-graphs, see below. In this way, we avoid use of rather ``non-canonical'' constructions such as  ``substitutions'' in the definition of the 2-operadic composition for the the twisted tensor product 2-operad, and in its action on $\Cat_\dg^\coh(\k)$. 

The key observation, from which the operad structure on the 2-collection $\mathcal{O}$ and its action on $\Cat_\dg^\coh(\k)$ follow formally, is possibility to redefine the twisted tensor product at the globular level. That is, there is a multitensor $E_\sotimes$ (a lax-monoidal structure [BW])  on the category of dg 1-graphs so that any strict 1-category enriched in $(\Cat_\dg(\k),\sotimes)$ is the same that a 1-category enriched in $E_\sotimes$. 
Next, one constructs a monad called $\Gamma E_\sotimes$ in loc.cit., acting on dg 2-graphs from the multitensor $E_\sotimes$, and the operad $\mathcal{O}$ is encoded in this monad. In loc.cit. the authors extensively use the cartesianity of the monad $\Gamma E_\sotimes$ to ``decode'' the operad from it. 
Note that these ideas can't be applied to dg case directly, as we deal with non-cartesian monads and non-cartesian multitensors, though a suitable refinement of them can certainly be applied, and it improves and generalises our results. More precisely, the cartesianity is replaced by a grading by the set-enriched 2-level graphs which agree with the monad compositions. 
We hope to discuss these ideas in detail elsewhere.

\comment

\subsection{\sc }\label{introbraces}
Here we provide another, rather informal, argument for fulfilment of Theorem \ref{theorem3}.

We refer the reader to Section \ref{section11} for definition of the dg category $C\sotimes D$, as a dg category generated by $\{\id_c\otimes D\}_{c\in C}$, $\{C\otimes\id_d\}_{d\in D}$, and $\varepsilon(f;\ g_1,\dots,g_n)$. 
The differential $d(\varepsilon(f;\ g_1,\dots,g_n))$ is given by \eqref{eqd1}, and the relation for $\varepsilon(f_2f_1;\ g_1,\dots,g_n)$ is given in  \eqref{eqsuper}.

What makes the existence of the twisted composition dg functor $M$ possible is a complete similarity between these formulas for $\sotimes$ and the well-known identities for the brace operations $f\{g_1,\dots,g_n\}$ on the cochain complex $C^\udot(A,A)[1]$, see [GJ], [NT], [Ts]. Recall them for reader's reference:

\begin{equation}\label{brx1}
\begin{aligned}
\ &[d,f\{g_1,\dots,g_n\}]=(-1)^{|g_1||f|+|f|+1}g_1\cdot f\{g_2,\dots,g_n\}+\\
&\sum_{i=1}^{n-1}(-1)^{|f|+1+\sum_{j=1}^{i}(|g_i|+1)}f\{g_1,\dots, g_ig_{i+1},\dots,g_n\}+(-1)^{|f|+\sum_{i=1}^{n-1}(|g_i|+1)}f\{g_1,\dots,g_{n-1}\}\cdot g_n
\end{aligned}
\end{equation}
and
\begin{equation}\label{brx2}
(f_1\cdot f_2)\{g_1,\dots,g_n\}=\sum_{k=0}^n(-1)^{(|f_2|-1)(\sum_{j=1}^k|g_k|-k)}f_1\{g_1,\dots,g_k\}\cdot f_2\{g_{k+1},\dots,g_n\}
\end{equation}
(where $|f|=|f|_0+n$ is the degree of $f\colon A^{\otimes n}\to A$ in $C^\udot(A,A)$; $|f|_0$ denotes the degree of $f\colon A^{\otimes n}\to A$ as a linear map).

To model more general categorical complexes $\Coh_\dg(C,D)(F,G)$, one considers $C^\udot(A,A)$, not $C^\udot(A,A)[1]$, as a pattern. 
The signs in the corresponding formulas \eqref{eqd1} and \eqref{eqsuper} for $C\sotimes D$ are slightly different, because of the degree shift by [1], see Appendix \ref{app13} and Remark \ref{rembraceshift}.

Recall the following nice result which is due to B.Tsygan. Consider the map
\begin{equation}\label{maptsx}
C^\udot(A,A)[1]\to C^\udot(C^\udot(A,A))[1]
\end{equation}
defined as 
$$
D\mapsto\sum_{k\ge 0}D^{(k)}
$$
where
\begin{equation}
D^{(k)}(D_1,\dots,D_k)=D\{D_1,\dots,D_k\}
\end{equation}
Then it is a map of $\Br$-algebras, see [Ts, Prop.3 and Prop.4].

Note that $M\colon \Coh_\dg(D,E)\sotimes\Coh_\dg(C,D)\to\Coh_\dg(C,E)$ defines, by adjunction \eqref{eqx01}, a dg functor
$$
M_1\colon \Coh_\dg(D,E)\to\Coh_\dg(\Coh_\dg(C,D),\Coh_\dg(C,E))
$$
In this way, Theorem \ref{theorem3} is a generalisation of  the fact that the map \eqref{maptsx} is a map of $\Br$-algebras.

\endcomment

\subsection{\sc Organisation of the paper}
Below we outline the contents of the individual Sections. 

In Section \ref{section1}, we recall the definition and the basic properties of coherent natural transformations $\Coh(F,G)$ for $F,G\colon C\to D$ dg functors, define the dg category $\Coh_\dg(C,D)$, and recall the construction of the twisted tensor product $C\sotimes D$ from [Sh3], as well as its basic properties. We also define complexes $\mathcal{O}(n_1,\dots,n_k)$ and prove their contractibility. Later in Section \ref{section2opn} we show that $\mathcal{O}(n_1,\dots,n_k)$ are the components of a 2-operad $\mathcal{O}$.

In Section \ref{sectionassocmain}, we construct a one-side associativity map for $(\Cat_\dg(\k),\sotimes)$. We show that the structure we get fulfils the axioms of a skew monoidal category, [S], [LS], [BL]. Moreover, this skew monoidal category is {\it perfect}, see Definition \ref{defskewperf}. We prove in Proposition \ref{propskew} that the coherence theorem for perfect skew-monoidal categories is exactly analogous to the classical MacLane coherence for monoidal categories, and is therefore much simpler than the general case of loc.cit. 
Moreover, we define skew closed categories [S] and discuss the general categorical set-up of Section \ref{closedskewmonoidalintro}. It gives rise to an explicit expression for the twisted composition $M$ and to a proof of Theorem \ref{theorem3}.

In Section \ref{section2opn}, we apply the toolkit developed in Section \ref{sectionassocmain} to a proof of Theorem  \ref{theorem1}. We equip the collection of complexes $\mathcal{O}(n_1,\dots,n_k)$ with a structure of a 2-operad $\mathcal{O}$, show that this 2-operad is contractible, and that it acts on the pre-2-category $\Cat_\dg^\coh(\k)$.

Appendix A contains detailed proofs of two technical Theorems \ref{theorassoc} and \ref{theorskew}.

We briefly recall some basic definitions related to higher operads in Appendix \ref{sectionba}.

\subsection*{} 
\subsubsection*{\sc Acknowledgements}
The author is thankful to Michael Batanin and to Dima Tamarkin for illuminating discussions on their work, and to 
the anonymous referees for their careful reading of the manuscript and for many valuable suggestions. 

The work was financed 
by the Support grant
for International Mathematical Centres Creation and Development, by the Ministry of Higher Education and Science
of Russian Federation and PDMI RAS agreement № 075-15-2022-289 issued on April 6, 2022.

\section{\sc The twisted tensor product and a 2-operad $\mathcal{O}$}\label{section1}
Here we recall basic facts on the twisted tensor product $C\sotimes D$ of small dg categories, introduced in [Sh3]. After that, we define complexes $\mathcal{O}(n_1,\dots,n_k)$, by means of the twisted tensor product:
$$
\mathcal{O}(n_1,\dots,n_k)=I_{n_k}\sotimes (I_{n_{k-1}}\sotimes(\dots\sotimes (I_{n_2}\sotimes I_{n_1})\dots))(\min,\max)
$$
where $I_n$ is (the $\k$-linear span of) the length $n$ interval category, and $\min$ (resp., $\max$) is the minimal (resp., the maximal) object.
Later in Section \ref{section2opn} we prove that these complexes are the components of a 2-operad $\mathcal{O}$. Here we show that each of these complexes is canonically quasi-isomorphic to $\k[0]$.

\subsection{\sc Coherent natural transformations}\label{sectionapp1}
We recall the definition of a {\it coherent natural transformation} $F\Rightarrow G\colon C\to D$, where $C,D$ are small dg categories over $\k$, and $F,G$ are dg (resp., $A_\infty$) functors $C\to D$. 

Let $C,D\in \Cat_\dg(\k)$, and let $F,G\colon C\to D$ be dg functors. Associate with $(F,G)$ a cosimplicial set $\coh_\ldot(F,G)$, as follows.

Set $$\coh_0(F,G)=\prod_{X\in C}\Hom_D(F(X),G(X))$$ and 
\begin{equation}
\coh_n(F,G)=\prod_{X_0,X_1,\dots,X_n\in C}\underline{\Hom}_{\k}\Big(C(X_{n-1},X_n)\otimes \dots\otimes C(X_0,X_1),\ D\big(F(X_0),G(X_n)\big)\Big)
\end{equation}
where $\underline{\Hom}_{\k}$ is the internal Hom in the category of complexes over $\k$.

The coface maps 
$$
d_0,\dots,d_{n+1}\colon \coh_n(F,G)\to\coh_{n+1}(F,G)
$$
and the codegeneracy maps
$$
\eta_0,\dots,\eta_n\colon \coh_{n+1}(F,G)\to\coh_n(F,G)
$$
are defined in the standard way, see e.g. [T2, Sect. 3].

For example, recall the coface maps $d_0,d_1,d_2\colon \coh_1(F,G)\to \coh_2(F,G)$. For $$\Psi\in \prod_{X_0,X_1\in C}
\underline{\Hom}_\k\big(\Hom_C(X_0,X_1),\ \Hom_D(F(X_0),G(X_1)\big)
$$
one has:
\begin{equation}
\begin{aligned}
\ &d_0(\Psi)(X_0\xrightarrow{f_1}X_1\xrightarrow{f_2}X_2)=G(X_1\xrightarrow{f_2}X_2)\circ \Psi(X_0\xrightarrow{f_1}X_1)\\
&d_1(\Psi)(X_0\xrightarrow{f}X_1\xrightarrow{g}X_2)=\Psi(X_0\xrightarrow{f_2f_1}X_2)\\
&d_2(\Psi)(X_0\xrightarrow{f}X_1\xrightarrow{g}X_2)=\Psi(X_1\xrightarrow{f_2}X_2)\circ F(X_0\xrightarrow{f_1}X_1)
\end{aligned}
\end{equation}

One defines $\Coh(F,H)$ as the {\it totalization} of this cosimplicial dg vector space:
\begin{equation}\label{riehl}
\Coh(F,G)=\Tot(\coh_\ldot(F,G))=\int_{n\in\Delta}\underline{\Hom}_\k(C_\ldot(\k\Delta(-,n)),\coh_n(F,G))
\end{equation}
see e.g. [R, Ch.4].
Here
$C_\ldot(-)$ denotes the reduced Moore complex of a simplicial abelian group, and $\int_\mathscr{C}(-)$ denotes the {\it end} of a functor $\mathscr{C}\times\mathscr{C}^\op\to\mathscr{E}$.

This is a topologist's definition. It is the best one, in particular, it encodes all signs, and it makes sense for an arbitrary symmetric monoidal enrichment.

Below we unwind this definition, making it more explicit.

The definition given above is equivalent to defining $\Coh(F,G)$ the total complex of the bicomplex, one of whose differentials is the differential $\delta_\dg$ (coming from the differentials on the underlying dg vector spaces), and another differential is the cochain differential with the corrected signs:
\begin{equation}\label{cochaindiff}
\delta=\varepsilon_0d_0-\varepsilon_1d_1+\varepsilon_2d_2-\dots+(-1)^{n+1}\varepsilon_{n+1}d_{n+1}\colon \coh_{n}(F,G)\to\coh_{n+1}(F,G)
\end{equation}
Here $\varepsilon_i=\pm 1$ are sign corrections, depending on degrees of the cochains and of the arguments, see \eqref{diffsigns} below. 
We denote the cochain complex with the corrected signs by $\tilde{C}^\udot(-)$. Then
\begin{equation}
\Coh(F,G)=\Tot_{\Pi}\big(\tilde{C}^\udot(\coh_\ldot(F,G))\big)
\end{equation}
where $\Tot_{\Pi}(-)$ stands for the total {\it product} complex of a bicomplex, with the differential 
$$
\delta_\tot(\Psi)=(-1)^{|\Psi|_0}\delta+\delta_\dg
$$
We have the following formulas for the differentials (the reader is referred to [Sh6], Appendix D for a discussion of the signs):
\footnote{Here and below we denote by $|\Psi|_0$ the degree of $\Psi$ as a linear map; the total (Hochschild) degree is $|\Psi|=|\Psi|_0+n$, where $n$ is the number of arguments of $\Psi$. For the arguments $f_i$, we denote by $|f_i|$ its underlying grading in the complex of morphisms.}
\begin{equation}\label{firstdg}
\delta_\dg(\Psi)(f_n,\dots,f_1)=d_\dg(\Psi(f_n,\dots,f_1))-\sum_{i=1}^n(-1)^{|\Psi|+|f_n|+\dots+|f_{i+1}|+n-i}(\Psi(f_n,\dots,d_\dg(f_i),\dots,f_1)
\end{equation}
where in the r.h.s. $d_\dg$ stand for the differentials in the complexes of morphisms.
\comment
{\small \begin{remark}
{\rm Note that one should rescale by a sign our favourite cocycles, such as $m(f_2,f_1)=f_2\circ f_1$, to maintain $d_\dg(-)=0$. For example, one defines 
\begin{equation}\label{sc1}
M(f_2,f_1)=(-1)^{2|f_2|+|f_1|}m(f_2,f_1)
\end{equation}
in $C^\udot(A,A)$, and
\begin{equation}\label{sc1bis}
M[1](f_2,f_1)=(-1)^{|f_2|}m(f_2,f_1)
\end{equation}
in $C^\udot(A,A)[1]$.

In general, let $V$ be a complex, and let $\Psi\colon V^{\otimes n}\to V$ a cochain. Define 
\begin{equation}\label{sc2}
(d_\dg^\prime\Psi)(v_n,\dots,v_1)=d_\dg(\Psi(v_n,\dots,v_1))-\sum_{i=1}^n(-1)^{|\Psi|_0+|v_n|+\dots+|v_{i+1}|}\Psi(v_n,\dots,d_\dg v_i,\dots,v_1)
\end{equation}
Then the correspondence
\begin{equation}\label{sc3}
\tilde{\Psi}(f_n,\dots,f_1)=(-1)^{n|f_n|+\dots+|f_1|}\Psi(f_n,\dots,f_1)
\end{equation}
intertwines the differentials $d_\dg$ and $d_\dg^\prime$.
If we consider the shifted by 1 chain complex, the correct sign becomes 
\begin{equation}\label{sc4}
\tilde{\Psi}[1](f_n,\dots,f_1)=(-1)^{(n-1)|f_n|+(n-2)|f_{n-1}|+\dots+|f_2|}\Psi(f_n,\dots,f_1)
\end{equation}
Finally, the passage from the cochain complex $C^\udot(A,A)$ to the shifted cochain complex $C^\udot(A,A)[1]$ is given by 
\begin{equation}\label{sc5}
\Psi(f_n,\dots,f_1)\mapsto (-1)^{|f_1|+\dots+|f_n|}\Psi(f_n,\dots,f_1)
\end{equation}
and then in the shifted cochain complex one still has
\begin{equation}\label{sc6}
\delta_\dg(\Psi)(f_n,\dots,f_1)=d_\dg(\Psi(f_n,\dots,f_1))-\sum_{i=1}^n(-1)^{|\Psi|^\prime+|f_n|+\dots+|f_{i+1}|+n-i}(\Psi(f_n,\dots,d_\dg(f_i),\dots,f_1)
\end{equation}
where $|\Psi|^\prime=|\Psi|-1=|\Psi|_0+n-1$.
}
\end{remark}
}

\endcomment
The differential \eqref{cochaindiff}, for a chain of morphisms
$$
X_0\xrightarrow{f_1}X_1\xrightarrow{f_2}X_2\xrightarrow{f_3}\dots \xrightarrow{f_{n-1}}X_{n-1}\xrightarrow{f_n}X_n
$$
reads:
\begin{equation}\label{diffsigns}
\begin{aligned}
\ &(\delta\Psi)(f_n,\dots,f_1)=\\
&(-1)^{|\Psi||f_n|+|\Psi|}G(f_n)\circ \Psi(f_{n-1},\dots,f_{1})+(-1)^{|\Psi|+1+\sum_{i=2}^{n}(|f_i|+1)}\Psi(f_n,\dots,f_2)\circ F(f_1)+\\
&\sum_{i=1}^{n-1}(-1)^{|\Psi|+\sum_{j=i+1}^n(|f_j|+1)}\Psi(f_n,\dots,f_{i+1}f_i,\dots,f_1)
\end{aligned}
\end{equation}
The reader is referred to [Sh6], Appendix D, for a detailed discussion of signs.

\comment
This formula holds for the {\it shifted} cochain complex $\Coh(F,G)[1]$. The passage to $\Coh(F,G)$ is given by \eqref{sc5}, it only affects the two extreme terms. One gets:
\begin{equation}\label{diffsignsbis}
\begin{aligned}
\ &(\delta\Psi)(f_n,\dots,f_1)=\\
&(-1)^{|\Psi||f_n|+|\Psi|+|f_n|+1}G(f_n)\circ \Psi(f_{n-1},\dots,f_{1})+(-1)^{|\Psi|+|f_1|+\sum_{i=2}^{n}(|f_i|+1)}\Psi(f_n,\dots,f_2)\circ F(f_1)+\\
&\sum_{i=1}^{n-1}(-1)^{|\Psi|+1+\sum_{j=i+1}^n(|f_j|+1)}\Psi(f_n,\dots,f_{i+1}f_i,\dots,f_1)
\end{aligned}
\end{equation}
\endcomment

\subsection{\sc }
For fixed small dg categories $C,D$, one endows $\Coh_\dg(C,D)$ with a dg category structure. That is, we define the {\it vertical} composition of coherent natural transformations.

The dg category $\Coh_\dg(C,D)$ has the dg functors $F\colon C\to D$ as its objects, and
$$
\Coh_\dg(C,D)(F,G):=\Coh(F,G)
$$
as its $\Hom$-complexes. The composition is defined with a sign correction, as follows.
\sevafigc{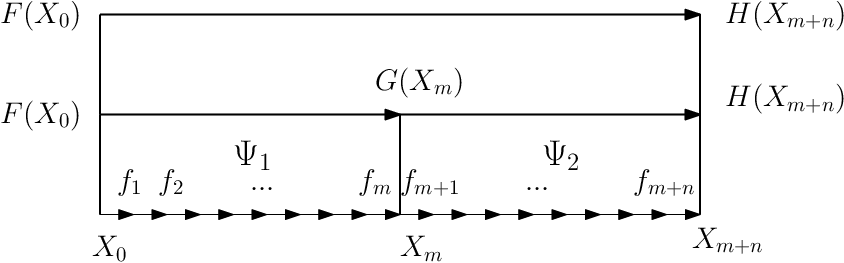}{100mm}{0}{The vertical composition\label{fig2}}
For chains $X_0\xrightarrow{f_1}X_1\xrightarrow{f_2}\dots \xrightarrow{f_m}X_m$ and $X_m\xrightarrow{f_{m+1}}X_{m+1}\xrightarrow{f_{m+2}}\dots\xrightarrow{f_{m+n}}X_{m+n}$ in $C$ one sets
\begin{equation}\label{compv}
(\Psi_2\circ_v\Psi_1)(f_{m+n},\dots,f_1,f_0)=(-1)^{(\sum_{i=m+1}^{m+n}|f_i|+n)|\Psi_1|}\Psi_2(f_{m+n},\dots,f_{m+1})\circ \Psi_1(f_m,\dots,f_1)
\end{equation}
One has
\begin{equation}\label{lr2}
\delta_\dg(\Psi_2\circ_v\Psi_1)=(\delta_\dg\Psi_2)\circ_v\Psi_1+(-1)^{|\Psi_2|}\Psi_2\circ_v\delta_\dg\Psi_1
\end{equation}
and
\begin{equation}\label{lr2bis}
{\delta}(\Psi_2\circ_v\Psi_1)=({\delta}\Psi_2)\circ_v\Psi_1+(-1)^{|\Psi_2|}\Psi_2\circ_v{\delta}\Psi_1
\end{equation}

\vspace{1mm}

\subsection{\sc }
The $A_\infty$ category $\Coh_{A_\infty}(C,D)$ has the $A_\infty$ functors $C\to D$ as its objects and
$$
\Coh_{A_\infty}(C,D)(F,G):=\Coh(F,G)
$$
as its $\Hom$-complexes, and the composition is defined similarly.

The construction of $\Coh_*(C,D)$ is functorial with respect to dg (corresp. $A_\infty$) functors $f\colon C_1\to C$ and $g\colon D\to D_1$, and gives rise to dg (corresp., $A_\infty$) functors
$$
f^*\colon\Coh_*(C,D)\to \Coh_*(C_1,D),\ \ g_*\colon \Coh_*(C,D)\to\Coh_*(C,D_1)
$$
where $*=\dg$ (corresp., $*=A_\infty$).

The following result has a fundamental value:
\begin{prop}\label{fquis}
Let dg functors $f\colon C_1\to C$ and $g\colon D\to D_1$ be quasi-equivalences. Then the dg functors $f^*\colon \Coh_{A_\infty}(C,D)\to\Coh_{A_\infty}(C_1,D)$ and $g_*\colon \Coh_{A_\infty}(C,D)\to\Coh_{A_\infty}(C,D_1)$ are quasi-equivalences. 
\end{prop}
\begin{proof}
It is proven in [LH, Ch.8] that $\Coh_{A_\infty}(C, D)$ is bi-functorial with respect to the $A_\infty$ functors.
It follows from [LH, Theorem 9.2.0.4] that a weak equivalence can be inverted as an $A_\infty$ functor.
The statement follows from these two results.
\end{proof}

\begin{coroll}\label{corfquis}
Let $C$ and $C_1$ be cofibrant, and $f,g$ as above. Then the dg functors \\$f^*\colon \Coh_{\dg}(C,D)\to\Coh_{\dg}(C_1,D)$ and $g_*\colon \Coh_{\dg}(C,D)\to\Coh_{\dg}(C,D_1)$ are quasi-equivalences of dg categories. 
\end{coroll}
\begin{proof}
It is proven e.g. in [Sh3], Lemma 4.2 and Corollary 4.3.
\end{proof}

\subsection{\sc The twisted tensor product}\label{section11}
\subsubsection{\sc The construction}\label{section111}
Let $C$ and $D$ be two small dg categories over $\k$. We define {\it the twisted dg tensor product} $C\sotimes D$, as follows.

The set of objects of $C\sotimes D$ is $\Ob(C)\times \Ob(D)$.
Consider the graded $\k$-linear category $F(C,D)$ with objects $\Ob(C)\times \Ob(D)$ freely generated by $\{f\otimes\id_d\}_{f\in\Mor(C), d\in D}$, $\{\id_c\otimes g\}_{c\in C, g\in \Mor(D)}$, and by the new morphisms $\varepsilon(f;g_1,\dots,g_n)$, specified below. (Below we write $\{C\otimes \id_d\}_{d\in D}$ assuming $\{f\otimes\id_d\}_{f\in\Mor(C), d\in D}$ etc).

Let
$c_0\xrightarrow{f}c_1$ be a morphism in $C$, and let $$d_0\xrightarrow{g_1}d_1\xrightarrow{g_2}\dots\xrightarrow{g_n}d_n$$

are chains of composable maps in $D$. For any such chains, with $n\ge 1$, one introduces a new morphism
$$
\varepsilon(f;g_1,\dots,g_n)\in \Hom((c_0,d_0),(c_1, d_n))
$$
of degree
\begin{equation}
\deg \varepsilon(f;g_1,\dots,g_n)=-n+\deg f+\sum\deg g_j
\end{equation}

The new morphisms $\varepsilon(f;g_1,\dots,g_n)$ are subject to the following identities:

\begin{itemize}
\item[$(R_1)$]
$(\id_c\otimes g_1)* (\id_c\otimes g_2)=\id_c\otimes (g_1g_2)$, $(f_1\otimes\id_d)*(f_2\otimes\id_d)=(f_1f_2)*\id_d$
\item[$(R_2)$] $\varepsilon(f;g_1,\dots,g_n)$ is linear in each argument,
\item[$(R_3)$] 
$\varepsilon(f; g_1,\dots,g_n)=0$ if $g_i=\id_y$ for some $y\in \Ob(D)$ and for some $1\le i\le n$,\\
$\varepsilon(\id_x; g_1,\dots,g_n)=0$ for $x\in\Ob(C)$ and $n\ge 1$,
\item[$(R_4)$]
for any $c_0\xrightarrow{f_1}c_1\xrightarrow{f_2}c_2$ and $d_0\xrightarrow{g_1}d_1\xrightarrow{g_2}\dots\xrightarrow{g_N}d_N$
one has:
\begin{equation}\label{eqsuper}
\varepsilon(f_2f_1;g_1,\dots,g_N)=\sum_{0\le m\le N}(-1)^{|f_1|(\sum_{j=m+1}^{N}|g_j|+N-m)}\varepsilon(f_2;g_{m+1},\dots,g_N)\star\varepsilon(f_1;g_1,\dots,g_m)
\end{equation}
\end{itemize}

To make $F(C,D)$ a dg category, one should define the differential $d\varepsilon(f;g_1,\dots,g_n)$.

For $n=1$ we set:
\begin{equation}\label{eqd0}
\begin{aligned}
\ & -d\varepsilon (f;g)+\varepsilon(df;g)+(-1)^{|f|}\varepsilon(f;dg)=\\
& (-1)^{|f||g|}(\id_{c_1}\otimes g)\star (f\otimes \id_{d_0})
-(f\otimes \id_{d_1})\star (\id_{c_0}\otimes g)
\end{aligned}
\end{equation}
For $n\ge 2$:
\begin{equation}\label{eqd1}
\begin{aligned}
\ &\varepsilon(df;g_1,\dots,g_n)=\\
&d\varepsilon({f;g_1,\dots,g_n})-\sum_{j=1}^n(-1)^{|f|+|g_n|+\dots+|g_{j+1}|+n-j} \varepsilon({f;g_1,\dots,dg_j,\dots,g_n})\big)+(-1)^{|f|+n-1}\Big[\\
&
(-1)^{|f||g_n|+|f|}(\id_{c_1}\otimes g_n)\star \varepsilon({f;g_1,\dots,g_{n-1}})
+(-1)^{|f|+\sum_{i=2}^n(|g_i|+1)+1}\varepsilon({f;g_2,\dots,g_n})\star (\id_{c_0}\otimes g_1)+\\
&\sum_{i=1}^{n-1} (-1)^{|f|+\sum_{j=i+1}^n(|g_j|+1)  } \varepsilon({f;g_1,\dots,g_{i+1}\circ g_i,\dots,g_n})\Big]
\end{aligned}
\end{equation}

We have:
\begin{lemma}\label{lemma1}
One has $d^2=0$. The differential agrees with relations ($R_1$)-($R_4$) above. 
\end{lemma}

\qed

It is clear that the twisted tensor product $C\sotimes D$ is functorial in each argument, for dg functors $C\to C^\prime$ and $D\to D^\prime$.

Note that the twisted product $C\sotimes D$ is not symmetric in $C$ and $D$.

It is {\it not} true in general that the dg category $C\sotimes D$ is quasi-equivalent to $C\otimes D$, or that these two dg categories are isomorphic as objects of $\Hot(\Cat_\dg(\k))$. See Theorem \ref{theorht} for a result on the homotopy type of $C\sotimes D$. 

\subsubsection{\sc The adjunction}
Our interest in the twisted tensor product $C\sotimes D$ is motivated by the following fact:
\begin{prop}\label{prop1}
Let $C,D,E$ be three small dg categories over $\k$. 
Then there is a 3-natural isomorphism of sets:
\begin{equation}\label{adj1}
\Phi\colon \Fun_{\dg}(C\sotimes D,E)\simeq \Fun_{\dg}(C,\Coh_\dg(D,E))
\end{equation}
\end{prop}
In fact, the definition of $C\sotimes D$ has been designed especially to fulfil this adjunction. 

The map $\Phi$ sends $\id_c\otimes D$ to a dg functor $F_c\colon D\to E$.
Then $\Phi(\id_c\otimes g)=F_c(g)$, $\Phi(f\otimes \id_d)\in F_{c_0}(d)\to F_{c_1}(d)$, for $f\colon c_0\to c_1$. 
Thus, we assign to $f\colon c_0\to c_1$ a coherent natural transformation $\Phi(f)=\Phi(\varepsilon(f,---))\colon F_{c_0}\Rightarrow F_{c_1}$, and $\Phi(f\otimes \id_d)$ is its 0-th component. Then $\Phi(\varepsilon(f;g))=\Phi(f)(g)$ is its first component.

The image $\Phi(d(\varepsilon(f;---))=d_E(\Phi(\varepsilon(f,---))$,
the image $\Phi(\varepsilon(df;---))=\delta_\tot\Phi(\varepsilon(f,---))$ is the total differential in $\Coh(D,E)(F_{c_0},F_{c_1})$,
and the summands $\Phi(\varepsilon(f;-d(-)-))$ are mapped to $\Phi(f)(-,d-,-)$.

Finally, \eqref{eqsuper} implies that the assignment $f\mapsto \Phi(\varepsilon(f,---))$ sends the composition in $C$ to the vertical composition in $\Coh_\dg(D,E)$.

See [Sh3, Th. 2.2] for detail.

\qed

\begin{coroll}\label{corproj}
There is a dg functor $p_{C,D}\colon C\sotimes D\to C\otimes D$, equal to the identity on objects, and sending all $\varepsilon(f;\ g_1,\dots,g_s)$ with $s\ge 1$ to 0.
\end{coroll}
\begin{proof}
It can be either seen directly, or can be deduced from Proposition \ref{prop1} and the natural dg embedding $\FFun_\dg(D,E)\to \Coh_\dg(D,E)$, along with the classic adjunction
\begin{equation}\label{adj1bis}
\Fun_\dg(C\otimes D,E)=\Fun_\dg(C,\FFun_\dg(D,E))
\end{equation}
Here $\FFun_\dg(C,D)$ is the dg category whose objects are dg functors $C\to D$, and $\FFun_\dg(C,D)(F,G)$ is the complex of naive natural transformations.
\end{proof}

\subsubsection{\sc The homotopy type of $C\sotimes D$}
For general $C,D$, we do not know the homotopy type of the dg category $C\sotimes D$. However, one has:
\begin{theorem}\label{theorht}
Let $C,D$ be small dg categories over $\k$. Assume both $C,D$ are cofibrant for the Tabuada closed model structure.
Then $C\sotimes D$ is also cofibrant and is isomorphic to $C\otimes D$ as an object of $\Hot(\Cat_\dg(\k))$. Moreover, the map of Corollary \ref{corproj} is a quasi-equivalence.
\end{theorem}
A proof was given in [Sh3, Th. 2.4].

\qed

\subsection{\sc The 2-operad $\mathcal{O}$}\label{sectionfg}
Here we define dg vector spaces $\mathcal{O}(n_1,\dots,n_k)$, $k\ge 1$, $n_1,\dots,n_k\ge 1$. Later in Section \ref{section2opn} we prove that these dg vector spaces are  the components of a 2-operad $\mathcal{O}$ (the reader is referred to Appendix \ref{sectionba} for definition of Batanin 2-operads).

Denote by $I_n$ the $\k$-linear span of the simplex category (defined over $\Sets$), which has objects $0,1,\dots, n$, and a unique morphism in $I_n(i,j)$ for any $i\le j$.

Denote 
\begin{equation}
I_{n_1,\dots,n_k}=I_{n_k}\sotimes(I_{n_{k-1}}\sotimes(\dots \sotimes(I_{n_2}\sotimes I_{n_1})\dots))
\end{equation}
Let 
$$
\min=(0,0,\dots,0)\text{  and  }\max=(n_k,n_{k-1},\dots,n_1)
$$
be the two ``extreme'' objects of $I_{n_1,\dots,n_k}$.

Define
\begin{equation}
\mathcal{O}(n_1,\dots,n_k)=I_{n_1,\dots,n_k}(\min,\max)
\end{equation}

It is a $\mathbb{Z}_{\le 0}$-graded complex of vector spaces over $\k$. 

\begin{prop}\label{propcontr}
Let $k,n_1,\dots,n_k\ge 0$. There is a map of complexes of vector spaces
$$
p_{n_1,\dots,n_k}\colon \mathcal{O}(n_1,\dots,n_k)\to\k[0]
$$
which is a quasi-isomorphism.
\end{prop}

\begin{proof}
The categories $I_n$ are cofibrant, and $C\sotimes D$ is cofibrant if $C,D$ are, by [Sh3, Lemma 4.5]. 
The statement is proven by induction, using this fact and Theorem \ref{theorht}, as follows.
At the first step, we see that the natural projection
$$
I_{n_k}\sotimes(I_{n_{k-1}}\sotimes(\dots \sotimes(I_{n_2}\sotimes I_{n_1})\dots))\to I_{n_k}\otimes\Big(I_{n_{k-1}}\sotimes(\dots \sotimes(I_{n_2}\sotimes I_{n_1})\dots)\Big)
$$
is a quasi-equivalence. Then we apply the same argument to $I_{n_{k-1}}\sotimes(\dots \sotimes(I_{n_2}\sotimes I_{n_1})\dots)$
and see that the natural projection 
$$
I_{n_{k-1}}\sotimes(\dots \sotimes(I_{n_2}\sotimes I_{n_1})\dots)\to I_{n_{k-1}}\otimes\Big(\dots \sotimes(I_{n_2}\sotimes I_{n_1})\dots\Big)
$$
We get the composition
\begin{equation}
\begin{aligned}
\ &I_{n_k}\sotimes(I_{n_{k-1}}\sotimes(\dots \sotimes(I_{n_2}\sotimes I_{n_1})\dots))\xrightarrow{p} I_{n_k}\otimes\Big(I_{n_{k-1}}\sotimes(\dots \sotimes(I_{n_2}\sotimes I_{n_1})\dots)\Big)\xrightarrow{\id\otimes p} \\
&I_{n_k}\otimes I_{n_{k-1}}\otimes\Big(I_{n_{k-1}}\sotimes\dots \sotimes(I_{n_2}\sotimes I_{n_1})\dots)\Big)
\end{aligned}
\end{equation}
(to show that the second map is a quasi-equivalence make use that $A\otimes B\xrightarrow{\id_A\otimes f}A\otimes B^\prime$ is a quasi-equivalence as soon as $B\xrightarrow{f}B^\prime$ is), and so on.
Finally, we see that
\begin{equation}
P_{n_1,\dots,n_k}\colon I_{n_k}\sotimes(I_{n_{k-1}}\sotimes(\dots \sotimes(I_{n_2}\sotimes I_{n_1})\dots))\to I_{n_k}\otimes I_{n_{k-1}}\otimes\dots\otimes  I_{n_2}\otimes I_{n_1}
\end{equation}
is a quasi-equivalence.\\
Its restriction to $\Hom(\min,\max)$ gives the quasi-isomorphism $p_{n_1,\dots,n_k}$.\footnote{In fact, the fact that $p\colon \mathcal{O}(D)\to\k$ is a quasi-isomorphism for each $D$ is simpler than the general result of Theorem \ref{theorht}, and can be proven by elementary methods, as it is shown in [Sh4].}

\end{proof}

\section{\sc One-side associativity and skew monoidal categories}\label{sectionassocmain}
Here we construct, for small dg categories $C,D,E$, an associativity dg functor
$$
\alpha_{C,D,E}\colon (C\sotimes D)\sotimes E\to C\sotimes (D\sotimes E)
$$
natural in each argument. The dg functor $\alpha_{C,D,E}$ {\it is not an isomorphism} in general. It does not give rise to a monoidal structure, therefore. The structure we get is described as a {\it skew monoidal category}, see [LS], [BL], [S]. We essentially use a coherence theorem proven in loc.cit., which substitutes the Mac Lane coherence theorem for the case of skew monoidal categories. In our example, the skew monoidal structure on $(\Cat_\dg(\k),\sotimes)$ is {\it perfect}, see Definition \ref{defskewperf}. For perfect skew monoidal categories, the coherence theorem is as simple as its classical pattern, see Proposition \ref{propskew}.

\subsection{\sc The associativity map}\label{sectionassoc}
\begin{theorem}\label{theorassoc}
For any three dg categories $C,D,E$, there is a unique dg functor
$$
\alpha_{C,D,E}\colon (C\sotimes D)\sotimes E\to C\sotimes (D\sotimes E)
$$
natural in each argument, 
which is the identity map on objects, and which is defined on morphisms as follows:
\begin{itemize}
\item[(i)]
for $f\in C$, $g\in D$, $h\in E$, $X,Y,Z$ objects of $C,D,E$ correspondingly, one has:
\begin{equation}\label{eqassoc1}
\begin{aligned}
\ &\alpha_{C,D,E}((f\star \id_Y)\star \id_Z)=f\star (\id_Y\star\id_Z)\\
&\alpha_{C,D,E}((\id_X\star g)\star\id_Z)=\id_X\star (g\star \id_Z)\\
&\alpha_{C,D,E}((\id_X\star\id_Y)\star h)=\id_X\star(\id_Y\star h)
\end{aligned}
\end{equation}
\item[(ii)]
for $f\in C$, $g_1,\dots,g_k\in D$, $k\ge 1$, and $Z$ an object in $E$, one has:
\begin{equation}\label{eqassoc1bis}
\alpha_{C,D,E}(\varepsilon(f;g_1,\dots,g_k)\star \id_Z)=\varepsilon(f;g_1\star\id_{Z},\dots, g_k\star\id_{Z})
\end{equation}

\item[(iii)]
for $f\in C$, $g\in D$, $h_1,\dots,h_n\in E$, $X$ an object of $C$, $Y$ an object of $D$, one has:
\begin{equation}\label{eqassoc1bisbis}
\begin{aligned}
\ &\alpha_{C,D,E}(\varepsilon(f\star \id_Y;h_1,\dots,h_n))=\varepsilon(f; \id_Y\star h_1,\dots,\id_Y\star h_n)\\
&\alpha_{C,D,E}(\varepsilon(\id_X\star g; h_1,\dots, h_n))=\id_X\star \varepsilon(g;h_1,\dots,h_n)
\end{aligned}
\end{equation}
\item[(iv)]
for $f\in C$, $g_1,\dots,g_k\in D$, $h_1,\dots,h_N\in E$, one has:
\begin{equation}\label{eqassoc4bis}
\begin{aligned}
\ &\alpha_{C,D,E}\big(\varepsilon(\varepsilon(f; g_1,\dots,g_k);h_1,\dots,h_N)\big)=\sum_{1\le i_1\le j_1\le i_2\le j_2\le \dots\le j_k\le N}(-1)^{\sum_{\ell=1}^k|g|_\ell\cdot\sum_{s\le i_\ell}(|h_s|-1)}\\
&\varepsilon\big(f;\   h_1,\dots, h_{i_1}, \varepsilon(g_1;h_{i_1+1},\dots,h_{j_1}), h_{j_1+1},\dots, h_{i_2},\varepsilon(g_2;h_{i_2+1},\dots,h_{j_2}), h_{j_2+1},\dots \big)
\end{aligned}
\end{equation}
where the sum is taken over all ordered sets $\{1\le i_1\le j_1\le i_2\le j_2\le\dots\le j_{k}\le N\}$; for the case $j_\ell=i_\ell$, the corresponding term $\varepsilon(g_\ell;\ h_{i_{\ell+1}},\dots,h_{j_\ell})$ is replaced by $g_\ell\otimes \id_{-}$.
\end{itemize}

\end{theorem}

We provide a proof of this Theorem in Appendix \ref{approofs}. 

\begin{remark}\label{remfu}{\rm 
In the r.h.s. of formula \eqref{eqassoc4bis}, those $A_i$ in $\varepsilon(f; A_1,\dots,A_M)$  which are equal to $h_j$ are understood as $\id_Y\star h_j$, where $\id\star_Y$ is the identity morphism of the corresponding object in $Y$.
}
\end{remark}

\comment
\begin{remark}{\rm
Formula \eqref{eqassoc4bis} may remind the reader the well-known formula for the composition of two brace operations on the Hochschild complex $C^\udot(A,A)$ of an associative algebra $A$, see \eqref{eqapp5x}.
We make use of this similarity when discuss the twisted tensor product, see Section \ref{sectionm2} below.
}
\end{remark}
\endcomment

\subsection{\sc Skew monoidal categories}\label{sectionskmc}
As we have mentioned, the associativity dg functors $\alpha_{C,D,E}$ are not, in general, isomorphisms. Therefore, the most plausible structure one can get out of them is not the structure of a monoidal category (where the associativity maps are isomorphisms). What we'll arrive to is a structure of a {\it skew monoidal category}, see [LS], [BS], [S]. We provide here a short overview of this theory. There is a {\it coherence} theorem for skew monoidal categories, which is, however, a more complicated statement. In our example from Section \ref{sectionassoc}, the closed monoidal category is {\it perfect}, see Definition \ref{defskewperf} below. For the perfect closed monoidal categories, the coherence theorem becomes as simple as its classical counterpart, the MacLane coherence for monoidal categories (see Proposition \ref{propskew}). It is essentially used in the proof of the associativity of the operadic composition for the 2-operad $\mathscr{O}$, see Section \ref{sectionopcomp}. 
\comment
In this Subsection, we recall the definition of a skew monoidal category [LS], and show that the category of small dg categories over $\k$, with the twisted tensor product $\sotimes$, the associativity dg functors $\alpha_{C,D,E}$, and the unit object $\underline{\k}$ (the dg category with a single object whose space of endomorphisms is $\k$), gives rise to a structure of a skew monoidal category. 
\endcomment
\begin{defn}\label{defskew}{\rm
Let $\mathscr{C}$ be a category with an object $I$ called the {\it skew unit}, and with a functor $\otimes\colon \mathscr{C}\times\mathscr{C}\to\mathscr{C}$ (called {\it skew tensor product}), and natural families of {\it lax constraints} having the directions
\begin{equation}\label{eqalphan}
\alpha_{XYZ}\colon (X\otimes Y)\otimes Z\to X\otimes (Y\otimes Z)
\end{equation}
\begin{equation}\label{eqlambdan}
\lambda_X\colon I\otimes X\to X
\end{equation}
\begin{equation}\label{eqrhon}
\rho_X\colon X\to X\otimes I
\end{equation}
subject to the following conditions:
\begin{equation}\label{eqskewmonn1}
\xymatrix{
&(W\otimes X)\otimes (Y\otimes Z)\ar[rd]^{\alpha_{W,X,Y\otimes Z}}\\
((W\otimes X)\otimes Y)\otimes Z\ar[ru]^{\alpha_{W\otimes X,Y,Z}}\ar[d]_{\alpha_{W,X,Y\otimes \id_Z}}&&W\otimes(X\otimes(Y\otimes Z))\\
(W\otimes(X\otimes Y))\otimes Z\ar[rr]_{\alpha_{W,X\otimes Y,Z}}&&W\otimes((X\otimes Y)\otimes Z)\ar[u]_{\id_W\otimes\alpha_{X,Y,Z}}
}
\end{equation}
\begin{equation}\label{eqskewmonn2}
\xymatrix{
(I\otimes X)\otimes Y\ar[dr]_{\lambda_{X\otimes\id_Y}}\ar[rr]^{\alpha_{I,X,Y}}&&I\otimes(X\otimes Y)\ar[dl]^{\lambda_{X\otimes Y}}\\
&X\otimes Y
}
\end{equation}
\begin{equation}\label{eqskewmonn3}
\xymatrix{
(X\otimes I)\otimes Y\ar[r]^{\alpha_{X,I,Y}}&X\otimes (I\otimes Y)\ar[d]^{\id_X\otimes\lambda_Y}\\
X\otimes Y\ar[u]^{\rho_X\otimes \id_Y}\ar[r]_{\id_{X\otimes Y}}&X\otimes Y
}
\end{equation}
\begin{equation}\label{eqskewmonn4}
\xymatrix{
(X\otimes Y)\otimes I\ar[rr]^{\alpha_{X,Y,I}}&&X\otimes(Y\otimes I)\\
&X\otimes Y\ar[ul]^{\rho_{X\otimes Y}}\ar[ur]_{\id_X\otimes\rho_Y}
}
\end{equation}
\begin{equation}\label{eqskewmonn5}
\xymatrix{
I\ar[rr]^{\id_I}\ar[dr]_{\rho_I}&&I\\
&I\otimes I\ar[ru]_{\lambda_I}
}
\end{equation}
}
\end{defn}
Note that $\alpha_{X,Y,Z},\lambda_X,\rho_X$ are not assumed to be isomorphisms. 

It follows from the definition that the maps
\begin{equation}\label{skewidem}
\begin{aligned}
\ &\varepsilon_{X,Y}^\ell\colon (X\otimes I)\otimes Y\xrightarrow{\alpha}X\otimes(I\otimes Y)\xrightarrow{\id\otimes \lambda}X\otimes Y\xrightarrow{\rho\otimes \id}(X\otimes I)\otimes Y\\
&\varepsilon_{X,Y}^r\colon X\otimes (I\otimes Y)\xrightarrow{\id\otimes\lambda}X\otimes Y\xrightarrow{\rho\otimes\id}(X\otimes I)\otimes Y\xrightarrow{\alpha}X\otimes(I\otimes Y)\\
&\varepsilon_0\colon I\otimes I\xrightarrow{\lambda}I\xrightarrow{\rho}I\otimes I
\end{aligned}
\end{equation}
are {\it idempotents}, but are {\it not equal to identity}, in general. 

It makes the coherence theorem for skew monoidal categories a non-trivial issue. It is proven in [LS] and is refined in [BL].

We suggest the following definition:
\begin{defn}\label{defskewperf}{\rm
A skew monoidal category is called {\it perfect} if $\lambda_X$, $\rho_X$ are isomorphisms, for any object $X$. 
}
\end{defn}
One sees immediately that in a perfect skew monoidal categories the morphisms \eqref{skewidem} are equal to identity morphisms. Indeed, they are idempotents and isomorphisms. 
More generally, the following Proposition shows that for a {perfect} skew monoidal category the classical MacLane coherence holds, what makes the situation much easier than for general skew monoidal categories, see [LS], [BL].
\begin{prop}\label{propskew}
Let $\mathscr{C}$ be a perfect skew monoidal category. Then the classical MacLane coherence theorem holds in $\mathscr{C}$. More precisely, any two morphisms between the same pair of objects, formed by successive application of $\alpha,\lambda,\rho$, are equal.
\end{prop}

\begin{proof}
Using the morphisms $\lambda_X$ and $\rho_X^{-1}$ one gets a canonical map $p_Y$ from any monomial $Y$ in $\mathscr{C}$ in objects $X_1,\dots,X_n\ne I$, and in several copies of the unit $I$, to a monomial in $\bar{Y}$ in $X_1,\dots,X_n$ without $I$. 
Let $f\colon Y\to Y^\prime$ be an {\it elementary} morphism, that is, which is a single application of either $\alpha$, or $\rho_X,\lambda_X,\rho_X^{-1},\lambda_X^{-1}$. It induces an elementary morphism $\bar{f}\colon \bar{Y}\to \bar{Y}^\prime$, which is $\alpha$ or $\id$ when $f=\alpha$, and which is the identity morphism if $f$ is $\lambda_X,\rho_X$ or its inverse. 
It follows from \eqref{eqskewmonn2}-\eqref{eqskewmonn5} that for any elementary morphism $f$ the diagram 
\begin{equation}\label{eqskewproof}
\xymatrix{
Y\ar[r]^{f}\ar[d]_{p_Y}&Y^\prime\ar[d]^{p_{Y^\prime}}\\
\bar{Y}\ar[r]^{\bar{f}}&\bar{Y}^\prime
}
\end{equation}
is commutative (and the vertical maps are isomorphisms). \\
Now assume we are given two morphisms from $Y_1$ to $Y_N$ each of which is a composition of elementary morphisms. We get two maps from $\bar{Y}_1$ to $\bar{Y}_N$, each of which is a composition of several $\alpha$s. 
The MacLane coherence holds for a non-unital monoidal category $\mathscr{C}$ in which the associator is not necessarily isomorphisms, see [LS, Prop. 7.1] and [ML1]. It follows that the two morphisms $\bar{Y}_1\to\bar{Y}_N$ are equal. Using commutativity of  diagrams \eqref{eqskewproof}, we conclude that the two morphisms $Y_1\to Y_N$ are equal. 
\end{proof}

\subsection{\sc Skew closed categories}\label{sectionscc}
Assume $\mathscr{C}$ is a skew monoidal category, and there is an {\it inner Hom} $[-,-]\colon \mathscr{C}^\op\times\mathscr{C}\to\mathscr{C}$, such there there is an adjunction
\begin{equation}\label{adjskew}
\mathscr{C}(X\otimes Y,Z)\simeq \mathscr{C}(X,[Y,Z])
\end{equation}
Then one can translate the commutative diagrams in Definition \ref{defskew} to commutative diagrams formulated entirely in terms of the inner product $[-,-]$. The resulting concept, introduced in [S], is called a {\it skew closed category}. Moreover, in presence of adjunction \eqref{adjskew} the two structures are equivalent, see [S, Sect. 8].

\comment
\begin{defn}
{\rm
A {\it skew closed} category $\mathscr{C}$ is a category with the following data:
a functor $[-,-]\colon \mathscr{C}^\op\times\mathscr{C}\to\mathscr{C}$, a {\it unit} object $I\in\mathscr{C}$, a natural transformation
\begin{equation}
L=L_{BC}^A\colon [B,C]\to[[A,B],[A,C]]
\end{equation}
a natural transformation
\begin{equation}
i=i_A: [I,A]\to A
\end{equation}
and a natural transformation
\begin{equation}
j=j_A: I\to [A,A]
\end{equation}
which are subject to the 5 axioms the reader can find in [S], Section 2.}
\end{defn}

\begin{equation}
\xymatrix{
&[[A,C],[A,D]]\ar[rd]^{L}\\
[C,D]\ar[d]_{L}\ar[ur]^{L}&&[[[A,B],[A,C]],[[A,B],[A,D]]]\ar[d]^{[L,\id]}\\
[[B,C],[B,D]]\ar[rr]^{[\id,L]}&&[[B,C],[[A,B],[A,D]]]
}
\end{equation}
\begin{equation}
\xymatrix{
[[A,A][A,C]]\ar[rr]^{\hspace{3mm}[j,1]}&&[I,[A,C]]\ar[dl]^{i}\\
&[A,C]\ar[ul]^{L}
}
\end{equation}
\begin{equation}
\xymatrix{
[B,B]\ar[rr]^{L\hspace{5mm}}&&[[A,B],[A,B]]\\
&I\ar[ul]^{j}\ar[ur]_{j}
}
\end{equation}
\begin{equation}
\xymatrix{
[B,C]\ar[rd]_{[i,\id]}\ar[rr]^{L}&&[[I,B],[I,C]]\ar[ld]^{[\id,i]}\\
&[[I,B],C]
}
\end{equation}
\begin{equation}
\xymatrix{
I\ar[rr]^{j}\ar[dr]_{\id}&&[I,I]\ar[dl]^{i}\\
&I
}
\end{equation}

\endcomment

\comment
First of all, we have a morphism $p\colon [B\otimes C,D]\to [B,[C,D]]$ (not necessarily an isomorphism), for any three objects $B,C,D\in\mathscr{C}$. It follows from the commutativity of the diagram
\begin{equation}\label{eqsm1}
\xymatrix{
\mathscr{C}(A\otimes (B\otimes C),D)\ar[rr]^{\simeq}\ar[d]_{\alpha_{A\otimes B,C}^*}&&\mathscr{C}(A, [B\otimes C,D])\ar[d]\\
\mathscr{C}((A\otimes B)\otimes C,D)\ar[rr]^{\simeq}&&\mathscr{C}(A,[B,[C,D]])
}
\end{equation}
in which the horizontal maps are isomorphisms. Then the r.h.s. vertical arrow defines, via the Yoneda lemma, a morphism 
$p_{BCD}$.

Then define $L_{AC}^B$ as the composition
\begin{equation}\label{eqsm2}
[A,C]\xrightarrow{[e,\id]}[[B,A]\otimes B,C]\xrightarrow{p_{[B,A],B,C}}[[B,A],[B,C]]
\end{equation}
where the morphism $e\colon [B,A]\otimes B\to A$ is obtained from the identity morphism $\id_{[B,A]}$ via the adjunction \eqref{adjskew}. 

Furthermore, define morphism $j_A\colon I\to [A,A]$ as corresponded to $\lambda_A\colon I\otimes A\to A$, via the adjunction \eqref{adjskew}. 

Finally, define $i_B\colon [I,B]\to B$ from the diagram
\begin{equation}\label{eqsm3}
\xymatrix{
\mathscr{C}(A\otimes I,B)\ar[rr]^{\rho_A}\ar[dr]_{\simeq}&&\mathscr{C}(A,B)\\
&\mathscr{C}(A,[I,B])\ar[ur]_{\mathscr{C}(\id,i_B)}
}
\end{equation}

The following statement is a part of [S, Prop.18], and is essential for our paper:
\begin{prop}\label{propsm}
Assume $(\mathscr{C}, I,\alpha,r,\lambda)$ is a skew monoidal category, such that there exists an inner hom $[-,-]$ such that \eqref{adjskew} holds. Then $(\mathscr{C}, I, L,i,j)$ is a skew closed category. Vice versa, assume that $(\mathscr{C}, I, L,i,j)$ is a skew closed category, and \eqref{adjskew} holds. Then the same diagrams \eqref{eqsm1}-\eqref{eqsm3} define the associativity morphism $\alpha$ and the unit maps $r$ and $\lambda$, making $(\mathscr{C}, I,\alpha,r,\lambda)$ a skew monoidal category. 
Moreover the diagrams \eqref{eqcompskew1}-\eqref{eqcompskew3} below are commutative. 
\end{prop}
For the first two statements, see [S, Section 8] for a proof. For the third statement, it means that $\mathscr{C}$ is a $\mathscr{C}$-enriched category, see [S, Rem. 20]. The proof given in [Ke, Section 1.6] for the classical closed monoidal case, is easily adopted to the skew monoidal situation. 

\qed
\endcomment

Assume \eqref{adjskew} holds, define $M_{A,B,C}\colon [B,C]\otimes [A,B]\to [A,C]$ as follows. First of all, there is 
$$
e\colon [A,B]\otimes A\to B
$$
adjoint to $\id\colon [A,B]\to [A,B]$. One has the composition
$$
([B,C]\otimes [A,B])\otimes A\xrightarrow{\alpha}[B,C]\otimes ([A,B]\otimes A)\xrightarrow{\id\otimes e}[B,C]\otimes B\xrightarrow{e} C
$$
whose right adjoint gives $M_{A,B,C}$. 

\begin{prop}\label{propsm}
The composition $M$ defined above makes a skew closed monoidal category enriched over itself. More precisely, the following commutative diagrams hold:
\begin{equation}\label{eqcompskew1}
\xymatrix{
([C,D]\otimes [B,C])\otimes [A,B]\ar[rr]^{\alpha}\ar[d]_{M_{BCD}\otimes\id}&&[C,D]\otimes ([B,C]\otimes[A,B])\ar[d]^{\id\otimes M_{ABC}}\\
[B,D]\otimes [A,B]\ar[dr]_{M_{ABD}}&&[C,D]\otimes [A,C]\ar[dl]^{M_{ACD}}\\
&[A,D]
}
\end{equation}

\begin{equation}\label{eqcompskew2}
\xymatrix{
[A,B]\otimes I\ar[r]^{\id\otimes j\hspace{3mm}}&[A,B]\otimes [A,A]\ar[d]^{M_{AAB}}\\
[A,B]\ar[u]^{\rho}\ar[r]^{\id}&[A,B]
}
\end{equation}

\begin{equation}\label{eqcompskew3}
\xymatrix{
[B,B]\otimes [A,B]\ar[r]^{\hspace{7mm}M_{ABB}}&[A,B]\\
I\otimes [A,B]\ar[u]^{j\otimes \id}\ar[r]_{\lambda}&[A,B]\ar[u]_{\id}
}
\end{equation}
\end{prop}
\begin{proof}
See [S], Example 4 and Remark 20. It is proven similarly to the classical counterpart for closed monoidal categories [Ke], Sect. 1.6. 
\end{proof}

As it stated below in Theorem \ref{theorskew}, the category $(\Cat_\dg(\k),\alpha,\underline{\k})$ is skew-monoidal, and the adjunction \eqref{adjskew} holds, by \eqref{adj1}. Thus it gives rise to maps $M_{ABC}$ such that \eqref{eqcompskew1}-\eqref{eqcompskew3} hold. 
The map $M$ for  $(\Cat_\dg(\k),\alpha,\underline{\k})$ is called {\it twisted composition}.

\subsection{\sc The category $\Cat_\dg(\k)$ is perfect skew monoidal}
Denote by $\underline{\k}$ the dg category with a single object * such that $\underline{\k}(*,*)=\k$. Then there are the following unit maps:
\begin{equation}
\lambda_C\colon \underline{\k}\sotimes C\to C
\end{equation}
and
\begin{equation}
\rho_C\colon C\to C\sotimes \underline{\k}
\end{equation}
which are defined as the identity maps, due to axiom $(R_3)$ in Section \ref{section111}.

\begin{theorem}\label{theorskew}
The category $\mathscr{C}=\Cat_\dg(\k)$ of small dg categories, equipped with the twisted product $-\sotimes-$, the unit $\underline{\k}$, the associativity constrains $\alpha$, and with the natural isomorphisms $\lambda_C$ and $\rho_C$, $C\in \Cat_\dg(\k)$, is a perfect skew monoidal category, see Definition \ref{defskewperf}. 

\end{theorem}
\begin{proof}

We prove this Theorem in Appendix \ref{approofs}.

\end{proof}

\subsection{\sc Explicit dg functor $M$, associated with the skew-associativity $\alpha$}
Here we compute the twisted composition $M$ explicitly, for the case of $(\Cat_\dg(\k),\sotimes,\underline{\k})$. 
Our proofs are not based on this computation, but, if we want to ask ourselves {\it how} the dg 2-operad $\mathcal{O}$ acts on the 2-graph $\Cat_\dg^\coh(\k)$, this explicit form is necessary. We will be rather concise here. The reader who is interested in more details is referred to the preliminary preprint version of this paper [Sh6], Section 4.

\subsubsection{\sc The unit map $e\colon \Coh_\dg(C,D)\sotimes C\to D$}
The dg functor $e\colon \Coh_\dg(C,D)\sotimes C\to D$ corresponds to $\id_{\Coh_\dg(C,D)}$ via the adjunction \eqref{adj1}. 
It is given on objects in the standard way $F\sotimes c\mapsto F(c)$, where $F\in \Coh_\dg(C,D)$ a dg functor $F\colon C\to D$, $c\in \Ob(C)$.

We have to specify $e$ on generating morphisms of $\Coh_\dg(C,D)\sotimes C$, which are 
(i) $\Psi\sotimes \id_c$, (ii) $\id_F\sotimes f$, (iii) $\varepsilon(\Psi; f_1,\dots,f_n)$, where $\Psi\in\Coh_\dg(C,D)$, $c\in\Ob(C)$, $F\colon C\to D$ a dg functor, $f$ a morphism in $C$. 

One has:

\begin{equation}\label{eqebis}
e(\Psi\sotimes\id_c)=\Psi_0(c)
\end{equation}
where $\Psi_0$ is 0th component of $\Psi$,

\begin{equation}\label{eqebis2}
e(\id_F\sotimes f)=F(f)
\end{equation}

\begin{equation}\label{eqe}
e(\varepsilon(\Psi; f_1,\dots,f_n))=\Psi_n(f_n\otimes\dots\otimes f_1)
\end{equation}
where $\Psi_n$ is the $n$th component of $\Psi$.

\comment
\subsubsection{\sc The dg functor $p\colon \Coh_\dg(C\sotimes D,E)\to \Coh_\dg(C,\Coh_\dg(D,E))$}
The action of the dg functor $p\colon \Coh_\dg(C\sotimes D,E)\to \Coh_\dg(C,\Coh_\dg(D,E))$ on objects is determined easily from adjunction \eqref{adj1}. The new thing is its action on morphisms. 

Fix a morphism $\Psi\in\Coh_\dg(C\sotimes D,E)(-,=)$, it is determined by its values on composable chain of morphisms in $C\sotimes D$. Similarly, the morphism $p(\Psi)$ is determined by chains of composable morphisms
$\xrightarrow{g_1}\xrightarrow{g_2}\dots\xrightarrow{g_n}$ in $C$ and
$\xrightarrow{h_1}\xrightarrow{h_2}\dots\xrightarrow{h_m}$ in $D$. 

One has:

\begin{equation}\label{eqmorphismp}
\begin{aligned}
\ &p(\Psi)(g_1,\dots, g_n)(h_1,\dots,h_m)=\sum_{1\le i_1\le j_1\le i_2\le j_2\le \dots\le j_k\le N}(-1)^{\sum_{\ell=1}^k|g|_\ell\cdot\sum_{s\le i_\ell}(|h_s|-1)} \\
&\Psi\big(h_1,\dots,h_{i_1},\varepsilon(g_1; h_{i_1+1},\dots,h_{j_1}), h_{j_1+1},\dots,h_{i_2},\varepsilon(g_2; h_{i_2+1},\dots,h_{j_2}),\dots\big)
\end{aligned}
\end{equation}
where the sum is taken over all ordered sets $\{1\le i_1\le j_1\le i_2\le j_2\le\dots\le j_{k}\le N\}$; for the case $j_\ell=i_\ell$, the corresponding term $\varepsilon(g_\ell;\ h_{i_{\ell+1}},\dots,h_{j_\ell})$ is replaced by $g_\ell\otimes \id_{-}$. 

The reader noticed that the summation set, the signs, and the overall structure of the formula are the same as for \eqref{eqassoc4bis}.

Formula \eqref{eqmorphismp} easily follows from the definition of $p$, see \eqref{eqsm1}.

\subsubsection{\sc The dg functor $L\colon \Coh_\dg(D,E)\to\Coh_\dg(\Coh_\dg(C,D),\Coh_\dg(C,E))$}
Here determine the dg functor $L$ explicitly, by its definition in \eqref{eqsm2}.

 Fix $\Psi\in \Coh_\dg(D,E)$. We compute $e(\Psi)$ for $[e,\id]\colon \Coh_\dg(D,E)\to\Coh_\dg(\Coh_\dg(C,D)\sotimes C,E)$, the composition 
\eqref{eqsm2}, computed on a {\it particular} chain of morphisms in $\Coh_\dg(C,D)\sotimes C$, each of which is of the form $\varepsilon(\Phi; f_1,\dots,f_n)$ or $\id_F\sotimes f$, where $\Phi$ is a morphism in $\Coh_\dg(C,D)$, and $f,f_1,\dots,f_n$ morphisms in $C$, and $F\colon C\to D$ a dg functor (an object of $\Coh_\dg(C,D)$). The following expression follows directly from \eqref{eqe}:
\begin{equation}\label{eql1}
\begin{aligned}
\ &e(\Psi)\Big(\varepsilon(\Phi_1; f_{11},\dots,f_{1n_1}),\dots,\id_F\sotimes f,\dots,\varepsilon(\Phi_k; f_{k1},\dots,f_{kn_k})\Big)=\\
&\Psi\Big(\Phi_1(f_{11},\dots,f_{1n_1}),\dots,F(f),\dots,\Phi_k(f_{k1},\dots,f_{kn_k})\Big)
\end{aligned}
\end{equation}
Next, we compute the image of $e(\Psi)$ under the second map in \eqref{eqsm2}:
$$
p\colon \Coh_\dg(\Coh_\dg(C,D)\sotimes C,E)\to\Coh_\dg(\Coh_\dg(C,D),\Coh_\dg(C,E))
$$
Let $F_0\xrightarrow{\Phi_1}F_1\xrightarrow{\Phi_2}F_2\xrightarrow{\Phi_3}\dots \xrightarrow{\Phi_{n-1}}F_{n-1}\xrightarrow{\Phi_n}F_n$ be a composable chain of morphisms in $\Coh_\dg(C,D)$, and 
$\xrightarrow{f_1}\xrightarrow{f_2}\dots\xrightarrow{f_N}$ a composable chain of morphisms in $C$. By \eqref{eqmorphismp}, one has:
\begin{equation}\label{eql2}
\begin{aligned}
\ &p(e(\Psi))(\Phi_1,\dots,\Phi_n)(f_1,\dots,f_n)=\sum_{1\le i_1\le j_1\le i_2\le j_2\le \dots\le j_n\le N}\pm\\
&e(\Psi)\Big(\id_{F_0}\sotimes f_1,\dots,\id_{F_0}\sotimes f_{i_1},\ \varepsilon(\Phi_1; f_{i_1+1},\dots,f_{j_1}),\  \id_{F_1}\sotimes f_{j_1+1},\dots, \varepsilon(\Phi_2;f_{i_2+1},\dots,f_{j_2}),\ \id_{F_2}\sotimes f_{j_2+1},\dots\Big)
\end{aligned}
\end{equation}
where the signs are as in \eqref{eqmorphismp}. 

By \eqref{eql1},
\begin{equation}\label{eql3}
\begin{aligned}
\ &\text{r.h.s. of }(\ref{eql2})=\sum_{1\le i_1\le j_1\le i_2\le j_2\le \dots\le j_n\le N}\pm\\
&\Psi\Big(F_0(f_1),\dots,F_0(f_{i_1}),\Phi_1(f_{i_1+1},\dots,f_{j_1}), F_1(f_{j_1+1}),\dots,F_1(f_{i_2}),\Phi_2(f_{i_2+1},\dots,f_{j_2}), F_2(f_{j_2+1}),\dots\Big)
\end{aligned}
\end{equation}
with the signs as in \eqref{eqmorphismp}.

The r.h.s. of \eqref{eql3} is our final formula for $L(\Psi)$.

Below in Section \ref{sectionm} we get the the same map from somewhat different motivation. The main reason for keeping track of $e,p,L$ coming from the the twisted associativity $\alpha$ is that the corresponding map $M$ right adjoint to $L$ (which appear below under the name {\it twisted composition}) automatically obeys \eqref{eqcompskew1}-\eqref{eqcompskew3}, as soon as $\alpha$ obeys the skew associativity axioms. \footnote{The author is thankful to the anonymous referee who suggested to him this line of reasoning.} 
\endcomment

\subsubsection{\sc The twisted composition $M_{A,B,C}\colon \Coh_\dg(B,C)\sotimes\Coh_\dg(A,B)\to\Coh_\dg(A,C)$}
By definition of $M$, the following two maps 
$([B,C]\otimes [A,B])\otimes A\to C$ coincide:
$$
([B,C]\otimes [A,B])\otimes A\xrightarrow{\alpha}[B,C]\otimes ([A,B]\otimes A)\xrightarrow{\id\otimes e}[B,C]\otimes B\xrightarrow{e}C
$$
and
$$
([B,C]\otimes [A,B])\otimes A\xrightarrow{M\otimes\id}[A,C]\otimes A\xrightarrow{e}C
$$
Denote them $t_1$ and $t_2$, correspondingly. The formulas for $\alpha$ (see Theorem \ref{theorassoc}), and formulas \eqref{eqebis}, \eqref{eqebis2}, \eqref{eqe} for $e$ provide us with explicit map $t_1$, then we find $M$ from $t_2$. 

The result is as follows.

The dg functor $M$ is defined on objects, that is, pairs of dg functors $(G,F)$ where $F\colon A\to B, G\colon B\to C$, as the composition $G\circ F$. 

Determine it on the generator morphisms. They are of the following three types: (i) $\Theta\sotimes\id_F$, (ii) $G\sotimes\Psi$, (iii)
$\varepsilon(\Theta; \Psi_1,\dots,\Psi_n)$. 

A morphism $\Theta*\id_F$, for $\Theta\colon G_0\Rightarrow G_1\colon B\to C$, $F\colon A\to B$, is mapped to
\begin{equation}\label{eqmtx1}
M(\Theta*\id_F)(f_1,f_2,\dots,f_n)=\Theta(F(f_1),F(f_2),\dots,F(f_n))
\end{equation}
and a morphism $\id_G*\Psi$, for $\Psi\colon F_0\Rightarrow F_1\colon A\to B$, $G\colon B\to C$, is mapped to
\begin{equation}\label{eqmtx2}
M(\id_G*\Psi)(f_1,\dots,f_n)=G(\Psi(f_1,\dots,f_n))
\end{equation}

The expression for $M(\varepsilon(\Theta; \Psi_1,\dots,\Psi_n))$ is more complicated, of course.
Denote by $\Theta\{\Psi_k,\dots,\Psi_1\}_{i_1,\dots,i_{k}}$ the cochain shown in Figure \ref{fig8}. 
Here $X_0,\dots, X_N$ are objects of the dg category $A$, $F_0,\dots, F_k\colon A\to B$ are dg functors, $G_0,G_1\colon B\to C$ are dg functors,
$Psi_i\colon F_{i-1}\Rightarrow F_i\in \Coh_\dg(A,B)(F_{i-1},F_i)$, $\Theta\colon G_0\Rightarrow G_1\in\Coh_\dg(B,C)(G_0,G_1)$. 
Each small arrow at the bottom line indicates a morphism $X_j\to X_{j+1}$ in $A$. 

One has:
\begin{equation}\label{eqmtx3}
\begin{aligned}
\ &M(\varepsilon(\Theta;\ \Psi_1,\dots,\Psi_k))= 
\Theta\{\Psi_k,\dots,\Psi_1\}=
\sum_{i_1,\dots,i_{k}}\pm
\Theta\{\Psi_k,\dots,\Psi_1\}_{i_1,\dots,i_{k}}
\end{aligned}
\end{equation}
where the signs are specified in [Sh6, Section 4.2]. 

\sevafigc{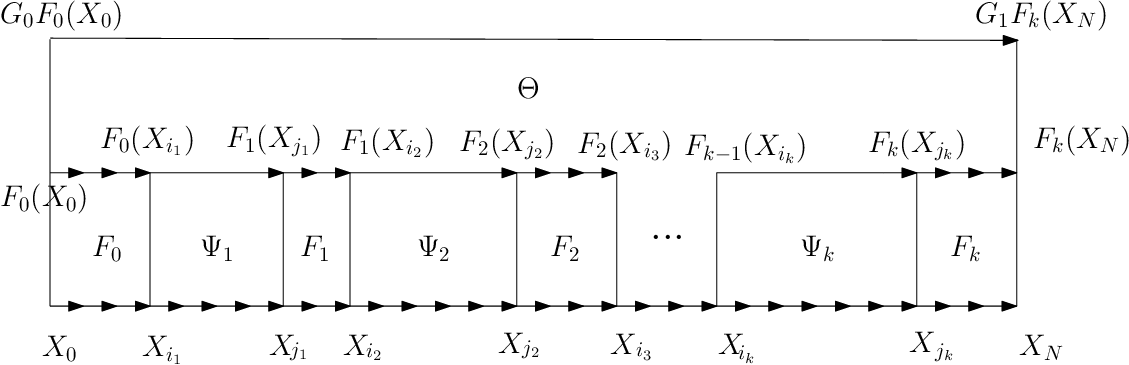}{120mm}{0}{The cochain $\Theta\{\Psi_k,\dots,\Psi_1\}_{i_1i_2,i_3,\dots,i_{k}}(f_N,\dots,f_1)$\label{fig8}}

We see that $M(\varepsilon(\Theta; \Psi_1,\dots,\Psi_n))$ is just a generalisation of the brace operations introduced in [GJ]. Note that within our approach these formulas are not definitions but a computation of $M$. In particular, \eqref{eqsuper}, \eqref{eqd0}, \eqref{eqd1} hold automatically. Not that \eqref{eqd0} becomes just a generalisation of the Getzler-Jones identity [GJ]
$$ 
[d, \Theta\{\Psi\}]=\Psi\cup\Theta\pm \Theta\cup\Psi
$$
for Hochschild cochains $\Hoch^\udot(A)$ of a dg algebra $A$.

\comment
\section{\sc The twisted composition}\label{sectionm}
\subsection{\sc }\label{sectionm1}
\subsubsection{\sc }\label{sectionm11}

As we noted in Section \ref{introbraces}, there is a similarity between our identity $(R_4)$, see \eqref{eqsuper}, with the well-known identity \eqref{brx2} for the compatibility of the brace operations $x\{y_1,\dots,y_n\}$ on the Hochschild cohomological complex with the cup-product.
Also, formula \eqref{eqd1} for $d(\varepsilon(f;\ g_1,\dots,g_n))$ looks similarly with the corresponding formula for $d(f\{g_1,\dots,g_n\})$, see \eqref{brx1}.

For three small dg categories $A,B,C$ over $\k$, we define a dg functor 
\begin{equation}\label{mapsuperx1}
M\colon \Coh_\dg(B,C)\sotimes \Coh_\dg(A,B)\to\Coh_\dg(A,C)
\end{equation}
as the left adjoint to the map $L$, see \eqref{eql3}.
We call it the {\it twisted composition}. (For general skew-closed categories, we denoted $M$ by $M$, see Section \ref{sectionscc}). 

The main fact on $M$ is the following
\begin{prop}
Diagrams \eqref{diagrcomp}, \eqref{diagrcomp2}, \eqref{diagrcomp3} below are commutative, for any four small dg categories $A,B,C,D$ over $\k$. 
\end{prop}
\begin{proof}
It is a particular case of general Proposition \ref{propsm}.
\end{proof}

The compatibility of $M$ and $\alpha$ is expressed by the commutative diagram below, for any four small dg categories $A,B,C,D$ over $\k$, the diagram
\begin{equation}\label{diagrcomp}
\xymatrix{
(\Coh_\dg(C,D)\sotimes \Coh_\dg(B,C))\sotimes\Coh_\dg(A,B)\ar[rr]^{\alpha}\ar[d]_{M}&&\Coh_\dg(C,D)\sotimes(\Coh_\dg(B,C)\sotimes\Coh_\dg(A,B))\ar[d]^{M}\\
\Coh_\dg(B,D)\sotimes \Coh_\dg(A,B)\ar[d]_{M}&&\Coh_\dg(C,D)\sotimes\Coh_\dg(A,C)\ar[d]^{M}\\
\Coh_\dg(A,D)\ar[rr]^{=}&&\Coh_\dg(A,D)
}
\end{equation}
Two remaining compatibilities of $M$ with unit maps are expressed by the two commutative diagrams below:
\begin{equation}\label{diagrcomp2}
\xymatrix{
\Coh_\dg(C,D)\sotimes\underline{\k}\ar[r]^{\id\sotimes j\hspace{10mm}}&\Coh_\dg(C,D)\sotimes\Coh_\dg(C,C)\ar[d]^{M}\\
\Coh_\dg(C,D)\ar[u]^{\rho}\ar[r]^{\id}&\Coh_\dg(C,D)
}
\end{equation}

\begin{equation}\label{diagrcomp3}
\xymatrix{
\Coh_\dg(D,D)\sotimes\Coh_\dg(C,D)\ar[r]^{\hspace{12mm}M}&\Coh_\dg(C,D)\\
\underline{\k}\sotimes \Coh_\dg(C,D)\ar[u]^{j\sotimes\id}\ar[r]^{\lambda}&\Coh_\dg(C,D)\ar[u]_{\id}
}
\end{equation}

In the rest of this Section, we provide a more explicit approach to $M$. In particular, we pursue an analogy with the brace operations on Hochschild cochains.

\subsubsection{\sc }\label{sectionm12}
Note that the dg functor $M$ corresponds, by  adjunction \eqref{adj1}, to a dg functor 
\begin{equation}\label{mapts2}
L\colon \Coh_\dg(B,C)\to \Coh_\dg(\Coh_\dg(A,B),\Coh_\dg(A,C))
\end{equation}
This dg functor is a dg categorical counterpart of the following well-known fact [Ts, Prop.4]:

Let $A$ be an associative dg algebra. Then one has a map of complexes
\begin{equation}\label{mapts}
T\colon C^\udot(A,A)[1]\to C^\udot(C^\udot(A,A),C^\udot(A,A))[1]
\end{equation}
\begin{equation}\label{maptsbis}
T(D)(D_1\otimes\dots\otimes D_n)=D\{D_1,\dots,D_n\}
\end{equation}
The map $T$ is a map of the shifted dg associative algebras. Moreover, it follows from [Ts, Prop.3] that the map $T$ is a map of dg algebras over the brace operad $\Br$.

We can shift the map \eqref{mapts} by [-1], and get a map
\begin{equation}\label{maptssh}
C^\udot(A,A)\to C^\udot(C^\udot(A,A),C^\udot(A,A))
\end{equation}

In the shifted form, the map \eqref{maptssh} is directly generalised to a more general setting of \eqref{mapts2}, and, by the adjunction, we get \eqref{mapsuperx1}.

\subsection{\sc The dg functor $M$}\label{sectionm2}
\subsubsection{\sc}
We define the dg functor  $M\colon \Coh_\dg(B,C)\sotimes \Coh_\dg(A,B)\to\Coh_\dg(A,C)$ on the objects, which are pairs $(G,F)$ of dg functors $F\colon A\to B, G\colon B\to C$,  as the composition:
$$
M(G,F)=G\circ F
$$
Let us define $M$ on the morphisms.

A morphism $\Theta*\id_F$, for $\Theta\colon G_0\Rightarrow G_1\colon B\to C$, $F\colon A\to B$, is mapped to
\begin{equation}\label{eqmtx1}
M(\Theta*\id_F)(f_1,f_2,\dots,f_n)=\Theta(F(f_1),F(f_2),\dots,F(f_n))
\end{equation}
and a morphism $\id_G*\Psi$, for $\Psi\colon F_0\Rightarrow F_1\colon A\to B$, $G\colon B\to C$, is mapped to
\begin{equation}\label{eqmtx2}
M(\id_G*\Psi)(f_1,\dots,f_n)=G(\Psi(f_1,\dots,f_n))
\end{equation}

We define $M$ as dg functor, therefore, composition of two morphisms in $\Coh_\dg(B,C)\sotimes\Coh_\dg(A,B)$ is mapped to the (standard) composition $\circ_v$ in $\Coh_\dg(A,C)$, see \eqref{compv}.

The images $M((\Theta* \id_{F_1})\sotimes (\id_{G_0}* \Psi))$ and $M((\id_{G_1}* \Psi)\sotimes (\Theta* \id_{F_0}))$ are coherent natural transformations $G_0F_0\Rightarrow G_1F_1\colon A\to C$, which both are candidates for the ``horizontal composition'' $\Theta\circ_h\Psi$.

Denote them $$\Theta\circ_h^{(1)}\Psi=M((\Theta* \id_{F_1})\sotimes (\id_{G_0}* \Psi))=M(\Theta*\id_{F_1})\circ_vM(\id_{G_0}*\Psi)$$ and $$\Theta\circ_h^{(2)}\Psi=M((\id_{G_1}* \Psi)\sotimes (\Theta* \id_{F_0}))=M(\id_{G_1}*\Psi)\circ_vM(\Theta*\id_{F_0})$$ The corresponding coherent natural transformations are shown in Figures \ref{fig3} and \ref{fig4}. These Figures should be understood as follows.

At the very bottom line, there is a chain of composable morphisms in $A$. There are ``boxes'' of two types: of type $\Psi,\Theta,\dots$, and of type $F_i,G_i,\dots$. A box of the type $\Psi,\Theta,\dots$ is drawn for the application of the corresponding {\it coherent natural transformation}; the output is a single arrow. A box of type $F_i,G_i,\dots$ is drawn for the application of the corresponding {\it dg functor}; we apply this dg functor to {\it each} of arrows in the chain of composable arrows in the input, the output is a chain of composable arrows having as many arrows as in the input. 
The very top box of the diagrams in Figures \ref{fig3} and \ref{fig4} is drawn for the {\it composition} of two arrows in the input, which is, up to a sign, the corresponding horizontal composition.

\sevafigc{pic3.eps}{100mm}{0}{The ``first'' horizontal composition $\Theta\circ_h^{(1)}\Psi$\label{fig3}}

\sevafigc{pic4.eps}{100mm}{0}{The ``second'' horizontal composition $\Theta\circ_h^{(2)}\Psi$\label{fig4}}

The matter is that two {\it different} horizontal compositions are possible. Explicitly, they are given as
\begin{equation}\label{comph1}
(\Theta\circ_h^{(1)}\Psi)(f_1\,\dots, f_{m+n})=(-1)^{|\Psi|(|f_{m+1}|+\dots+|f_{m+n}|+n)}\Theta(F_1(f_{m+1}),\dots, F_1(f_{m+n}))\circ G_0(\Psi(f_1,\dots, f_m))
\end{equation}
and
\begin{equation}\label{comph2}
(\Theta\circ_h^{(2)}\Psi)(f_1,\dots, f_{m+n})=(-1)^{|\Theta|(|f_{n+1}|+\dots+|f_{m+n}|+m)}G_1(\Psi(f_{n+1}\,\dots, f_{m+n}))  \circ \Theta(F_0(f_1),\dots, F_0(f_n))
\end{equation}
Note that the signs are defined by \eqref{compv}.

Next, we associate with the same pair $(\Psi,\Theta)$ a coherent natural transformation $h_i(\Theta,\Psi)$ (defined on a string of $m+n-1$ composable morphisms), $0\le i\le n$, see Figure \ref{fig5}.

One has:
\begin{equation}
h_i(\Theta,\Psi)(f_1,\dots, f_{m+n-1})=\Theta\big(F_0(f_1),\dots, F_0(f_i), \Psi(f_{i+1},\dots,f_{i+m}), F_1(f_{i+m+1})\,\dots, F_1(f_{m+n-1})\big)
\end{equation}

Then we set
\begin{equation}
\begin{aligned}
\ &(\Theta\{\Psi\}_{[-1]})(f_{m+n-1},\dots,f_1)=\sum_{i=0}^n(-1)^{|\Psi|(|f_{i+m+1}|+\dots+|f_{m+n-1}|+n-i-1)}h_i(\Theta,\Psi)(f_{m+n-1},\dots,f_1)
\end{aligned}
\end{equation}
\begin{equation}
M(\varepsilon(\Theta;\ \Psi))=\Theta\{\Psi\}_{[-1]}
\end{equation}

\sevafigc{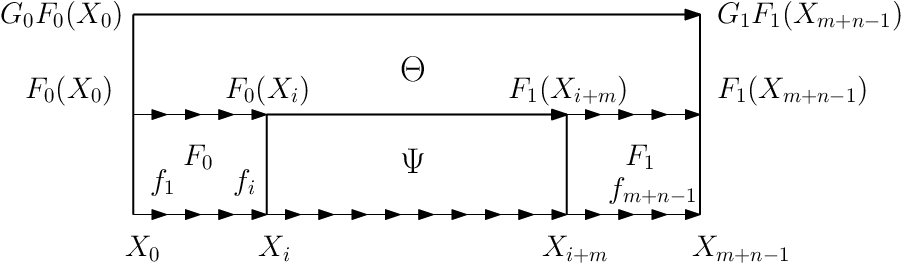}{100mm}{0}{The cochain $h_i(\Theta,\Psi)(f_{m+n-1},\dots,f_1)$\label{fig5}}

\begin{lemma}
One has:
\begin{equation}\label{eqnice1}
[d,M(\varepsilon(\Theta;\ \Psi))]=(-1)^{|\Theta|-1}\Theta\circ_h^{(1)}\Psi+(-1)^{|\Theta|(|\Psi|-1)}\Theta\circ_h^{(2)}\Psi
\end{equation}
\end{lemma}
\begin{proof}
It is a direct check.
\end{proof}

\subsubsection{\sc }

Finally, we define $M(\varepsilon(\Theta;\ \Psi_1,\dots,\Psi_k))$.

\sevafigc{pic8.eps}{120mm}{0}{The cochain $\Theta\{\Psi_k,\dots,\Psi_1\}_{[-1],i_1i_2,i_3,\dots,i_{k}}(f_N,\dots,f_1)$\label{fig8}}

We use notation $\Theta\{\Psi_k,\dots,\Psi_1\}_{[-1],i_1i_2,i_3,\dots,i_{k}}$ for the cochain shown in Figure \ref{fig8}.

One sets:
\begin{equation}
\Theta\{\Psi_k,\dots,\Psi_1\}_{[-1]}(f_N,\dots,f_1)=\sum_{i_1,\dots,i_{k}}(-1)^{\sum_{s=1}^k|\Psi_s|(\sum_{j\ge i_{s}}(|f_j|+1))}
\Theta\{\Psi_k,\dots,\Psi_1\}_{[-1],i_1,\dots,i_{k}}(f_N,\dots,f_1)
\end{equation}

\begin{equation}\label{eqmtx3}
\begin{aligned}
\ &M(\varepsilon(\Theta;\ \Psi_1,\dots,\Psi_k))= 
\Theta\{\Psi_k,\dots,\Psi_1\}_{[-1]}=
\sum_{i_1,\dots,i_{k}}
\Theta\{\Psi_k,\dots,\Psi_1\}_{[-1],i_1,\dots,i_{k}}
\end{aligned}
\end{equation}
\begin{prop}
Equations \eqref{eqmtx1},\eqref{eqmtx2},\eqref{eqmtx3} define a dg functor 
$$
M\colon \Coh_\dg(B,C)\sotimes \Coh_\dg(A,B)\to\Coh_\dg(A,C)
$$
That is, the following identities hold:
\begin{equation}\label{eqpropmt1}
d(M(\varepsilon(\Theta;\ \Psi_1,\dots,\Psi_k)))=M(d\varepsilon(\Theta;\ \Psi_1,\dots,\Psi_k))
\end{equation}
\begin{equation}\label{eqpropmt2}
\begin{aligned}
\ &\sum_{\ell=0}^k(-1)^{|\Theta_1|(|\Psi_{\ell+1}|+\dots+|\Psi_k|+k-\ell)}M(\varepsilon(\Theta_2;\ \Psi_{\ell+1},\dots,\Psi_k))\cup M(\varepsilon(\Theta_1;\ \Psi_1,\dots,\Psi_\ell))=\\
&M(\varepsilon(\Theta_2\cup \Theta_1;\ \Psi_1,\dots,\Psi_k))
\end{aligned}
\end{equation}
\end{prop}

\begin{proof}
By adjunction \eqref{adj1}, any dg functor $\tilde{\phi}_1\colon \Coh_\dg(B,C)\to\Coh_\dg(\Coh_\dg(A,B),\Coh_\dg(A,C))$
gives a unique dg functor $\phi\colon \Coh_\dg(B,C)\sotimes\Coh_\dg(A,B)\to\Coh_\dg(A,C)$.

The (shifted) Tsygan map \eqref{maptssh} is directly generalised to a more general dg functor
$$
M_1\colon \Coh_\dg(B,C)\to\Coh_\dg(\Coh_\dg(A,B),\Coh_\dg(A,C))
$$
$$
\Theta\mapsto (\Psi_1\otimes\dots\otimes\Psi_k\mapsto \Theta\{\Psi_1,\dots,\Psi_k\}_{[-1]})
$$
The fact that this map $M_1$ is a dg functor is immediately translated from the proof that \eqref{maptssh} is a map of complexes and preserves the cup-products, see [Ts, Prop. 3 and 4] and Appendix \ref{app13}.

By adjunction \eqref{adj1}, it gives a dg functor $M$. One checks directly that it is given by \eqref{eqmtx3}.

\end{proof}

\subsection{\sc The compatibility of $M$ with $\alpha$}\label{sectionm3}
Here we prove
\begin{prop}
The twisted composition $M$ and the one-side associativity map $\alpha$ are compatible so that the diagram \eqref{diagrcomp} commutes. 
\end{prop}

\begin{proof}

\end{proof}

\begin{remark}{\rm
The commutativity of \eqref{diagrcomp} essentially amounts to the identity 
\begin{equation}
\begin{aligned}
\ &(\Theta\{\Psi_1,\dots,\Psi_m\}_{[-1]})\{\Gamma_1,\dots,\Gamma_N\}_{[-1]}=\sum_{1\le i_1\le j_1\le i_2\le \dots\le j_m\le N}
(-1)^{\sum_{\ell=1}^m|\Psi_\ell|(\sum_{s\le i_\ell}(|\Gamma_s|-1))  }\\
&\Theta\{\Gamma_1,\dots,\Gamma_{i_1},\Psi_1\{\Gamma_{i_1+1},\dots,\Gamma_{j_1}\}_{[-1]},\Gamma_{j_1+1},\dots,\Gamma_{i_2},\Psi_2\{\Gamma_{i_2+1},\dots,\Gamma_{j_2}\}_{[-1]},\dots\}_{[-1]}
\end{aligned}
\end{equation}
for $\Theta\in \Coh_\dg(C,D)$, $\Psi_1,\dots,\Psi_m\in\Coh_\dg(B,C), \Gamma_1,\dots,\Gamma_N\in\Coh_\dg(A,B)$.

For the case of $C^\udot(A,A)$ this identity is known, see \eqref{eqapp5x}. We refer it to as the Tsygan identity. The general case amounts to the same phenomenon. 
}
\end{remark}

\endcomment

\section{\sc The 2-operad $\mathcal{O}$}\label{section2opn}
We refer the reader to Appendix \ref{sectionba} for a brief and elementary account on 2-operads. For more thorough treatment, see [Ba3-5], [BM1,2]. 
\subsection{\sc }\label{section2opn1}
Recall our notation
$$
I_{n_1,\dots,n_k}=I_{n_k}\sotimes (I_{n_{k-1}}\sotimes (\dots (I_{n_2}\sotimes I_{n_1})\dots))
$$
Below we use a shorter form of it:
$$
I_{n_1,\dots,n_k}=I_D
$$
where $D=(n_1,\dots,n_k)$ is the corresponding 2-globular pasting diagram (see Section \ref{sectionbcomp}).

Let $D_1=(n_1^1,\dots,n_{k_1}^1), D_2=(n_1^2,\dots,n_{k_2}^2)$ be two 2-globular pasting diagrams. We denote
$$
[D_1,D_2]=(n_1^1,\dots,n_{k_1}^1,n_1^2,\dots,n_{k_2}^2)
$$
the 2-globular pasting diagram obtained by the horizontal concatenation of the 2-globular pasting diagrams $D_1$ and $D_2$.

For a sequence $D_1,\dots,D_t$ of 2-globular pasting diagrams, we define similarly the total 2-globular pasting diagram $[D_1,\dots,D_t]$, so that
$$
[D_1,\dots,D_t]=[[D_1,\dots, D_{t-1}],D_t]
$$

We denote the ordered sequence $D_1,\dots,D_t$ by $\mathbf{D}$, and use the notation
\begin{equation}\label{eqvert0}
I_{\mathbf{D}}=I_{D_1,\dots,D_t}=I_{D_t}\sotimes (I_{D_{t-1}}\sotimes (\dots (I_{D_2}\sotimes I_{D_1})\dots))
\end{equation}

Let $D_1,\dots,D_t$ be a sequence of 2-globular pasting diagrams. We construct a dg functor
\begin{equation}\label{mapup}
\Upsilon(D_1,\dots,D_t)\colon I_{D_1,\dots,D_t}\to I_{[D_t, D_{t-1},\dots, D_1]}
\end{equation}

Let us start with the case $t=2$. The dg functor $\Upsilon(D_1,D_2)\colon I_{D_1}\sotimes I_{D_2}\to I_{[D_2,D_1]}$ is constructed as
\begin{equation}
\begin{aligned}
\ &\Big(I_{n_k}\sotimes\big(I_{n_{k-1}}\sotimes(\dots \sotimes (I_{n_2}\sotimes I_{n_1})\dots)\big)\Big)\sotimes \Big(I_{m_\ell}\sotimes \big(I_{m_{\ell-1}}\sotimes(\dots\sotimes (I_{m_2}\sotimes I_{m_1})\dots)\big)\Big)\xrightarrow{\alpha}\\
&I_{n_k}\sotimes\Big(\big(I_{n_{k-1}}\sotimes (I_{n_{k-2}}\sotimes \dots(I_{n_2}\sotimes I_{n_1})\dots)\big)\sotimes  \Big(I_{m_\ell}\sotimes (I_{m_{\ell-1}}\sotimes(\dots\sotimes (I_{m_2}\sotimes I_{m_1})\dots))\Big)\Big)\xrightarrow{\alpha}\\
&I_{n_k}\sotimes \Big(I_{n_{k-1}}\sotimes \Big(\big(I_{n_{k-2}}\sotimes (\dots)\big)\sotimes\Big(I_{m_\ell}\sotimes (I_{m_{\ell-1}}\sotimes(\dots\sotimes (I_{m_2}\sotimes I_{m_1})\dots))\Big)\Big)\Big)\xrightarrow{\alpha}\\
&\dots\xrightarrow{\alpha}I_{[D_2,D_1]}
\end{aligned}
\end{equation}
It is the composition of maps, each of which is an appropriate associativity constraint $\alpha$, see Theorem \ref{theorassoc}.

Now we define $\Upsilon(D_1,\dots,D_t)$ for any $t$ as the composition
\begin{equation}
I_{D_1,D_2,\dots,D_t}\to I_{[D_2,D_1],D_3,\dots,D_t}\to I_{[D_3,D_2,D_1],D_4,\dots,D_t}\to\dots\to I_{[D_t,D_{t-1},\dots,D_1]}
\end{equation}
of arrows each of this is given by $\Upsilon(D^\prime,D^\pprime)$.

The main technical point is that the maps $\Upsilon(D_1,\dots,D_t)$ are subject to some associativity, which we are going to formulate.

Let $D_1^1,\dots,D_{t_1}^1;\dots;D^k_1,\dots,D^k_{t_k}$ sequence of (sequences of) 2-globular pasting diagrams. 

We use notation
$$
I_{\mathbf{D}^i}=I_{D_1^i,\dots,D^i_{t_i}}
$$
Consider
\begin{equation}\label{notbf}
I_{\mathbf{D}^1,\dots,\mathbf{D}^k}:=I_{\mathbf{D}^k}\sotimes(I_{\mathbf{D}^{k-1}}\sotimes(\dots\sotimes (I_{\mathbf{D}^2}\sotimes I_{\mathbf{D}^1})\dots))
\end{equation}
As well, denote 
$$
[\mathbf{D}^i]:=[D^i_{t_i},\dots,D^i_1]
$$
and
$$
[[\mathbf{D}^k,\dots,\mathbf{D}^1]]:=[[\mathbf{D}^k],\dots,[\mathbf{D}^1]]
$$
There are two maps
\begin{equation}
\Upsilon_1,\Upsilon_2\colon I_{\mathbf{D}^1,\dots,\mathbf{D}^k}\to I_{[[\mathbf{D}^k,\dots,\mathbf{D}^1]]}
\end{equation}
They are defined as follows:
\begin{equation}\label{ups1}
\Upsilon_1\colon \ \ I_{\mathbf{D}^1,\dots,\mathbf{D}^k}\xrightarrow{\Upsilon} I_{[\mathbf{D}^1],\dots,[\mathbf{D}^k]}\xrightarrow{\Upsilon} I_{[[\mathbf{D}^k],\dots,[\mathbf{D}^1]]}
\end{equation}
and
\begin{equation}\label{ups2}
\Upsilon_2\colon\ \  I_{\mathbf{D}^1,\dots,\mathbf{D}^k}\xrightarrow{\tilde{\Upsilon}}I_{D_1^1,\dots,D_{t_1}^1,\dots,D_1^{k},\dots,D_{t_k}^k}\xrightarrow{\Upsilon}I_{[[\mathbf{D}^k,\dots,\mathbf{D}^1]]}
\end{equation}
In \eqref{ups2}, the first arrow in not literally equal to $\Upsilon(-,\dots,-)$, but is defined similarly; we leave the details to the reader.

\begin{prop}\label{propups}
In the notations as above, the two maps $\Upsilon_1,\Upsilon_2\colon I_{\mathbf{D}^1,\dots,\mathbf{D}^k}\to I_{[[\mathbf{D}^k,\dots,\mathbf{D}^1]]}$ are equal.
\end{prop}
\begin{proof}
It follows from Theorem \ref{theorskew} and Proposition \ref{propskew}.
\end{proof}

\subsection{\sc }\label{section2opn2}
Recall the 2-sequence $\mathcal{O}$. 

For a 2-globular pasting diagram $D=(n_1,\dots,n_k)$, set 
$$
I_{n_1,\dots,n_k}=I_{n_k}\sotimes (I_{n_{k-1}}\sotimes (\dots (I_{n_2}\sotimes I_{n_1})\dots))
$$
and 
\begin{equation}
\mathcal{O}(D)=I_{n_1,\dots,n_k}(\min,\max)
\end{equation}
where $\min=(0,0,\dots,0),\max=(n_k,\dots,n_1)$.
Recall that all $n_i\ge 1$. 

We start with the following Lemma:

\begin{lemma}\label{lemmann}
Let $D$ be a 2-globular pasting diagram and let $i\colon D^0\to  D$ be a connected subdiagram having the same set of objects. Let $D=(n_1,\dots,n_k), i(D^0)=([a_1,b_1],[a_2,b_2],\dots,[a_k,b_k])$, $0\le a_i<b_i\le n_i$, $D^0=(b_1-a_1,b_2-a_2,\dots,b_k-a_k)$. Then
$$
I_{D^0}(\min,\max)=I_D(\min_0,\max_0)
$$
where $\min_0=(a_k,a_{k-1},\dots,a_1)$, $\max_0=(b_k,b_{k-1},\dots,b_0)$.
\end{lemma}
\begin{proof}
It is clear. More generally, the embedding $I_{D^0}\to I_D$ is fully faithful. 
\end{proof}

More generally, let $D_1,\dots,D_n$ be 2-globular pasting diagrams,  $I_{D_1,\dots,D_n}$ the dg category \eqref{eqvert0}.
\begin{lemma}\label{lemmann2}
Let $D_1,\dots,D_n$ be 2-globular pasting diagram, and let $i_k\colon D_k^0\to D_k$ connected subdiagrams, $1\le k\le n$, such that $i_k(\min)=\min_0^k, i_k(\max)=\max_0^k$. Denote $\min_{0,\tot}=(\min_n^0,\dots,\min_1^0)$, $\max_{0,\tot}=(\max_n^0,\dots,\max_1^0)$. Then
\begin{equation}\label{eqvert3}
I_{D_1^0,\dots,D_n^0}(\min,\max)=I_{D_1,\dots,D_n}(\min_{0,\tot},\max_{0,\tot})
\end{equation}
and the following diagram commutes:
\begin{equation}\label{eqvert4}
\xymatrix{
I_{D_1^0,\dots,D_n^0}(\min,\max)\ar[rr]^{\Upsilon}\ar[d]&&I_{[D_1^0,\dots,D_n^0]}(\min,\max)\ar[d]\\
I_{D_1,\dots,D_n}(\min_{0,\tot},\max_{0,\tot})\ar[rr]^{\Upsilon}&&I_{[D_1,\dots,D_n]}(\min_{0,\tot},\max_{0,\tot})
}
\end{equation}
where $\Upsilon$ is the map \eqref{mapup}, and the vertical maps are given by \eqref{eqvert3}.  More generally, the diagram of dg categories below commutes:
\begin{equation}\label{eqvert5}
\xymatrix{
I_{D_1^0,\dots,D_n^0}\ar[rr]^{\Upsilon}\ar[d]&&I_{[D_1^0,\dots,D_n^0]}\ar[d]\\
I_{D_1,\dots,D_n}\ar[rr]^{\Upsilon}&&I_{[D_1,\dots,D_n]}
}
\end{equation}
\end{lemma}
\begin{proof}
The first claim \eqref{eqvert3} easily follows from the fully faithfulness of the embedding $I_{D_1^0,\dots,D_n^0}\to I_{D_1,\dots,D_n}$.
For the commutativity of \eqref{eqvert5}, note that the map $\Upsilon$ was defined via an iterative application of the associativity map $\alpha$. This $\alpha$ is a natural transformation, thus is functorial with respect to the dg functors (the morphisms in $\Cat_\dg(\k)$). 
It follows that $\Upsilon$ is functorial for the dg functor $I_{D_1^0,\dots,D_n^0}\to I_{D_1,\dots,D_n}$.
Then the commutativity of \eqref{eqvert4} follows from the commutativity of \eqref{eqvert5} and from \eqref{eqvert3}.
\end{proof}

Let $D_1,D_2,\dots,D_\ell$ be 2-globular pasting diagrams, each of which has $k+1$ vertices.
Denote by $D_1\circ \dots\circ D_\ell$ their ``vertical product'', which identifies the maximal element in $D_a(i,i+1)$ with minimal element in $D_{a+1}(i,i+1)$, where $0\le i\le k$, $1\le a\le \ell-1$. 
Thus, $D_1\circ\dots\circ D_\ell$ has $k+1$ vertices, and $(D_1\circ\dots\circ D_\ell)(i,i+1)$ has $\sum_{s=1}^\ell\sharp D_s(i,i+1)-\ell+1$ elements. 

Assume that $D_s=(n_{1s},\dots,n_{ks})$, $s=1,\dots,\ell$. Then 
$$
D_1\circ\dots\circ D_\ell=(\sum_{s=1}^\ell n_{1s},\dots,\sum_{s=1}^\ell n_{\ell s})
$$
Denote by $\min_i,\max_i$ the minimal and the maximal object of the image of embedding of $D_i$ to $D_1\circ\dots\circ D_\ell$.
Clearly $\min_1=\min$ and $\max_\ell=\max$.
One has the following composition
\begin{equation}\label{eqvertcn1}
\begin{aligned}
\ &\mathcal{O}({D_1})\otimes\dots\otimes \mathcal{O}(D_\ell)\simeq I_{D_1\circ\dots\circ D_\ell}(\min_\ell,\max_\ell)\otimes\dots\otimes 
I_{D_1\circ\dots\circ D_\ell}(\min_1,\max_1)\to\\
&I_{D_1\circ\dots\circ D_\ell}(\min,\max)=\mathcal{O}(D_1\circ\dots\circ D_\ell)
\end{aligned}
\end{equation}
where the first arrow is given by Lemma \ref{lemmann}, and the second one is the composition in $D_1\circ\dots\circ D_\ell$.

We denote the map \eqref{eqvertcn1} by $m_v$ (where the subscript $v$ stands for ``vertical''). 

Assume $D_1$ and $D_2$ have equal numbers of objects. we denote by $I_{D_1}\tilde{\circ}I_{D_2}$ their categorical cofibred sum (pushout) over the discrete category of objects, where the object $\max$ of $I_{D_1}$ is identified with the object $\min$ of $I_{D_2}$. We use similar notation for more irregular parenthesizing. For example, assume $D_1$ and $D_2$ have equal numbers of objects, as well as $D_3$ and $D_4$. Then we use notation $(I_{D_1}\sotimes I_{D_3})\tilde{\circ}(I_{D_2}\sotimes I_{D_4})$, etc.

We denote by $M_v$ the dg functor
$$
M_v\colon I_{D_1}\tilde{\circ}I_{D_2}\to I_{D_1\circ D_2}
$$
defined by the two embeddings $I_{D_1}\to I_{D_1\circ D_2}$ and $I_{D_2}\to I_{D_1\circ D_2}$. These two embeddings  clearly agree on the $\max\in I_{D_1}$ and $\min\in I_{D_2}$, and therefore define a map from the categorical pushout. 

The following lemma can be regarded as a sort of the Eckmann-Hilton compatibility. It plays an essential role in the proof of the associativity of the operadic composition in Section \ref{sectionopcomp}. 
\begin{lemma}\label{lemmann3}
Assume $D_1,D_2,D_3,D_4$ are 2-globular pasting diagrams such that the number of objects in $D_1$ is equal to the number of objects in $D_2$, as well as the numbers of objects in $D_3,D_4$ (so that $D_1\circ D_2$ and $D_3\circ D_4$ are defined). Denote by $D=[D_1,D_3]\circ [D_2,D_4]=
[D_1\circ D_2, D_3\circ D_4]$. The following diagram is commutative:
\begin{equation}\label{eqehop}
\xymatrix{
(I_{D_1}\sotimes I_{D_3})\tilde{\circ}(I_{D_2}\sotimes I_{D_4})\ar[rr]^{\hspace{4mm}\Upsilon\tilde{\circ} \Upsilon}\ar[d]_{M_v\sotimes M_v}&&I_{[D_1,D_3]}\tilde{\circ}I_{[D_2,D_4]}\ar[d]^{M_v}\\
I_{D_1\circ D_2}\sotimes I_{D_3\circ D_4}\ar[rr]^{\hspace{6mm}\Upsilon}&&I_{D}
}
\end{equation}
where the vertical maps $m_v$ are maps \eqref{eqvertcn1}, and the horizontal maps are $\Upsilon$ maps \eqref{mapup}. 
\end{lemma}
\begin{proof}
The diagrams
\begin{equation}
\xymatrix{
I_{D_1}\sotimes I_{D_3}\ar[r]\ar[d]&I_{[D_1,D_3]}\ar[d]\\
I_{D_1\circ D_2}\sotimes I_{D_3\circ D_4}\ar[r]&I_D
}
\end{equation}
and
\begin{equation}
\xymatrix{
I_{D_2}\sotimes I_{D_4}\ar[r]\ar[d]&I_{[D_2,D_4]}\ar[d]\\
I_{D_1\circ D_2}\sotimes I_{D_3\circ D_4}\ar[r]&I_D
}
\end{equation}
commute by Lemma \ref{lemmann2}, and agree on the discrete subcategories of minimal-maximal objects. Thus they define the commutative diagram on pushouts. 
\end{proof}

In what follows we refer to Lemma \ref{lemmann3} for a more general statement, with iterated horizontal and iterated vertical composition.

\subsection{\sc The 2-operadic composition}\label{sectionopcomp}

Now we are ready to define the 2-operadic composition \eqref{eqopcomp} on the 2-sequence $\{\mathcal{O}_D\}$.

We start with a simple example. 

Consider four 2-globular pasting diagrams $D_1,D_2,D_3,D_4$, such that $D_1$ and $D_2$, as well as $D_3$ and $D_4$ have equal numbers of objects. In particular, the compositions $D_1\circ D_2$ and $D_3\circ D_4$ are defined. Denote by $D$ the total diagram 
$D=[D_1\circ D_2, D_3\circ D_4]=[D_1,D_3]\circ [D_2,D_4]$, see Figure \ref{figureem}. Denote $D_0=(2,2)$.

\sevafigc{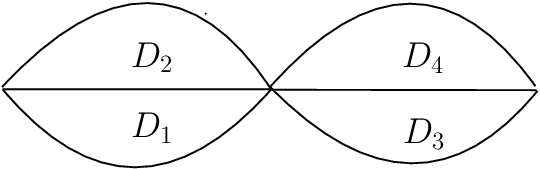}{70mm}{0}{\label{figureem}}

We define a composition
\begin{equation}
\mathcal{O}(D_0)\otimes\Big(\mathcal{O}(D_1)\otimes\mathcal{O}(D_2)\otimes\mathcal{O}(D_3)\otimes\mathcal{O}(D_4)\Big)\to\mathcal{O}(D)
\end{equation}

Assume we are given elements $\Psi_i\in\mathcal{O}(D_i)$, $1\le i\le 4$, and an element $\Theta\in \mathcal{O}(D_0)$.

We are going to define the composition $\Theta(\Psi_1,\Psi_2,\Psi_3,\Psi_4)$. 

Consider the dg category $I_2\sotimes I_2$. Denote by $e_1\in I_2(0,1)$ and $e_2\in I_2(1,2)$ the generators of the left copy of $I_2$, and by $e_3,e_4$ the corresponding generators of the right copy. Any morphism in $I_2\sotimes I_2((0,0),(2,2))$ is a sum of monomials, each of which contains each of $e_1,e_2,e_3,e_4$ exactly 1 time. 

For example, consider $(\id_2*e_4)\circ(\id_2*e_3)\circ (e_2*\id_0)\circ (e_1*\id_0)$, or $\varepsilon(e_2;e_4)\circ \varepsilon(e_1;e_3)$. The point is that an arbitrary monomial contains each $e_i$ exactly 1 time. 

The element $\Theta$ is a sum of such monomials, in a unique way. For simplicity, we may assume that $\Theta$ is a single monomial.
The idea is to substitute $\Psi_1$ for $e_1$, $\Psi_2$ for $e_2$, $\Psi_3$ for $e_3$, and $\Psi_4$ for $e_4$. 

This substitution does not give directly an element in $\mathcal{O}(D)$, due to a ``wrong'' parenthesizing. 
The element one obtains after the substitution belongs to $(I_{D_1\circ D_2}\sotimes I_{D_3\circ D_4})(\min,\max)$. 
Then we apply the map $\Upsilon$ to it, see \eqref{mapup}:
$$
I_{D_1\circ D_2}\sotimes I_{D_3\circ D_4}\xrightarrow{\Upsilon}I_{[D_1\circ D_2,D_3\circ D_4]}=I_D
$$
It gives an element in $I_D(\min,\max)$. 

 \vspace{3mm}

The general composition is similar to this example.\\
Let $P\colon U\to V$ be a map of 2-globular pasting diagrams. 
Recall that it is defined as a dominant map $[P]\colon [V] \to [U]$ of  the corresponding free 2-categories, see \eqref{eqopcomp} and Section \ref{sectionbcomp}. 
The image of $[P]$ gives a ``subdivision'' of $[U]$ into smaller diagrams. Namely, for each minimal ball $\nu$ in $[V]$, consider $P^{-1}(\nu):=[P](\nu)$. The 2-globular pasting diagram $U$ is divided into the union $U=\cup_{\nu\in \mathcal{F}(V)}P^{-1}(\nu)$.

Let $V=(m_1,\dots,m_p), U=(n_1,\dots,n_k)$, and 
let $[P](0)=d_0=0, [P](1)=d_1, [P](2)=d_2,\dots,[P](p)=d_p=k$. 

The set $\{\nu\in V(i,i+1)\}$ has a natural order, for a given $0\le i\le p-1$. Let 
$D_{i1}, \dots , D_{im_i}$ be the ordered set $\{P^{-1}(\nu)|\ \nu \in V(i,i+1)\}$, with the corresponding order. 
Finally, define 
\begin{equation}\label{dcirc}
\circD_i:=D_{i1}\circ \dots \circ D_{im_i}
\end{equation}
It is clear that
\begin{equation}\label{compnn2}
[\circD_1,\circD_2,\dots,\circD_p]=U
\end{equation}

We have to define an operadic composition
$$
\Op\colon \mathcal{O}(V)\otimes (\otimes_{\nu\in \mathcal{F}(V)}\mathcal{O}(P^{-1}(\nu)))\to\mathcal{O}(U) 
$$
where $\mathcal{F}(V)$ is the set of all 2-morphisms in $V$.

An element $\xi_\nu$  in $\mathcal{O}(P^{-1}(\nu)$ is given by a morphism in $I_{P^{-1}}(\nu)(\min,\max)$.

An element $\omega$ in $\mathcal{O}(V)$ is uniquely a sum of monomials, in each of which each indecomposable element in $e_{\nu}\in I_{m_j}(s,s+1)$, $1\le j\le p$, $0\le s\le m_j-1$, occurs exactly once.
We plug $\xi_\nu$ in place of $e_{\nu}$. What we get is an element in
$$
m(\omega,\{\xi_\nu\})\in \Big(I_{\circD_1}\sotimes\dots\sotimes I_{\circD_p}\Big)(\min,\max)
$$ 
due to Lemma \ref{lemmann2}.
Although $[\circD_1,\dots,\circD_p]=D$, the category $I_{\circD_1}\sotimes\dots\sotimes I_{\circD_p}$ differs from $I_D$, due to another parenthesizing. 
The element 
\begin{equation}\label{thecomposition}
\Upsilon(\circD_1,\dots,\circD_p)(m(\omega,\{\xi_\nu\}))=:\Op(\omega,\{\xi_\nu\})
\end{equation}
is, by definition, the result of our operadic composition (see \eqref{mapup}). Clearly it is linear in each argument.

\begin{theorem}
The operation $\Op$, defined in \eqref{thecomposition}, fulfils the identities (i)-(iii) in Definition \ref{def2op}. That is, it makes the 2-sequence $\mathcal{O}$ a dg pruned reduced 1-terminal  2-operad.
\end{theorem}
\begin{proof}
(i) and (iii) in Definition \eqref{def2op} are clear. 

Prove the identity (ii). Consider two maps of 2-globular pasting diagrams $U\xrightarrow{P}V\xrightarrow{Q}W$, and prove the associativity for this chain. The maps above are defined via maps of strict 2-categories generated by the corresponding 2-globular sets 
$$
[W]\xrightarrow{Q}[V]\xrightarrow{P}[U]
$$
Let $W=(\ell_1,\dots,\ell_t)$, $V=(m_1,\dots,m_p)$, $U=(n_1,\dots,n_k)$.
We use notations $\circD_{1V},\dots,\circD_{tV}$ for the subdiagrams of $V$, associated with the map $Q$ as in  \eqref{dcirc}, and the notations $\circD_{1U},\dots, \circD_{pU}$ for the subdiagrams of $U$, similarly associated with the map $P$. We have 
$$
[\circD_{1V},\dots,\circD_{tV}]=V\text{   and   }[\circD_{1U},\dots,\circD_{pU}]=U
$$
We have also a subdivision of $U$ into bigger diagrams, associated with the composition $P\circ Q$. Assume that $\circD_{1V}$ has $a_1+1$ vertices, ..., $\circD_{tV}$ has $a_t+1$ vertices. Taking images with $P$, we get a subdivision 
$$
\begin{aligned}
\ &\circD_{1U+}=[\circD_{1U},\dots,\circD_{a_1,U}]\\
&\dots\\
&\circD_{tU+}=[\circD_{a_1+\dots,a_{t-1}+1,U},\dots,\circD_{pU}]
\end{aligned}
$$
where $\circD_{1U+},\dots,\circD_{tU+}$ are the diagrams associated with the composition $P\circ Q$, as in \eqref{dcirc}. Also,
$$
[\circD_{1UV},\dots,\circD_{tUV}]=U
$$
Denote
$$
\begin{aligned}
 &\mathbf{D}^\circ_{1U+}=(\circD_{1U},\dots,\circD_{a_1U})\\
 &\dots\\
 &\mathbf{D}^\circ_{tU+}=(\circD_{a_1+\dots,a_{t-1}+1,U},\dots,\circD_{pU})
 \end{aligned}
$$
Consider elements  $\omega\in \mathcal{O}(W)$, $\xi_\nu\in \{\mathcal{O}(Q^{-1}(\nu))\}_{\nu\in\mathcal{F}(W)}$, $\{\eta_\mu\in \mathcal{O}(P^{-1}(\mu))\}_{\mu\in\mathcal{F}(V)}$.

Now we construct the ``double substitution'', denoted by $m(\omega; \{\xi_\nu\},; \{\eta_\mu\})$, as follows. The element $\omega$ is uniquely a sum of monomials in the elementary generators $\{e_\nu\}_{\nu\in \mathcal{F}(W)}$. At first, we substitute the elements $\xi_\nu$ in place of the corresponding generators $e_\nu$.
What we get, is a morphism in $I_{D_{1V},\dots, D_{tV}}(\min,\max)$. Once again, any such morphism is uniquely a sum of monomials, each of which contains each elementary morphism $\{e_\mu\}_{\mu\in \mathcal{F}(V)}$ exactly 1 time. Then we plug the elements $\eta_\mu$ in place of the corresponding $e_\mu$. What we get is a morphism 
$$m(\omega;\{\xi_\nu\};\{\eta_\mu\})\in I_{\mathbf{D}^\circ_{1U+},\dots, \mathbf{D}^\circ_{tU+}}(\min,\max)$$ (We use implicitly Lemmas \ref{lemmann2} and \ref{lemmann3}).

There are two maps
\begin{equation}\label{operadprf2}
\Upsilon_1,\Upsilon_2: I_{\mathbf{D}^\circ_{1U+},\dots, \mathbf{D}^\circ_{tU+}}\to I_{[[\mathbf{D}^\circ_{1U+},\dots, \mathbf{D}^\circ_{tU+}]]}=I_U
\end{equation}
constructed in \eqref{ups1}, \eqref{ups2}. 

The application of $\Upsilon_1$ and $\Upsilon_2$ to $m(\omega;\{\xi_\nu\};\{\eta_\mu\})$ are identified with the two ways of the operadic compositions, correspondingly. The operadic associativity requires that the two ways are equal. 
It follows from Proposition \ref{propups} which states that $\Upsilon_1=\Upsilon_2$, as well as from the Eckmann-Hilton compatibility stated in Lemma \ref{lemmann3}. 
\end{proof}

\begin{remark}
{\rm
Note that the proof essentially relies on the coherence theorem (Proposition \ref{propskew}) for skew monoidal categories, which is used in the proof of identity $\Upsilon_1=\Upsilon_2$ (Proposition \ref{propups}). This coherence essentially relies on the fact that $\sotimes$ makes $\Cat_\dg(\k)$ a {\it perfect} skew-monoidal category (see Definition \ref{defskewperf}), as in the general case of a skew-monoidal category the coherence theorem is a more sophisticated statement (see [LS], [BL]).
}
\end{remark}

\subsection{\sc The 2-operad $\mathcal{O}$ is homotopically trivial}
Recall the map of complexes 
$$
p_{n_1,\dots,n_k}\colon \mathcal{O}(n_1,\dots,n_k)\to\k[0], n_1,\dots,n_k\ge 1, k\ge 1
$$
which is a quasi-isomorphism, see Proposition \ref{propcontr}.

This map comes from the corresponding quasi-equivalence of dg categories
$$
P_{n_1,\dots,n_k}\colon I_{n_k}\sotimes(I_{n_{k-1}}\sotimes(\dots\sotimes(I_{n_2}\sotimes I_{n_1})\dots))\to I_{n_k}\otimes I_{n_{k-1}}\otimes\dots \otimes I_{n_2}\otimes I_{n_1}
$$ 
which is the projection along the ideal generated by all $\varepsilon(f;\ g_1,\dots,g_n)$, $n\ge 1$. The map $p_{n_1,\dots,n_k}$
is then $P_{n_1,\dots,n_k}(\Hom(\min,\max))$.

\begin{prop}\label{prophomtriv}
The map $p\colon \mathcal{O}(-)\to \k[0]$ is compatible with the operadic composition, and thus gives rise to a map of operads $p\colon \mathcal{O}\to\mathbf{triv}$, where $\mathbf{triv}$ is the trivial 2-operad, $\mathbf{triv}(n_1,\dots,n_k)=\k$, and all operadic compositions are identity maps of $\k$. In other words, the 2-operad $\mathcal{O}$ is homotopically trivial.
\end{prop}
\begin{proof}
Each component of $\mathcal{O}(D)$ is $\mathbb{Z}_{\le 0}$-graded complex, and one sees directly from \eqref{thecomposition} that the operadic composition preserves this grading. \\
If a homogeneous element has a negative degree, it is a sum of monomials each of which contains at least one  $\varepsilon(-;-,..,-)$, and, thus, is mapped to 0 under the map $p$. On the other hand, $\mathbf{triv}(D)=\k$ contains only degree 0 elements for any $D$. 

It is enough to prove that on the degree 0 elements the map $p\colon \mathcal{O}(D)^0\to\mathbf{triv}(D)=\k$ agrees with the operadic composition. It is easy to describe the vector space $\mathcal{O}(D)^0$, $D=(n_1,\dots,n_k)$. It has a basis each element of which consists of the composition of elements $\id*\id*\dots*\id*e_{i,j}*\id*\dots*\id$ in some order, where $e_{i,j}$ runs through elementary morphisms $I_{n_i}(j,j+1)$. All such compositions are corresponded to $(n_1,\dots,n_k)$-shuffle permutations. 
Each basis element is clearly mapped to the only basis element $e$ in $(I_{n_1}\otimes\dots\otimes I_{n_k})(\min,\max)$.
On the other hand, operadic compositions $\Op(\omega; \{\xi_\nu\})$, in which $\omega$ and all $\xi_\nu$ are basis vectors of the type described above, is a basis vector once again. 

\end{proof}

\subsection{\sc The 2-operad $\mathcal{O}$ acts on $\Cat_\dg^\coh(\k)$}

Assume we are given dg categories $C_0,C_1,\dots,C_k\in\Cat_\dg(\k)$, and dg functors
\begin{equation}
\begin{aligned}
\ &F_{10},\dots, F_{1n_1}\colon C_0\to C_1\\
&F_{20},\dots, F_{2n_2}\colon C_1\to C_2\\
&\dots\\
&F_{k1},\dots,F_{kn_k}\colon C_{k-1}\to C_k
\end{aligned}
\end{equation}
(see Figure \ref{fig1}). Assume we are given coherent natural transformations
$$
\Psi_{ij}\colon F_{ij}\Rightarrow F_{i,j+1}\colon C_{i-1}\to C_1,\ \ i=1\dots k,\  j=0\dots {n_i-1}
$$
That is, $\Psi_{ij}\in\Coh_\dg(C_{i-1},C_i)(F_{ij},F_{i,j+1})$. 

Applying successively the twisted composition $M$, we get a dg functor
\begin{equation}\label{psitotbis}
\begin{aligned}
\ &M_\tot\colon Coh_\dg(C_{k-1},C_k)\sotimes \big(Coh_\dg(C_{k-2},C_{k-1})\sotimes\big(\dots\sotimes\big(\Coh_\dg(C_1,C_2)\sotimes \Coh_\dg(C_0,C_1)\big)\dots\big)\big)\\ &\to \Coh_\dg(C_0,C_k)
\end{aligned}
\end{equation}
Denote the l.h.s. of \eqref{psitotbis} by $\Coh_\dg(C_0,C_1,\dots,C_k)$. 
Then \eqref{psitotbis} sends
$$
M_\tot\colon \Coh_\dg(C_0,\dots,C_k)(F_{k0}\times\dots\times F_{10}, F_{kn_k}\times\dots\times F_{1n_1}) 
\to \Coh_\dg(C_0,C_k)(F_{k0}\circ\dots \circ F_{10}, F_{kn_k}\circ\dots\circ F_{1n_1})$$

The question is: how one can assign with the elements $\{\Psi_{ij}\}_{i=1\dots k, j=1\dots n_i}$ an element in the l.h.s. of \eqref{psitotbis}?

Denote $D=(n_1,\dots,n_k)$.

We associate to $\{\Psi_{ij}\}$ as above, and to an element $\omega\in \mathcal{O}(D)$, an element in 
$\Coh_\dg(C_0,\dots,C_k)(F_{k0}\times\dots\times F_{10}, F_{kn_k}\times\dots\times F_{1n_1})$, as follows.

Denote by $e_{ij}$ the generator in $I_{n_i}(j,j+1)$. The element $\omega$ is a sum of monomials in which $e_{ij}$ occurs exactly ones. We can plug $\Psi_{ij}$ for $e_{ij}$, it gives an element in $\Coh_\dg(C_0,C_k)(F_{k0}\circ\dots \circ F_{10}, F_{kn_k}\circ\dots\circ F_{1n_1})$. Denote this element by $m(\omega;\{\Psi_{ij}\})$. 

Define a map
\begin{equation}
\Theta(D)\colon \mathcal{O}(D)\to \underline{\Hom}\Big(\bigotimes_{i,j}\Coh_\dg(C_i,C_{i+1})(F_{ij},F_{i,j+1}),\Coh_\dg(F_{k0}\circ\dots \circ F_{10}, F_{kn_k}\circ\dots\circ F_{1n_1})\Big)
\end{equation}
as
\begin{equation}
\omega\otimes \bigotimes_{i,j}\Psi_{ij}\mapsto m(\omega;\{\Psi_{ij}\})
\end{equation}

\begin{theorem}
The maps $\{\Theta(D)\}$, for $D$ a 2-globular pasting diagram, give rise to an action of the dg 1-terminal 2-operad $\mathcal{O}$ on the dg 2-graph $\Cat_\dg^\coh(\k)$.
\end{theorem}
\begin{proof}
Both the operadic composition and the operadic action on the dg 2-graph $\Cat_\dg^\coh(\k)$ are defined via the substitutions, as well as via the maps $\Upsilon$, see \eqref{mapup}, and the twisted composition $M$. The statement that the operad $\mathcal{O}$ acts on $\Cat_\dg^\coh(\k)$ follows directly from the compatibility 
\eqref{eqcompskew1} of $M$ and the associativity map. The statement that this action is strict unital follows from \eqref{eqcompskew2} and \eqref{eqcompskew3}.
\end{proof}

\appendix

\section{\sc Proofs of Theorem \ref{theorassoc} and Theorem \ref{theorskew}}\label{approofs}
Here we provide  proofs of Theorems \ref{theorassoc} and \ref{theorskew}. 

\vspace{5mm}

{\it Proof of Theorem \ref{theorassoc}}:
It is clear that if $\alpha_{C,D,E}$ gives rise to a dg functor, this dg functor is unique. Indeed, we fixed its value on morphisms which generate $(C\sotimes D)\sotimes E$. In particular, relation $(R_4)$ in Section \ref{section111} implies that 
for any $\phi_1,\phi_2\in C\sotimes D$, $h_1,\dots,h_n\in E$, one has:
\begin{equation}\label{eqassoc2}
\alpha_{C,D,E}(\varepsilon(\phi_2\star\phi_1; h_1,\dots,h_n))=\sum_{0\le a\le n}\pm\alpha_{C,D,E}(\varepsilon(\phi_2; h_{a+1},\dots,h_n))\star \alpha_{C,D,E}(\varepsilon(\phi_1;h_1,\dots,h_a))
\end{equation}

To check that $\alpha_{C,D,E}$ gives rise to a dg functor, one needs to check the following things (where (i)-(iv) are as in the statement of Theorem \ref{theorassoc}):
\begin{itemize}
\item[(s1)]
the compatibility of $\alpha_{C,D,E}$ with the differentials, which include:
\begin{itemize}
\item[(s1.1)] 
\begin{equation}
\alpha_{C,D,E}((d\varepsilon(f; g_1,\dots,g_k))\star \id_Z)=d(\alpha_{C,D,E}(\varepsilon(f;g_k,\dots,g_k)\star \id_Z)
\end{equation}

\item[(s1.2)]
\begin{equation}
\alpha_{C,D,E}(d\varepsilon(f\star\id_Y;h_1,\dots,h_N))=d(\alpha_{C,D,E}(\varepsilon(f\star\id_Y;h_1,\dots,h_N))
\end{equation}
\item[(s1.3)]
\begin{equation}
\alpha_{C,D,E}(d\varepsilon(\id_X\star g;h_1,\dots,h_N)=d(\alpha_{C,D,E}(\id_X\star g;h_1,\dots,h_N))
\end{equation}
\item[(s1.4)]
\begin{equation}
\alpha_{CDE}\Big(d\varepsilon\big(\varepsilon(f;g_1,\dots,g_k);h_1,\dots,h_N\big)\Big)=
d\alpha_{CDE}\Big(\varepsilon\big(\varepsilon(g;g_1,\dots,g_k);h_1,\dots,h_N\big)\Big)
\end{equation}

\end{itemize}
\item[(s2)]
the two expressions for
$$
\alpha_{C,D,E}(\varepsilon(f_1f_2; g_1,\dots,g_k)\star\id_Z)
$$
among which the first one is given by (ii), and the second one is given through $(R_4)$ applied to $\varepsilon(f_1f_2; g_1,\dots,g_k)$ followed by (ii), give rise to equal expressions,
\item[(s3)] the two expressions for 
$$
\alpha_{C,D,E}(\varepsilon((f_1f_2)\star \id_Y;h_1,\dots,h_n))=\alpha_{C,D,E}(\varepsilon((f_1\star \id_Y)(f_2\star\id_Y); h_1,\dots,h_n))
$$
and for
$$
\alpha_{C,D,E}(\varepsilon(\id_X\star(g_1g_2);h_1,\dots,h_n))=\alpha_{C,D,E}(\varepsilon((\id_X\star g_1)(\id_X\star g_2);h_1,\dots,h_n))
$$
among which the first one is given by (iii), and the second one is given through $(R_4)$ followed by (iii), give rise to equal expressions,
\item[(s4)] the two expressions for
$$
\alpha_{C,D,E}(\varepsilon(\varepsilon(f_1f_2; g_1,\dots,g_k);h_1,\dots,h_N)
$$
among which the first one is given by (iv), and the second one is given through $(R_4)$ applied to $\varepsilon(f_1f_2; g_1,\dots,g_k)$ followed by (iv), give rise to equal expressions.
\end{itemize}

One checks (s1)-(s4) by a cumbersome but straightforward computation, and we omit the detail.

\qed

\vspace{5mm}

{\it Proof of Theorem \ref{theorskew}}:

Prove the commutativity of \eqref{eqskewmonn1}. We have to prove that the two maps from 
$$((X\sotimes Y)\sotimes Z)\sotimes W\to X\sotimes (Y\sotimes (Z\sotimes W))$$ from \eqref{eqskewmonn1} coincide. We start with the ``most non-degenerate'' case, and keep track of both maps on 
\begin{equation}\label{eqfu1}
\varepsilon\Big(\varepsilon\big(\varepsilon(f;g,...,g);h,h,...,h\big);s,s,...,s,s\Big)
\end{equation}
(To simplify notations, we drop the lower indices of the arguments; here $f\in X, g\in Y, h\in Z,s\in W$ are morphisms). We check the commutativity up to signs, the coincidence of signs is straightforward. 

The result of composition of upper two arrows in \eqref{eqskewmonn1} on \eqref{eqfu1} is a sum (with appropriate signs) of all possible terms of the form
\begin{equation}\label{eqfu2}
\varepsilon\Big(f;s,s,,...\varepsilon\big(g; s,s,,..,\varepsilon(h;s,s,..,s),s,s,...\big),s,..,\varepsilon(h;s,s,..,s),s,s,s,...,s\Big)
\end{equation}
Each such expression is of the form
\begin{equation}\label{eqfu2.1}
\varepsilon(f; S_1,\dots,S_M)
\end{equation}
where each $S_i$ is either $s$, or $\varepsilon(h;s,s,...,s)$, or $\varepsilon\big(g;s,s,..,s,\varepsilon(h;s,s,..,s),s,s,..,\varepsilon(h;s,s,..,s),s,s\dots\big)$. \\
\\
Let us keep track of the composition of three lower arrows in \eqref{eqskewmonn1}. The first arrow does not affect the $s$-arguments, and produces from \eqref{eqfu1} an expression of the form
\begin{equation}\label{eqfu3}
\varepsilon\Big(\varepsilon\big(f;h,h,...,h,\varepsilon(g;h,...,h),h,h,...,h,\varepsilon(g;h,...,h),...\big);s,s,s,...,s\Big)
\end{equation}
The next map in the lower composition produces 
\begin{equation}\label{eqfu4}
\varepsilon\Big(f;s,s..,s,\varepsilon(h; s,...s), s,s,s,\varepsilon(h; s,s,s,..s),s,s,..,\varepsilon\big(\varepsilon(g;h,h,...h);s,s,...s\big),s,...,s\Big)
\end{equation}
Finally, the third map in the lower composition acts only on the arguments of the form $\varepsilon\big(\varepsilon(g;h,h,...,h),s,s,...,ss\big)$ and maps them as in \eqref{eqassoc4bis}.

What we get finally is an expression of the form
\begin{equation}\label{eqfu4.1}
\varepsilon(f;T_1,T_2,..,T_N)
\end{equation}
where each $T_i$ is either $s$, or $\varepsilon(h;s,s,..,s)$, or $\varepsilon\big(g;s,..,s,\varepsilon(h;s,s,..,s),s,s,...,s,\varepsilon(h;s,s,..,s),s,...\big)$.

We see that there is a 1-to-1 correspondence between the terms in \eqref{eqfu2.1} and \eqref{eqfu4.1}, which proves \eqref{eqskewmonn1} on \eqref{eqfu1}. \\
\\
It remains to prove \eqref{eqskewmonn1} on expressions such as 
$\big((f\star \id)\star \id\big)\star \id$, or $\varepsilon(f;\varepsilon(g;h,\dots,h))\star \id$ (there is a rather long list of all possibilities). \eqref{eqskewmonn1} is straightforward on expressions which do not contain any $\varepsilon$ or contain exactly 1 $\varepsilon$. So we consider all possibilities which contain 2 of $\varepsilon$. Below is the exhaustive list (note that $\varepsilon(\id;...)$ and $\varepsilon(f; g,g,..,\id,g,g..)=0$ by ($R_3$) in Section \ref{section111}):
\begin{equation}\label{eqfu5}
\begin{aligned}
\ &\varepsilon\big(\varepsilon(f;g,g,\dots,g)\star\id_Z;s,s,\dots,s\big)\\
&\varepsilon\big(\varepsilon(f\star \id_Y;h,h,\dots,h);s,s,\dots,s\big)\\
&\varepsilon\big(\varepsilon(\id_X\star g; h,h,\dots,h); s,s,s,\dots,s\big)
\end{aligned}
\end{equation}
The proof of commutativity of \eqref{eqskewmonn1} for all of them is analogous, we provide a proof for the element in the first line of \eqref{eqfu5}, the others are left to the reader.\\
\\
The action of the first arrow in the composition of the two upper arrows of \eqref{eqskewmonn1} on $\varepsilon\big(\varepsilon(f;g,g,\dots,g)\star \id_Z;s,s,\dots,s\big)$ gives
\begin{equation}\label{eqfu6}
\varepsilon\big(\varepsilon(f; g,\dots,g);\id_Z\star s,\id_Z\star s,\dots,\id_Z\star s  \big)
\end{equation}
by \eqref{eqassoc1bisbis}. The application of the second upper map gives 
\begin{equation}\label{eqfu7}
\sum\pm \varepsilon\big(f; \id_Z\star s,\dots,\varepsilon(g; \id_Z\star s,\dots,\id_Z\star s), \id_Z\star s,\dots,\varepsilon(g; \id_Z\star s,\dots,\id_Z\star s),\dots\big)
\end{equation}
where the signs are as in \eqref{eqassoc4bis}.\\
Now keep track of application of the composition of the lower three arrows to the same element.
The application of the first lower arrow gives
\begin{equation}\label{eqfu8}
\varepsilon\big(\varepsilon(f; g\star\id_Z,\dots,g\star\id_Z); s,s,\dots,s\big)
\end{equation}
by \eqref{eqassoc1bis}. Next, the second lower arrow maps \eqref{eqfu8} to
\begin{equation}\label{eqfu9}
\sum\pm \varepsilon\big(\varepsilon(f; s,s,\dots, \varepsilon(g\star\id_Z; s,s,\dots,s),s,\dots,s,\varepsilon(g\star\id_Z;s,s,\dots,s),s,\dots\big)
\end{equation}
Finally, the third lower arrow maps it to 
\begin{equation}\label{eqfu10}
\sum\pm \varepsilon\big(\varepsilon(f; s,s,\dots, \varepsilon(g\star; \id_Z\star s,\id_Z\star s,\dots,\id_Z\star s),s,\dots,s,\varepsilon(g\star;\id_Z\star s,\id_Z\star s,\dots,\id_Z\star s),s,\dots\big)
\end{equation}
(here no new signs emerge, the signs at the corresponding terms are as for \eqref{eqfu9}). 

There is a minor difference between \eqref{eqfu7} and \eqref{eqfu9}; namely, some arguments $s$ in \eqref{eqfu10} emerge as $\id_Z\star s$ in \eqref{eqfu10}. However, in this context they are identically the same, due to Remark \ref{remfu}. \\
The commutativity of diagram \eqref{eqskewmonn1} is proven.\\
The commutativity of the remaining diagrams \eqref{eqskewmonn2}-\eqref{eqskewmonn5} is clear. 

\qed

\section{\sc Batanin 2-operads}\label{sectionba}
The theory of $n$-operads is due to Michael Batanin [Ba3-Ba5], see also [T2]. A contemporary approach uses operadic categories [BM1,2]. In this paper, we adopt the approach and terminology of [Sh5], Section 1.1, the reader is referred to. Here we very briefly outline some points. 
\subsection{}
The category $\Tree_n$ is defined as follows. Its object $T$ is an $n$-string of surjective maps in $\Delta$:
$$
T=[k_n-1]\xrightarrow{\rho_{n-1}} [k_{n-1}-1]\xrightarrow{\rho_{n-2}}\dots\xrightarrow{\rho_0}[0]
$$
Such $T$ is visualised as a $n$-level tree. An $n$-tree is called {\it pruned} if all $\rho_i$ are surjective. For a pruned $n$-tree, all its leaves are at the highest level $n$. The finite set of leaves of an $n$-tree $T$ is denoted by $|T|$.
The maps $\rho_i$ are referred to as the {\it structure maps} of an $n$-tree.

A morphism $F\colon T\to S$, where 
$$
S=[\ell_n-1]\xrightarrow{\xi_{n-1}} [\ell_{n-1}-1]\xrightarrow{\xi_{n-2}}\dots\xrightarrow{\xi_0}[0]
$$
is defined as a sequence of maps $f_i\colon \{k_i-1\}\to \{\ell_i-1\}$, $i=0,1,\dots,n$ (not monotonous, in general), which commute with the structure maps, and such that  
for each $0\le i\le n$ and each $j\in [k_{i-1}-1]$ the restriction of $f_{i}$ on $\rho_{i-1}^{-1}(j)$ is monotonous. That is, $f_i$ has to be monotonous when restricted to the fibers of the structure map $\rho_{i-1}$. 
It is clear that a map of $n$-trees is uniquely defined by the map $f_n$. Conversely, any map $f_n$ which is a map of {\it $n$-ordered sets}, associated with $n$-trees $S$ and $T$, defines a map of $n$-trees (see [Ba3, Lemma 2.3]). 

The fiber $F^{-1}(a)$ for a morphism $F\colon T\to S$, $a\in |S|$, is defined as the set-theoretical preimage of the linear subtree $\Out(a)$ of $S$ spanned by $a$. This linear subtree $\Out(a)$ is defined as follows. Let $a\in [\ell_i]$, then $\Out(a)$ has no vertices at levels $>i$, and the only vertex of $\Out(a)$ at level $j\le i$ is defined as $\rho_j\dots\rho_{i-2}\rho_{i-1}(a)$. Note that a fiber of a map of pruned $n$-trees is not necessarily a pruned $n$-tree, even if all components $\{f_i\}$ of the map $F$ are surjective, see [Sh5, Remark 1.2].

There are two {\it operadic categories} [BM1] based on $n$-trees as objects. The operadic category $\Tree_n$ has as objects all $n$-trees, and the fibers are defined as above. Another one, which we use in this paper, $\Ord_n$, has {\it pruned} $n$-trees as its objects, the morphisms are surjective on leaves, and the fibers are defined as {\it prunisation} of the naive fibers. This procedure of prunisation amounts to ignoring of all leaves at levels $<n$ and all its descendants. 

The ($n-1$-terminal) $n$-operads based on the operadic category $\Ord_n$ are used for describing (weak) $n$-categories {\it with strict units}.
Indeed, the strictness of units  means, for example for $n=2$, that the identity  2-morphisms give rise to identical vertical and horizontal (aka whiskering) compositions.

\subsection{}
\begin{defn}\label{def2op}{\rm
Let $V$ be a symmetric monoidal category. 
An assignment $T\rightsquigarrow \mathcal{O}(T)\in V$, for any pruned $n$-tree $T$, is called a (pruned) $n$-collection in $V$. 
A pruned reduced $(n-1)$-terminal $n$-operad $\mathcal{O}$ in a symmetric monoidal category $V$ is given by an $n$-collection $\{\mathcal{O}(T)\}_{T\in\mathbf{Ord}_n}$, so that for any {\it surjective} map $\sigma\colon T\to S$ of pruned $n$-trees, one is given the composition
\begin{equation}\label{eqopcomp}
m_\sigma\colon \mathcal{O}(S)\otimes \mathcal{O}(P(\sigma^{-1}(1)))\otimes\dots\otimes \mathcal{O}(P(\sigma^{-1}(k)))\to\mathcal{O}(T)
\end{equation}
where $k=|S|$ is the number of leaves of $S$, and $P(-)$ is the prunisation of the corresponding $n$-tree. It is subject to the following conditions (in which we assume that $V=C^\udot(\k)$ is the category of complexes of $\k$-vector spaces):
\begin{itemize}
\item[(i)] $\mathcal{O}(U_n)=\k$, and $1\in \k$ is the operadic unit,
\item[(ii)] the associativity for the composition of two surjective morphisms
$T\xrightarrow{\sigma}S\xrightarrow{\rho} Q$ of pruned $n$-trees, see [Ba2] Def. 5.1,
\item[(iii)] the two unit axioms, see [Ba2], Def. 5.1.
\end{itemize}

The category of pruned reduced $(n-1)$-terminal $n$-operads in a symmetric monoidal category $V$ is denoted by $\Op_n(V)$, or simply by $\Op_n$. 
}
\end{defn}

The $(n-1)$-terminality makes us possible to restrict with $n$-operads taking values in a symmetric monoidal globular category $\Sigma^nV$, where $V$ is a closed symmetric monoidal category, see [Ba2], Sect. 5. By a slight abuse of terminology, we say that an operad takes values in the closed symmetric monoidal category $V$ (not indicating $\Sigma^nV$).

For $n=1$ we recover the concept of a {\it non-symmetric} (non-$\Sigma$) operads. 
To describe $n$-algebras for $n>1$, or $n$-categories, one needs some partial symmetry, but less than the entire symmetric groups $\Sigma_n$ actions on the components $\mathcal{O}(n)$ of an ordinary (symmetric) operads. Morally, with $n$-operads we restrict ourselves with minimal possible symmetry. 

One essential difference with the $n=1$ is that for $n>1$ the corresponding {\it composition product} on the category of (pruned reduced $(n-1)$-terminal) $n$-sequences gives rise only to a {\it skew-monoidal} $V$-category, not to a monoidal one. This phenomenon leads, in particular, to {\it non-associative} graphical presentation of free $n$-operads. 
We refer the reader to [L], Theorem 3.4 where a more general result that collections based on any operadic categories form a skew-monoidal category is proven.

\subsection{}\label{sectionbcomp}
Following [T2], we use in this paper the Joyal dual description of the category $\Trees_n$ via Joyal $n$-disks (more precisely, via {\it $n$-globular pasting diagrams}). This description is more ``globular'', and it fits better for questions of globular nature, such as for describing weak $n$-categories. Let us recall it, for simplicity restricting ourselves with the case $n=2$.

A 2-globular diagram $D$ is given by  sets, $D_0, D_1, D_2$, and the following maps
\begin{equation}\label{globpres}
D_2\doublerightarrow{s_1}{t_1}D_1\doublerightarrow{s_0}{t_0}D_0
\end{equation}
such that $s_0s_1=s_0t_1$, $t_0s_1=t_0s_1$. A morphism $D\to D^\prime$ of two 2-globular diagrams is defined as collection of maps $\{D_i\to D_i^\prime\}_{i=0,1,2}$, which commute with all $s_j$ and $t_j$.

We consider globular $n$-diagrams for which the set $D_0$ is linearly ordered such that for any $f\in D_1$ the element $t_0(f)$ is the least element greater than $s_0(f)$, and similarly $D_1$ is a disjoint union of linearly ordered sets such that for any $\nu\in D_2$ the elements $s_1(\nu), t_1(\nu)\in D_1$ belong to the same connected component, and $t_1(\nu)$ is the least element greater than $s_1(\nu)$. Such diagrams are called  {\it globular pasting $2$-diagram}, see Figure \ref{fig1}.

A 2-globular pasting diagram $D$ is considered as a quiver for generating a free strict 2-category, which we denote by $\omega_2(D)$. The functor $D\rightsquigarrow\omega_2(D)$ is the left adjoint to the forgetful functor from the category of strict 2-categories to 2-globular diagrams.

Assume we are given a globular diagram as at the Figure \ref{fig1}, then one associates to it the 2-tree 
$$T(D)= [n_1+\dots+n_k-1]\xrightarrow{\rho} [k-1]\to [0]$$
where $\sharp\{\rho^{-1}(i)\}=n_i$ and the map $\rho$ is monotonous surjective. The $n_1, n_2,\dots,n_k$ leaves at the corresponding branches of the tree $T(D)$ are thought of as the elementary generating 2-morphisms (thus, the vertical intervals) in $D$.
We denote such globular diagram by $D=(n_1,\dots,n_k)$. As we model the category of pruned 2-trees, we impose the conditions
$$
k\ge 1, n_i\ge 1 \text{  for } 1\le i \le k 
$$

We define a category whose objects are 2-globular diagrams. We define a map $f\colon D\to D^\prime$ as a strict 2-functor between the strict 2-categories $f\colon \omega_2(D)\to \omega_2(D^\prime)$ which is {\it dominant} in the following sense. The objects of $\omega_2(D)$ are linearly ordered, and for any two objects $i\le j$, the set of 1-morphisms $\omega_2(D)(i,j)$ is partially ordered, with the minimal and the maximal elements. The dominance of $f$ means that $\omega_2)(f)$ preserves the minimal and the maximal vertices, and for any $i\le j\in \omega_2(D)$, $\omega_2(f)$ maps the minimal element of $\omega_2(D)(i,j)$ to the minimal element in $\omega_2(D^\prime)(f(i),f(j))$, and similarly for the maximal elements. We denote the category whose objects are 2-globular diagrams, and morphism are strict dominant 2-functors as above, by $\Glob_2^\dom$. For $D\in \Glob_2^\dom$, $D=(n_1,\dots,n_k)$, we assume that $k\ge 1$, $n_1,\dots,n_k\ge 1$. 

One has:
\begin{prop}\label{propduality}
The category $\Glob_2^\dom$ is anti-equivalent to the category $\Tree_2$ of pruned 2-trees.
\end{prop}
See [B], Prop. 2.2 (along with [Sh5], Rem. 1.6 and [T2], 4.1.9-10 for a proof.

The duality of Proposition \ref{propduality} is the $n=2$ analogue of the classical Joyal duality between the category $\Delta_+$ (the usual category $\Delta$ augmented with an initial object [-1]) and the category $\Delta_{fi}$ of {\it finite intervals}, see [J]. 

We adopt the following notations for the 2-globular pasting diagram presentation of 2-operads:

$U,V,...$ for 2-globular pasting diagrams, $[U]$ for $\omega_2(U)$. A {\it minimal ball} in $U$ is an elementary 2-morphism of $\omega_2(U)=[U]$, that is, an element of $U_2$ in the presentation \eqref{globpres}. Minimal balls are denoted by $\mu,\nu,\dots$. The set of all minimal balls in a 2-globular diagram $U$ is denoted by $\mathcal{F}(U)$. 

Denote by $\mathcal{C}(U)$ the set of {\it intervals} of $U$, that is, the set $U_1$ in the presentation \eqref{globpres}. (It is similar to the ordered set of intervals in an ordinal; for the ordinal $[k]=\{0<1<2<\dots<k\}$, the cardinality of the set of intervals is $k$, and the intervals are $(01), (12),\dots, (k-1,k)$. We sometimes write $\overrightarrow{i,i+1}$ for the interval $(i,i+1)$.

For a 2-globular set $U$, there is a map 
$$
\pi(U)\colon \mathcal{F}(U)\to\mathcal{C}(U)
$$
defined as $\pi_U(\nu)=(i,i+1)$ if $s_1(\nu)=i, t_1(\nu)=i+1$.

\subsection{\sc Batanin Theorem}\label{sectionbtheoremm}
Denote the category of symmetric operads (in a given symmetric monoidal category) by $\Op_\Sigma$.

Batanin [Ba2], Sect. 6 and 8, constructs a pair of functors relating symmetric operads and $n$-operads:

$$
\Symm\colon \Op_n^{n-1}\rightleftarrows   \Op_\Sigma \colon \Des
$$
The right adjoint functor of desymmetrisation $\Des$ associates to each pruned $n$-tree $T$ its set of leaves $|T|$ (which are all at the level $n$):
$$
\Des(\mathcal{O})(T)=\mathcal{O}(|T|)
$$
and for a map $\sigma\colon T\to S$ of $n$-trees, the $n$-operadic composition associated with $\sigma$ is defined as the corresponding composition for $|\sigma|=|\sigma_n|\colon |T|\to |S|$, twisted by the shuffle permutation $\pi(\sigma_n)$ of the map $|\sigma_n|\colon |T|\to |S|$ defined by the condition that the composition of $\pi(\sigma)$ followed by an order preserving map of finite sets is $\sigma_n$ (see [Ba2], Sect. 6). 

The symmetrisation functor is defined as the left adjoint to $\Des$, its existence is established in [Ba2], Sect. 8.

The main result on $n$-operads was proven in [Ba3] Th.8.6 for topological spaces and in [Ba3] Th.8.7 for complexes of vector spaces. We provide below the statement for $C^\udot(\k)$, as the one we use here.
Denote by $\underline{\k}$ the constant $n$-operad, $\underline{\k}(T)=\k$, with evident operadic compositions. We say that an $n$-operad in $C^\udot(\k)$ is {\it augmented} by $\underline{\k}$ if there is a map of $n$-operads $p\colon \mathcal{O}\to\underline{\k}$, called the augmentation map. 

\begin{theorem}\label{theoremm}{\rm [Batanin]}
Let $\mathcal{O}$ be reduced pruned $(n-1)$-terminal $n$ operad in the symmetric monoidal category $C^\udot(\k)$. 
Assume $\mathcal{O}$ is augmented to the constant $n$-operad $\underline{\k}$, and that for any arity $T$ the augmentation map $p(T)\colon \mathcal{O}(T)\to \k$ is a quasi-isomorphis of complexes. 
Then there is a morphism of $\Sigma$-operads $C_\udot(E_n;\k)\to\Sym(\mathcal{O})$, thus making any $\mathcal{O}$-algebra a $C_\udot(E_n;\k)$-algebra.
\end{theorem}

\begin{remark}{\rm
There are closed model structures on the categories of $\Sigma$-operads and $n$-operads, constructed in [BB2]. Within these model structures, $(\Symm, \Des)$ is a Quillen pair, with $\Symm$ the left adjoint. The stronger version of this theorem [Ba3] actually says that the symmetrisation of a {\it cofibrant} contractible pruned, reduced, $(n-1)$-terminal is weakly equivalent to the symmetric operad $C^\udot(E_n;\k)$. 
}
\end{remark}

\comment
Here we briefly recall the definition of a Batanin 2-operad and an algebra over it [Ba3-Ba5], [T2], [BM1].

\subsection{\sc Ordinary operads}\label{sectionba1}
We assume the reader has some familiarity with ordinary operads.

Recall the basic definitions of the theory of {\it non-symmetric} operads.

Let $\mathscr{M}$ be a symmetric monoidal category. 
Assume $\mathscr{M}$ has coproducts which are preserved by the monoidal product. Denote by $0$ the initial object in $\mathscr{M}$, and by $e$ the unit object. Consider the category $\mathscr{M}_\coll$ of {\it collections} in $\mathscr{M}$. Its objects are sequence of objects $X_0,X_1,X_2,\dots\in\mathscr{M}$, its morphisms are ``level-wise'' morphisms. 

The category $\mathscr{M}_\coll$ is monoidal, with the monoidal product
\begin{equation}
(X\otimes Y)_n=\coprod_{k\ge 0}\coprod_{\substack{{j_1+\dots+j_k=n}\\{j_s\ge 0}}}X_{j_1}\otimes\dots\otimes X_{j_k}\otimes Y_k
\end{equation}
Its unit is the collection $(0,e,0,0,\dots)$.
An {\it non-symmetric operad} in $\mathscr{M}$ is a monoid in $\mathscr{M}_\coll$. 

The operads are designed for considering the algebras over them. Let $\mathcal{O}$ be an operad in $\mathscr{M}$. Assume $\mathscr{M}$ is closed, denote by $\underline{\Hom}$ the inner $\Hom$ in $\mathscr{M}$ right adjoint to the monoidal product. ``Very classically'', an algebra over $\mathcal{O}$ is an object $X$ of $\mathscr{M}$ such that the collection $\End(X)$, defined as
$$
\End(X)_n=\underline{\Hom}(X^{\otimes n},X)
$$
is an algebra over the monoid $\mathcal{O}$ in $\mathscr{M}_\coll$. That is, there are maps in $\mathscr{M}$
$$
\mathcal{O}_n\to\underline{\Hom}(X^{\otimes n},X), \ n\ge 0
$$
compatible with the operadic compositions. 

``Less classically'', an algebra over $\mathcal{O}$ may me an object of an $\mathscr{M}$-enriched category $\mathbf{K}$. Let $\mathbf{K}$ be a $\mathscr{M}$-enriched category, $X\in \mathbf{K}$ an object. Define the collection
$\mathbf{End}(X)\in\mathscr{M}_\coll$ by
$$
\mathbf{End}(X)_n=\mathbf{Hom}(X^{\otimes n},X)
$$
where we denote by $\mathbf{Hom}(-,-)$ the $\mathscr{M}$-valued $\Hom$ in $\mathbf{K}$. 

{\it An algebra over the operad $\mathcal{O}$} is an object $X\in\mathbf{K}$ such that the collection $\mathbf{End}(X)$ is an algebra over the monoid $\mathcal{O}$. 

This extension of the classical concept of an algebra over an operad as an object in an $\mathscr{M}$-enriched category $\mathbf{K}$ becomes very fruitful in Batanin's definition of higher operads [Ba3].

\subsection{\sc 2-{\bf disk}s}\label{sectionba2}
A {\it 2-{\bf disk}} (see Figure \ref{fig1}) is a finite connected graph $U$ whose vertices are totally ordered and labelled $0,1,\dots,k$,  such that the set of arrows $U(i,j)$ is empty unless $j\ne  i, i+1$, $U(i,i)=\{\id\}$, and
the set of arrows arrows $U(i,i+1)$, $0\le i\le k-1$ is totally ordered. 
\sevafigc{pic1.eps}{100mm}{0}{The 2-{\bf disk} $D=(n_1,\dots,n_k)$\label{fig1}}
Denote $n_{i+1}+1:=\sharp U(i,i+1)$, $1\le i+1\le k$. Then we often write
\begin{equation}\label{2ordform}
U=(n_1,n_2,\dots,n_k)
\end{equation}

We require:
\begin{equation}\label{sizes}
k\ge 1, \text{  and  }n_i\ge 0 \text{  for any  }1\le  i\le k
\end{equation}
A 2-{\bf disk} $U$ is uniquely defined by the ordered set $\{n_1,\dots,n_k\}$.

A 2-{\bf disk} is called {\it non-degenerate} if either $k=1, n_1=0$, or all $n_i\ge 1$. The 2-{\bf disk} with $k=1,n_1=0$ is called $\mathbf{dgnt}$, see Figure \ref{fig6new} (right). 

A 2-{\bf disk} $U$ generates a strict 2-category $[U]$,  as follows. Its objects are $0,1,\dots,k$, a 1-arrow $[U](i,j)$, $j\ge i$, is a sequence $\{f_\ell\in U(\ell,\ell+1)\}_{i\le \ell\le j-1}$ (we denote this 1-morphism as $(f_{j-1}\dots f_i)$), and for two 1-morphisms $f=(f_{j-1}\dots f_i)$ and $h=(h_{j-1}\dots h_i)$ there is a unique 2-morphism $f\Rightarrow h$ iff $f_\ell\le h_\ell$ for any $\ell=i\dots j-1$.

A morphism of 2-{\bf disk}s $P\colon U\to V$ is defined as a 2-functor $[P]\colon [V]\to [U]$ such that:
\begin{itemize}
\item[(i)] the map $[P]$ is monotonous and dominant (that is, $[P](\min)=\min$ and $[P](\max)=\max$) on vertices,
\item[(ii)] for any $i<j$, the induced map of 1-categories $[P]_{i,j}\colon [V](i,j)\to [U](P(i),P(j))$ preserves the minima and the maxima. 
\end{itemize}
The 2-{\bf disk}s form a category denoted by $\Disk_2$.

A morphism of 2-{\bf disk}s $P\colon U\to V$ is called {\it non-degenerate} if $[P]$ is injective on vertices, and is injective on $[P](i,j)$ for any $i<j$. 

The category of 2-{\bf disk}s is dual to the category of pruned 2-level graphs and {\it all} maps between them, see [B, Prop. 2.2], which can be considered as $n=2$ case of the Joyal duality. 

A morphism of 2-{\bf disk}s is called {\it surjective}, if the Joyal dual map of 2-ordinals is injective on the sets of leaves (to translate this definition on the level of {\bf disk}s is less straightforward). 

There is a special 2-{\bf disk} called {\it the globe}, it is the 2-{\bf disk} with $k=1$, $n_1=1$, see Figure \ref{fig6new}. We denote it by $\mathbf{globe}$. It is a final object in the category $\Disk_2$. The category $\Disk_2$ has an initial object: it is the 2-{\bf disk} $\mathbf{dgnt}$, it corresponds to the case $k=1$, $n_1=0$, see Figure \ref{fig6new}).

\sevafigc{pic6.eps}{80mm}{0}{The globe 2-{\bf disk} (left) and the degenerate 2-{\bf disk} (right).\label{fig6new}}

Finally, define a {\it minimal ball} $\nu$ in a 2-{\bf disk} $U$. It is a 2-subcategory in $[U]$, with a single 2-morphism $\nu$, of the form
\begin{equation}\label{eqglobe}
\begin{aligned}
\ &\nu\colon c_j\Rightarrow c_{j+1}\colon i\to i+1\\
&\text{$c_{j+1}$ is the least element greater than $c_j$ in the totally ordered set $[U](i,i+1)$}
\end{aligned}
\end{equation}
As a 2-{\bf disk}, it is isomorphic to the globe, see Figure \ref{fig7new}.

\sevafigc{pic7.eps}{100mm}{0}{A minimal ball in a 2-{\bf disk}.\label{fig7new}}

 By a (not necessarily minimal) ball in $U$ we mean the 2-subcategory of $[U]$ which is formed by all 1- and 2-morphisms  which are ``inside'' of the image of any 2-functor $\mathbf{globe}\to [U]$.
 
We use the notation $\mathcal{F}(U)$ for the set of all {\it minimal balls} of a 2-{\bf disk} $U$.

With a 2-{\bf disk} $U=(n_1,\dots, n_k)$ is associated the {\it finite interval} $[k]$ of its vertices, we denote by $\mathcal{C}(U)$ the set of its elementary intervals $\overrightarrow{1},\overrightarrow{2},\dots,\overrightarrow{k}$ (that is, $\overrightarrow{\ell}=[\ell-1,\ell]$).  Denote by 
$$
\pi_U\colon \mathcal{F}(U)\to \mathcal{C}(U)
$$
the map assigning $\overrightarrow{\ell}$ to a minimal ball $\nu$ if $\nu\in U(\ell-1,\ell)$. All minimal balls $\nu\in\mathcal{F}(U)$ for which $\pi_U(\nu)=\ell$ for the $\ell$-th column of the 2-{\bf disk}. It has a natural structure of a 1-ordinal. 

\begin{remark}{\rm
The category of 2-{\bf disk}s is isomorphic to the dual of the category of 2-ordinals used by Batanin in [Ba3]. The 2-ordinals are visualized as levelled trees. This duality is the 2-dimensional case of the Joyal duality $\Delta_0\simeq (\Delta_+)^\op$, where $\Delta_+$ is the extension of the simplicial category $\Delta$ by the empty ordinal (which is defined as the initial object in $\Delta_+$), and $\Delta_0$ is\footnote{The notation $\Delta_0$ is not conventional, but it has been used by the author.} the subcategory of $\Delta$ having the same objects as $\Delta$ (that is, all finite {\it non-empty} ordinals), and the morphisms of $\Delta_0$ are those morphisms of $\Delta$ which preserve the minima and the maxima of the ordinals. 

The reader is referred to [J] for this duality and for its higher dimensional versions. The 1-dimensional case can be visualized as follows. Let $[n]=\{0<1<\dots<n\}$ be an object of $\Delta_0$. The functor $\Delta_0\to (\Delta_+)^\op$ sends $[n]$ to $[n-1]$. We interpret $[n-1]$ as the ordered set of elementary intervals $[i,i+1]$ in $[n]$, there are $n$ of them (when $n=0$, $[0]$ is mapped to $\emptyset$).

Let $\phi\colon [m]\to [n]\in \Delta_0$. Then $\phi$ defines a map $\phi^*\colon [n-1]\to [m-1]$ on the ordered sets of elementary intervals, changing the direction. One sets:
$$
\phi^*([j,j+1])=[a,a+1]
$$ 
where $a+1$ is the maximal element in $[m]$ such that $\phi(a+1)=j+1$.
}
\end{remark}

\subsection{\sc 2-operads}\label{sectionba3}

In what follows, we recall some of definitions given in [Ba3] in a very specific situation. We consider the case $n=2$ only, our 2-operads are {\it 1-terminal}, and their top level is enriched over the $\Vect_\dg(\k)$.

A {\it 2-sequence} in $\Vect_\dg(\k)$ is a sequence $X(n_1,\dots,n_k)$ of elements in $\Vect_\dg(\k)$, labelled by 2-{\bf disk}s. 
The 2-sequences in $\Vect_\dg(\k)$ form a category, where a map $\phi\colon X\to Y$ is a sequence of maps of complexes $\phi(U)\colon X(U)\to Y(U)$, for the 2-{\bf disk}s $U$.

A {\it reduced unital 2-sequence} $X$ in $\Vect(\k)$ is a sequence $X(n_1,\dots,n_k)$ defined only on non-degenerate {\bf disk}s (that is, either $k=1, n_1=0$, or all $n_i\ge 1$), which extends (and an extension is given) to a presheaf on the category on non-degenerate {\bf disk}s and their surjective maps, and such that $X(\mathbf{dgnt})=X(\mathbf{globe})=\k$.

\begin{remark}
{\rm 
Batanin considers the presheaf on surjective maps of {\bf disk}s as a way to introduce units in $1\in X(\mathbf{dgnt})$, but considering only non-degenerate diagrams (except for $\mathbf{dgnt}$). It is important to require that the whiskering with this 1 preserves components of $X$. Morally, some fibres of a surjective map of {\bf disk}s (=pre-images of some globes) may be degenerate {\bf disk}s$\ne \mathbf{dgnt}$, which we want to exclude from data for $X$. 
}
\end{remark}

\begin{example}{\rm To illustrate what this definition is served for, consider as an example the following injective (and bijective on leaves) map of 2-label trees, see Figure \ref{fig9x}(left). (There are two of such maps, as the map on tips can be either $(1,2)$ or $(2,1)$; we consider the first of them).
\sevafigc{pic9.eps}{120mm}{0}{An injective map of 2-level trees and its Joyal dual map of 2-{\bf disk}s \label{fig9x}}

The map on {\bf disk}s is defined as the map of generated strict 2-categories, in backward direction. 
In Figure \ref{fig9x}(right) it is shown a map of Joyal dual {\bf disk}s. In Figure \ref{fig10x}, we show the two fibres, on the level of trees (left) and on the level of {\bf disk}s (right). 
\sevafigc{fig10x.eps}{120mm}{0}{The fibres of the map of trees (left) and of the Joyal dual map of {\bf disk}s (right)\label{fig10x}}
We see that degenerate trees (resp., {\bf disk}s) appear, which are interpreted as  horizontal products with $1\in X(\mathbf{dgnt})$, that is, the whiskering compositions. }
\end{example}

\vspace{2mm}

The category of reduced unital 2-sequences in $\Vect_\dg(\k)$ forms a {\it skew }monoidal category whose skew monoidal product we are going to consider, see [L, Sect.3] for detail. 

Introduce some notation. Let $U,V$ be 2-{\bf disk}s, and let $P\colon U\to V$ be a map of 2-{\bf disk}s, given as a map $[P]\colon [V]\to [U]$ of 2-categories. Let $\nu$ be a minimal ball in $V$. Its image $[P](\nu)$ is a ball in $U$. It is denoted by $P^{-1}(\nu)$. Clearly it is a 2-{\bf disk}.

Let $X,Y$ be two 2-sequences. Define new 2-sequence $X\circ Y$ as
\begin{equation}
(X\circ Y)(U)=\bigoplus_{P\colon U\to V}X(V)\otimes \big(\bigotimes_{\nu\in\mathcal{F}(V)}Y(P^{-1}(\nu)\big)
\end{equation}
Define a 2-sequence $E$ as
\begin{equation}
E(\mathbf{dgnt})=E(\mathbf{globe})=\k, \ E(n_1,\dots,n_k)=0 \ \text{for  }(n_1,\dots,n_k)\ne \mathbf{dgnt},\mathbf{globe}
\end{equation}
One has:
\begin{lemma}\label{lemmamon1}
The product $X\circ Y$ endows all reduced unital 2-sequences in $\Vect_\dg(\k)$ with a skew monoidal structure. Its skew unit is the 2-sequence $E$.
\end{lemma}
See [L, Sect.3] for a proof. 

\qed

\begin{defn}\label{def2op}{\rm
A reduced {\it 2-operad} $\Theta$ in $\Vect_\dg(\k)$ is a monoid in the skew monoidal category of reduced unital 2-sequences in $\Vect_\dg(\k)$. Unwinding the definition, it is an assignment $U\rightsquigarrow \Theta(U)\in\Vect_\dg(\k)$, where $U$ is a 2-{\bf disk}, such that $\Theta(\mathbf{dgnt})=\Theta(\mathbf{globe})=\k$, and for each morphism $P\colon U\to V$ of 2-{\bf disk}s there is a map of dg vector spaces
\begin{equation}\label{eq2operad}
m(P)\colon \Theta(V)\otimes \bigotimes_{\nu\in\mathcal{F}(V)}\Theta(P^{-1}(\nu))\to\Theta(U)
\end{equation}
which is subject to the following conditions:
\begin{itemize}
\item[(i)] there is a dg map $i\colon \k\to\Theta(\mathbf{globe})$ such that, for the map of 2-{\bf disk}s $\id\colon U\to U$ the composition
$$
\Theta(U)\otimes \k^{\otimes N}  \xrightarrow{-\otimes i^{\otimes N}} \Theta(U)\otimes \Theta(\mathbf{globe})^{\otimes N}\xrightarrow{P}\Theta(U)
$$
(where $N=\sharp\mathcal{F}(U)$) is the identity map, and for the (unique) map $p\colon U\to\mathbf{globe}$, the composition 
$$
\k\otimes \Theta(U)\xrightarrow{i\otimes-}\Theta(\mathbf{globe})\otimes\Theta(U)\xrightarrow{p}\Theta(U)
$$
is the identity map;

\item[(ii)] there is an associativity for two maps $U\xrightarrow{P}V\xrightarrow{Q}W$ of 2-operads, which reads as follows:

The two maps
\begin{equation}\label{2opassoc}
T_1,T_2\colon \Theta(W)\otimes\bigotimes_{\eta\in\mathcal{F}(W)}\Theta(Q^{-1}(\eta))\otimes \bigotimes_{\nu\in\mathcal{F}(V)}\Theta(P^{-1}(\nu))\to\Theta(U)
\end{equation}
defined below, are equal.

The map $T_1$ is equal to the composition
\begin{equation}
\underbrace{\Theta(W)\otimes\bigotimes_{\eta\in\mathcal{F}(W)}\Theta(Q^{-1}(\eta))}\otimes \bigotimes_{\nu\in\mathcal{F}(V)}\Theta(P^{-1}(\nu))\to \Theta(V)\otimes\bigotimes_{\nu\in\mathcal{F}(V)}\Theta(P^{-1}(\nu))\to\Theta(U) 
\end{equation}
For the map $T_2$, fix $\eta\in\mathcal{F}(W)$. One has a map
\begin{equation}
m(P|_{Q^{-1}(\eta)})\colon \ \ \Theta(Q^{-1}(\eta))\otimes\bigotimes_{\nu\in \mathcal{F}(Q^{-1}(\eta))}\Theta(P^{-1}(\nu))\to \Theta(P^{-1}Q^{-1}(\eta))
\end{equation}
Taking the tensor product over all $\eta\in\mathcal{F}(W)$, we get 
\begin{equation}
m_{P,Q}\colon\ \ \bigotimes_{\eta\in\mathcal{F}(W)}\Theta(Q^{-1}(\eta))\otimes\bigotimes_{\nu\in\mathcal{F}(Q^{-1}(\eta))}\Theta(P^{-1}(\nu))\to 
\bigotimes_{\eta\in\mathcal{F}(W)}\Theta(P^{-1}Q^{-1}(\eta))
\end{equation}
Finally, we define the map $T_2$ as
\begin{equation}
\Theta(W)\otimes\underset{m_{P,Q}}{\underbrace{\bigotimes_{\eta\in\mathcal{F}(W)}\Theta(Q^{-1}(\eta))\otimes \bigotimes_{\nu\in\mathcal{F}(V)}\Theta(P^{-1}(\nu))}}\to\Theta(W)\otimes\bigotimes_{\eta\in \mathcal{F}(W)}\Theta(P^{-1}Q^{-1}(\eta))\to\Theta(U)
\end{equation}

\end{itemize}
A {\it map of reduced 2-operads} $\Psi\colon\Theta_1\to\Theta_2$ is defined as a map of the reduced unital 2-sequences compatible with the monoid structures. That is, it is a map of complexes
$\Psi(U)\colon \Theta_1(U)\to\Theta_2(U)$, for all 2-{\bf disk}s $U$, compatible with the units and with the composition maps $\circ_P$, for all maps of 2-{\bf disk}s $P\colon U\to V$. 
}
\end{defn}

\vspace{3mm}

There is the {\it trivial} operad $\mathbf{triv}$, for which $\mathbf{triv}(U)=\k$ for any $U$, and the maps $\circ_P$ are identity maps. 

A 2-operad $\Theta$ is called {\it homotopically trivial} if there is a map of 2-operads $\Psi\colon\Theta\to\mathbf{triv}$ such that $\Psi(U)\colon \Theta(U)\to\k$ is a {\it quasi-isomorphism} of complexes, for any 2-{\bf disk} $U$.

\subsection{\sc Span 2-operads}\label{sectionba4}
A dg 2-graph $\mathscr{C}$ is a dg 1-category $\mathscr{C}$ such that, for each pair
of objects $a,b\in\mathscr{C}$ and each 
pair $f_0,f_1\in \mathscr{C}(a,b)$, there is a dg vector space $\Phi_\mathscr{C}(f_0,f_1)\in \Vect^\dg(\k)$.

Alternatively, a dg 2-graph is given by an ordinary category $C$, and by a function, assigning a dg vector space $\mathscr{C}$ to each globe $g\colon [\mathbf{globe}]_1\to C$ in $C$. In notations of Figure \ref{fig6new}, one has:
$$
a=g(0),\ b=g(1),\ f_0=g(c_0),\ f_1=g(c_1)
$$
We use the notation:
$$
\mathscr{C}(g(\mathbf{globe}))=\Phi_\mathscr{C}(f_0, f_1)
$$

The ordinary category $C$ is said to be the {\it base} of the 2-graph $\mathscr{C}$. We make all dg 2-graphs with fixed base $C$ a category, with external $\HHom$ defined as
$$
\HHom(\mathscr{C}^\prime,\mathscr{C}^\pprime)(g)=\HHom_{\Vect^\dg(\k)}(\mathscr{C}^\prime(g),\mathscr{C}^\pprime(g))
$$
for a globe $g$ in $C$. This category is denoted by $2\Grph(C)$.

A dg 2-graph can be informally thought of as a dg ``pre-2-category'' (whose top level is enriched over dg vector spaces), which is an ordinary 1-category, with given dg vector spaces of ``pre-2-morphisms'', but their vertical and horizontal compositions are not defined yet. 

\begin{example}\label{exdg}{\rm
Let $\mathscr{C}_\coh=\Cat_\coh^\dg(\k)$, and define, for $f_0, f_1\in \Fun_\dg(C,D)$ the complex $\Phi_\mathscr{C}(f_0,f_1)$ as $\Coh_\coh^\dg(C,D)(f_0,f_1)$. It is our main example of a dg 2-graph.
}
\end{example}

For a (dg) 2-category $\mathscr{C}$, we denote by $\mathscr{C}_1$ the underlying ordinary 1-category. As well, for a (dg) 2-graph $\mathscr{C}$ we denote by $\mathscr{C}_1$ the underlying base 1-category. 

Let $U$ be a 2-{\bf disk}, $C$ an ordinary category. By a {\it $U$-diagram in $C$} we mean a functor $D\colon [U]_1\to C$. The set of all $U$-diagrams in $C$ is denoted by $\mathcal{D}_U(C)$.

A {\it globe in an ordinary category} $C$ is a $U$-diagram in $C$ for $U=\mathbf{globe}$.

A globe $D\colon [\mathbf{globe}]_1\to C$ in an ordinary category $C$ is called {\it degenerate}, if, in notations of Figure \ref{fig2}, $D(c_0)=D(c_1)$.

For an ordinary 1-category $C$, one can alternatively define a (dg) 2-graph $\mathscr{C}$ with the base $C$ as a function assigning a dg vector space to each globe in $C$ (that is, to a functor $[\mathbf{globe}]_1\to C$).

Let $D$ be a $U$-diagram in $C$, and 
let $p\colon U\to \mathbf{globe}$ be the unique map. It gives the globe in ${C}$
$$
[\mathbf{globe}]_1\xrightarrow{[p]}[U]_1\xrightarrow{D}{C}
$$
which is denoted by $p_*D$. Informally, the two arrows of the globe $p_*D$ are obtained as the compositions of the images via $D$ of all minima and all maxima 1-morphisms in $U$, correspondingly.

We are ready to define {\it span} 2-operads (whose particular case is formed by the ordinary 2-operads, defined above). The span 2-operads were introduced in [BM1, Sect. 7], though we adapt here a more direct approach. 

\begin{defn}{\rm
A {\it span 2-operad} (for an ordinary category $C$) is a collection $\{\Theta(U)\in 2\Grph(C)\}$ of 2-graphs with base $C$, labelled by the 2-{\bf disk}s $U$, with the following ``operadic composition'': 

Let $U\xrightarrow{P}V$ be a map of 2-{\bf disk}s, $[P]\colon [V]\to [U]$ the corresponding map of 2-categories. Assume we are given a $V$-diagram $D$ in $C$ such that $p_*D=g$. For each minimal ball $\nu\in \mathcal{F}(V)$, denote by $g_\nu$ the  globe $D(\nu)$ in $C$. Then there is a map
\begin{equation}
\circ_{(P,D)}\colon \Theta(V)(g)\otimes\bigotimes_{\nu\in \mathcal{F}(V)}\Theta({P^{-1}(\nu)})(g_\nu)\to \Theta(U)(g)
\end{equation}
which is subject to the properties generalizing those for the case of ordinary 2-operads, see [BM1, Section 7].

}
\end{defn}

\vspace{3mm}

Note that ordinary 2-operads are corresponded to the case of span 2-operads $\Theta$ such that $\Theta(U)(g)$ do not depend on the globe $g\in C$, for all 2-{\bf disk}s $U$. Thus, we can define the ordinary 2-operads without any reference to the base ordinary category $C$.

\subsection{\sc Algebras over 2-operads}\label{sectionba5}
We construct, for a given dg 2-graph $\mathscr{C}$ with base $C$, and a 2-{\bf disk} $U$, a new dg 2-graph $\mathscr{C}^U$ with the same base $C$.

Let $\mathscr{C}$ be a (dg) 2-graph with $\mathscr{C}_1=C$, $U$ a 2-{\bf disk}.
We need to define $\mathscr{C}^U(g)$ for each globe $g\colon [\mathbf{globe}]_1\to C$.
Recall our notation $p_*D$ for a globe in $C$, where $D$ is a $U$-diagram in $C$.
We set
\begin{equation}
\mathscr{C}^U(g):=\bigoplus_{\substack{{D\in\mathcal{D}_U({C})}\\{p_*D=g}}}\ \bigotimes_{\nu\in\mathcal{F}(U)}\mathscr{C}(D(\nu))
\end{equation}
It is, in general, an infinite sum.

Define 
$$
\EEnd(\mathscr{C})(U)=\HHom_{2\Grph(C)}(\mathscr{C}^U,\mathscr{C})
$$
That is, 
$\EEnd(\mathscr{C})$ is a 2-graph whose value at a globe $g$ in $C$ is equal to 
$$
\EEnd(\mathscr{C})(U)(g)=\HHom_{2\Grph(C)}(\mathscr{C}^U,\mathscr{C})(g)=\HHom_{\Vect^\dg(\k)}(\mathscr{C}^U(g),\mathscr{C}(g))
$$

\begin{prop}
The assignment $U\rightsquigarrow\End(\mathscr{C})(U)$ gives rise to a span dg 2-operad $\EEnd(\mathscr{C})$.
\end{prop}
See [T2, Section 5.1] for a proof.

\qed

\begin{defn}{\rm
Let $\Theta$ be a (span or ordinary) dg 2-operad with base $C$. An algebra over $\Theta$ is a dg 2-graph $\mathscr{C}$ with base $C$, such that there is a map of dg 2-operads
\begin{equation}
\Theta\to \EEnd(\mathscr{C})
\end{equation}

}
\end{defn}

\begin{remark}{\rm
In our discussion of dg 2-operads and algebras over them in Sections \ref{sectionba2}-\ref{sectionba5}, we have simplified the things as much as possible, having aimed to define these concepts in the shortest way. However, this way has a conceptual drawback. For instance, the reader may ask in which sense 2-operads and their algebras agree with the scheme for ordinary non-symmetric operads, outlined in Section \ref{sectionba1}. The reader is referred to [Ba3] for a conceptual and thorough approach.
}
\end{remark}

\subsection{\sc Homotopy trivial 2-operads, Batanin's result, and Deligne conjecture}\label{sectionba6}
Batanin links algebras over homotopy trivial $n$-operads with homotopy $n$-algebras in the classical sense (here by a homotopy $n$-algebra we mean an algebra over the chain operad $C_\ldot(E_n,\k)$ of the topological operad $E_n$).

To state the result, recall some definitions. We use the language of globular monoidal categories here, see [Ba3].

Let $\mathscr{C}_*$ be an $n$-globular monoidal category, see [Ba3, Sect. 2].
It is given, in particular, by categories $\mathscr{C}_0,\dots,\mathscr{C}_n$, and functors $s_{k,i},t_{k,i}\colon \mathscr{C}_k\to \mathscr{C}_i$, $0\le i<k\le n$, called {\it the source} and {\it the target} functors, and functors $z_i\colon \mathscr{C}_{i-1}\to \mathscr{C}_{i}$, $i\le n$, called {\it the cylinder} functor, and having the meaning of the identity $i$-morphism of the corresponding $(i-1)$-morphism. 

Recall that a globular object $X$ in it is defined as a globular functor $I\to \mathscr{C}_*$ where $I$ is the terminal $n$-globular category (that is, the categories $I_0,\dots,I_n$ contain a unique objects and the only identity morphism). 

Given an $n$-globular monoidal category $\mathscr{C}_*$, one can define the truncated $n$-globular monoidal category $\mathscr{C}^{(k)}_*$, for any $k\le n$. By definition, the categories $\mathscr{C}^{(k)}_0,\dots, \mathscr{C}^{(k)}_{k-1}$ are terminal categories, and for $m\ge k$ the category $\mathscr{C}^{(k)}_m$ is the full subcategory of $\mathscr{C}_m$, formed by objects $x\in \mathscr{C}_m$ such that $s_{m,k-1}x=t_{m,k-1}x=z^{k-1}(*)$. 

Finally, a globular object $X\colon I\to\mathscr{C}_*$ is called {\it $(k-1)$-terminal}, if the functor $X$ factors as $I\to \mathscr{C}^{(k)}_*\to \mathscr{C}_*$.

\vspace{2mm}

{\sc Example.} Let us see what all these definitions mean in our favourite situation discussed in Example \ref{exdg}.

We have $n=2$. The 2-globular monoidal category $\mathscr{C}_*$ is the category $\Span_\dg$ of spans of sets, whose top level is enriched in dg vector spaces over $\k$, [Ba3, Sect. 3, Ex. 5]. That is, the category $\mathscr{C}_i=\Sets\sqcup \Sets$, $i=0,1$, and $\mathscr{C}_2=\Vect_\dg(\k)$.

\sevafigc{picspan.eps}{70mm}{0}{An element in $\Span$.\label{figspan}}

The objects of are shown in Figure \ref{figspan}. Here $C_{0,*}, C_{1,*},C_2$ are sets (where $*\in\{\ell,r\}$). Note that in this diagram the sets $C_{0,\ell}$ and $C_{0,r}$, as well as the sets $C_{1,\ell}$ and $C_{1,r}$, are {\it distinct}. 

For a globular object in $\Span_\dg$, we have a {\it single} set $X_0=C_{0,\ell}=C_{0,r}$, a {\it single} set $X_1=C_{1,\ell}=C_{1,r}$, and a dg vector space $X_2$, with maps $s,t\colon X_1\to X_0$, $s,t\colon X_2\to X_1$.

Describe the globular object $X$, corresponded to our problem. 
Define $X_0$ as ``the set'' of all small dg categories over $\k$, and $X_1$ as ``the set'' of all dg functors. For each dg functor $f\colon c_0\to c_1$, we have its source and its target dg categories, which are $s(f)=c_0$ and $t(f)=c_1$. 
Denote by $X_1(c_0,c_1)$ the fibre of the map $X_1\xrightarrow{s\times t} X_0\times X_0$ over $(c_0,c_1)$.
The dg vector space $X_2$ is defined as
$$
X_2=\prod_{c_0,c_1\in X_0}\prod_{f,g\in X_1(c_0,c_1)}\Coh(f,g)
$$
It is clear that $X_2$ is embodied as the top level of the diagram on Figure \ref{figspan}. (More precisely, there are maps $s,t\colon \sharp X_2\to X_1$, there $\sharp(-)$ stands for the underlying set).

This is ``the biggest'' among the globular objects in $\mathscr{C}$ we consider. Let us define some other ``smaller'' globular objects $Y$ in $\mathscr{C}$.

Let $c$ be a dg category. Define the globular object $Y(c)$, depending on $c$, as follows:
$Y(c)_0=\{c\}$, the 1-element set we interpret as the dg category $c$, that is, $Y(c)_0\subset X_0$. Next, define $Y(c)_1$ to be the set of all dg functors $f\colon c\to c$, and
$$
Y(c)_2=\prod_{f,g\in Y(c)_1}\Coh(f,g)
$$

The globular objects $Y(c)$ in $\mathscr{C}$ are 0-terminal.

One can define a family of even smaller globular objects $Z(c)$ in $\mathscr{C}$, depending on a dg category $c$.

For a dg category $c$, define $Z(c)_0=\{c\}$, $Z(c)_1=\{\id_c\}$, and $Z(c)_2=\Coh(\id_c,\id_c)$. The globular objects $Z(c)$ are 1-truncated.

\vspace{2mm}

In [Ba4], Batanin raises up the following question. Assume $X$ is a globular object in an $n$-globular monoidal category $\mathscr{C}$, such that there is an $n$-operad $\mathcal{O}$ acting on $X$. 
Assume for simplicity that $\mathscr{C}=\Span$, and that the top level of $X$ is enriched in dg vector spaces (this assumption is not necessary, see [Ba4, Sect.2] for discussion of the general case).
If $X$ is $k$-truncated, there should be some ``partial symmetrisation'' of $\mathcal{O}$ acting on $X$. In the very extreme case, when $X$ is $(n-1)$-truncated, $X$ is given by a single dg vector space $X_\top$. One can suppose that there is a {\it symmetrisation functor} $\Sym$, from $n$-operads to symmetric operads, such that, if an $n$-operad $\mathcal{O}$ acts on an $(n-1)$-terminal object $X$, the symmetric operad $\Sym(\mathcal{O})$ acts on the top component $X_\top$.

The symmetric operad $\Sym(\mathcal{O})$, for an $n$-operad $\mathcal{O}$, is constructed in [Ba4, Sect. 13].

The case when $\mathcal{O}$ is homotopically trivial was considered in [Ba5]. The following result was proven in [Ba5, Theorem 8.7]:

\begin{theorem}\label{theorembatanin}
Assume that $\mathcal{O}$ is a cofibrant homotopically trivial reduced $(n-1)$-terminal $n$-operad in $\Vect_\dg(\k)$. Then its symmetrisation $\Sym(\mathcal{O})$ is weakly equivalent to the chain operad $C_\ldot(E_n,\k)$. 
\end{theorem}

\qed

Note that the 2-operad $\mathcal{O}$ in $\Vect_\dg(\k)$, constructed in Section \ref{sectionfg}, is not cofibrant, but there is always exists a cofibrant replacement $R(\mathcal{O})$ of it, which acts on $X$ as soon as $\mathcal{O}$ does, and to which Theorem \ref{theorembatanin} is applicable. 

\vspace{2mm}

In our favourite example, let $c$ be a dg category, and let $Z(c)$ be the corresponding 1-terminal globular object in $\Span_\dg$. It is given by a dg vector space $X_\top$, equal to $\Coh(\id_c,\id_c)$, which is isomorphic to the Hochschild cohomological complex $\Hoch^\udot(c)$. Let $\mathcal{O}$ be either the dg 2-operad constructed here or the one of [T2]. It is homotopically trivial and acts on $Z(c)$. Therefore, $\Sym(\mathcal{O})\sim C_\ldot(E_2,\k)$ acts on $X_\top\simeq \Hoch^\udot(c)$. It is (a more general version of) the classical Deligne conjecture, cf. [T2, Sect. 7].

\endcomment

\comment
\section{\sc The signs}\label{sectionapp}
Here we provide an account on the signs in the formulas with Hochschild cochains. This Appendix makes no claim to originality, our main intention here is to understand the signs in a systematic way. Although we only consider here the case of the Hochschild cochain complex $C^\udot(A,A)$ of a dg associative $A$, the generalization for $\Coh(C,D)(F,G)$ we deal with in this paper, is straightforward.

\subsection{\sc The brace operad acts on $\prod_{n\ge 0}\underline{\Hom}(V^{\otimes n},V)$}\label{app11}
Let $V$ be a complex over $\k$, we denote its differential by $d_\dg$. 

Each graded space $\underline{\Hom}(V^{\otimes n},V)$ inherits the differential; it is
\begin{equation}\label{eqapp1}
d_\dg(\Psi)(a_1,\dots,a_n)=d_\dg(\Psi(a_1,\dots,a_n))-\sum_{i=1}^n(-1)^{|\Psi|_0+|a_1|+\dots+|a_{i-1}|}\Psi(a_1,\dots,d_\dg a_i,\dots,a_n)
\end{equation}
(for homogeneous $\Psi$ and $a_1,\dots,a_n$; we denote by $|\Psi|_0$ the degree of $\Psi$ as a map of graded vector spaces, and $|a_i|$ denotes the grading of $a_i\in V$).

Denote $X(V)=\prod_{n\ge 0}\underline{\Hom}(V^{\otimes n},V)$.

Recall the brace operad $\Br$ ([GJ], see also [MS1, Sect.1]). 

It is a dg operad generated by a binary operation $\cup$ of degree +1, and an $n$-ry operation $\{\}_n$ of degree 0, $n\ge 1$, such that $\cup$ is associative, $\{\}_1=\id$,
\begin{equation}
\partial\circ \cup+\cup\circ_1 \partial+\cup\circ_2 \partial=0
\end{equation}
\begin{equation}
\{\}_{n+1}\circ_1\cup=\sum_{k=0}^{n-1}(\{\}_{k+1}\cup \{\}_{n-k})\circ \tau
\end{equation}
where $\tau$ is the permutation that shuffles the second argument to the $(k+2)$ position,
\begin{equation}
[\partial,\{\}_n]=-(\cup\circ_2\{\}_{n-1})\circ\tau +\Big(\sum_{i=1}^{n-1}\{\}_{n-1}\circ_i\cup \Big)-(\cup\circ_1\{n-1\})
\end{equation}
where $\tau$ is the transposition switching the first and the second arguments,
\begin{equation}
\{\}_{n+1}\circ_1\{\}_{m+1}=\sum_{\substack{{i_1,\dots,i_m}\\{j_1,\dots,j_m}}}\{\}_{m+1+\ell}\Big(\id,\dots,\id,\underset{i_1+1}{\{\}_{j_1-i_1}},\id,\dots,\id,\underset{i_2+1}{\{\}_{j_2-i_2}},\id,\dots\dots,\id,\underset{i_m+1}{\{\}_{j_m-i_m}},\id,\dots,\id\Big)\circ\tau
\end{equation}
where $\tau$ is the the corresponding shuffle permutation, and $\ell=i_1+\dots+i_m-j_1-\dots-j_m+n$.

We consider the $(n+1)$-ary brace operations $f\{g_1,\dots,g_n\}$ of degree 0  on $X(V)$ (where the degree is defined as the degree of the corresponding map of graded vector spaces) defined as
\begin{equation}\label{eqapp2}
\begin{aligned}
\ &(f\{g_1,\dots,g_n\})(a_1,\dots,a_N)=\\
&\sum_{i_1,\dots,i_{n}}(-1)^{\varepsilon}f(a_1,\dots,a_{i_1},{g_1}(a_{i_1+1},\dots,a_{j_1}),a_{j_1+1},\dots,\dots,{g_n}(a_{i_{n}+1},\dots,a_{j_{n}}),a_{j_{n}+1},\dots a_N)
\end{aligned}
\end{equation}
where $\varepsilon=\sum_{\ell=1}^n|g_\ell|_0(\sum_{s\le i_\ell}  |a_s|)$.

We get an action of the dg operad $\Br$ on $X(V)$, when $\cup(f\otimes g)=f\cdot g=0$:
\begin{equation}\label{eqapp3}
(f\cdot g)\cdot h=(-1)^{|f|_0}f\cdot(g\cdot h)
\end{equation}
\begin{equation}\label{eqapp4}
(f_1\cdot f_2)\{g_1,\dots,g_n\}=\sum_{k=0}^n(-1)^{|f_2|_0(\sum_{j=1}^k|g_j|_0)}f_1\{g_1,\dots,g_k\}\cdot f_2\{g_{k+1},\dots,g_n\}
\end{equation}
\begin{equation}\label{eqapp5}
\begin{aligned}
\ &f\{g_1,\dots,g_m\}\{h_1,\dots,h_N\}=\\
&\sum_{i_1,\dots,i_n} (-1)^\varepsilon f\big\{h_1,\dots,h_{i_1}, g_1\{h_{i_1+1},\dots,h_{j_1}\},h_{j_1+1},\dots,h_{i_m},g_m\{h_{i_m+1},\dots,h_{j_m}\},\dots,h_N\big\}
\end{aligned}
\end{equation}
where $\varepsilon=\sum_{i=1}^m|g_i|_0(\sum_{s\le i_m}|h_s|_0)$.

As well, these operations should agree with the differential $d$ (equal to $d_\dg$ in our case) as
\begin{equation}\label{eqapp6}
[d,f\cdot g]=0
\end{equation}
and
\begin{equation}\label{eqapp7}
\begin{aligned}
\ &[d,f\{g_1,\dots,g_n\}]=-(-1)^{|g_1|_0|f|_0}g_1\cdot f\{g_2,\dots,g_n\}+\\
&\sum_{j=1}^{n-1}(-1)^{\sum_{s<j}|g_s|_0}f\{g_1,\dots,g_j\cdot g_{j+1},\dots,g_n\}-f\{g_1,\dots,g_{n-1}\}\cdot g_n
\end{aligned}
\end{equation}
(recall that $m$ has degree 1).

\subsection{\sc The case of $\prod_{n\ge 0}\underline{\Hom}(V^{\otimes n},V)[-n+1]$}\label{app12}
Consider the isomorphism of complexes of vector spaces
\begin{equation}\label{eqapp10}
\Phi\colon \prod_{n\ge 0}\underline{\Hom}(V[1]^{\otimes n},V[1])\to \prod_{n\ge 0}\underline{\Hom}(V^{\otimes n},V)[-n+1]
\end{equation}
The complex $V[1]$ has the differential $-d_\dg$.

The r.h.s. of \eqref{eqapp10} gets the differential $d_\dg^\prime$ defined as $d_\dg^\prime(\Psi[-n+1])=(-1)^{n-1}d_\dg(\Psi)$. That is,
\begin{equation}
\begin{aligned}
\ &d_\dg^\prime(\Psi[-n+1])(a_1,\dots,a_n)=\\
&(-1)^{n-1}d_\dg(\Psi(a_1,\dots,a_n))-(-1)^{|\Psi|_0+n-1}\sum_{i=1}^n(-1)^{|a_1|_0+\dots+|a_{i-1}|_0}\Psi(a_1,\dots,d_\dg a_i,\dots,a_n)
\end{aligned}
\end{equation}
for $\Psi\in\underline{\Hom}(V^{\otimes n},V)$.

The same $\Psi$ defines $\Psi_{[1]}\in\underline{\Hom}(V[1]^{\otimes n},V[1])$, as
$$
\Psi_{[1]}(b_1,\dots,b_n)=\Psi(b_1[-1],\dots,b_n[-1])[1]
$$
It has degree $|\Psi_{[1]}|=|\Psi|_0+n-1$.

The l.h.s. of \eqref{eqapp10} gets the differential which is a particular case of the one considered in Section \ref{app11}. 
We have
\begin{equation}
\begin{aligned}
\ &d_\dg(\Psi_{[1]})(b_1,\dots,b_n)=\\
&-d_\dg(\Psi_{[1]}(b_1,\dots,b_n))+(-1)^{|\Psi|_0+n-1}\sum_{i=1}^n(-1)^{|a_1|_0+\dots+|a_{i-1}|_0+i-1}\Psi_{[1]}(a_1,\dots,d_\dg a_i,\dots,a_n)
\end{aligned}
\end{equation}

This two differentials are different, and we need to implement a sign correction to $\Phi$ to make it a map of complexes. The reason for that is that the operad $\Br$ acts on the l.h.s. of \eqref{eqapp10}, the conventional differential is the one on the r.h.s.

\begin{lemma}\label{lemmaapp1}
Within the differentials as above, the map $\Phi$ is not a map of complexes. The map $\Phi^\prime$ defined as
\begin{equation}
\Phi^\prime(\Psi)(a_1,\dots,a_n)=(-1)^{(n-1)|a_1|+(n-2)|a_2|+\dots+|a_{n-1}|}\Psi(a_1[1],\dots,a_n[1])
\end{equation}
is a map of complexes.
\end{lemma}

For example, let $A$ be a dg associative algebra, $m(a_1,a_2)=a_1a_2$. If we want it to remain $d_\dg$-closed in the l.h.s. of \eqref{eqapp10},  it should be replaced by 
$$
m^\prime(a_1,a_2)=(-1)^{|a_1|}a_1a_2
$$
On the other hand, at the r.h.s. of \eqref{eqapp10} it is closed without this correction.

The left-hand side of \eqref{eqapp10} is a particular case of what we considered in Section \ref{app11}, therefore there is an action of $\Br$ on the right-hand side. 

For $a_1,\dots,a_n\in V$, and $f,g_1,\dots,g_n\in \prod_n\underline{\Hom}(V^{\otimes n},V)[-n+1]$, we set
\begin{equation}\label{eqapp11}
\begin{aligned}
\ &(f\{g_1,\dots,g_n\})(a_1,\dots,a_N)=\\
&\sum_{i_1,\dots,i_{n}}(-1)^{\varepsilon}f(a_1,\dots,a_{i_1},{g_1}(a_{i_1+1},\dots,a_{j_1}),a_{j_1+1},\dots,\dots,{g_n}(a_{i_{n}+1},\dots,a_{j_{n}}),a_{j_{n}+1},\dots a_N)
\end{aligned}
\end{equation}
with $\varepsilon=\sum_{\ell=1}^n(|g_\ell|-1)(\sum_{s\le i_{\ell}}  (|a_s|+1))$. 

Here we use notation  $|f|=|f|_0+\ell$ for $f\in \underline{\Hom}(V^{\otimes \ell},V)$.
It agrees with \eqref{eqapp2} and \eqref{eqapp10}.

Note that \eqref{eqapp3}-\eqref{eqapp7} hold, after changes $|a_i|$ by $|a_i|+1$, and $|f_i|_0$ by $|f_i|-1$. 
\endcomment

\comment
\begin{remark}\label{rem1app}{\rm
Note that, having a cochain in $\underline{\Hom}(V^{\otimes n},V)$ which is $d_\dg$-closed, it fails to remain $d_\dg$-closed after its formal shift , which belongs to $\underline{\Hom}(V[1]^{\otimes n},V[1])\simeq \underline{\Hom}(V^{\otimes n},V)[-n+1]$. Indeed, in the latter case the differential \eqref{eqapp1} is replaced by
\begin{equation}\label{eqapp12}
(d_\dg[1](\Psi))(a_1,\dots,a_n)=d_\dg(\Psi(a_1,\dots,a_n))-(-1)^{|\Psi|-1}\sum_{i=1}^n(-1)^{|a_1|+\dots+|a_{i-1}|}\Psi(a_1,\dots,d_\dg a_i,\dots,a_n)
\end{equation}
Here $|a_i|=|a_i|_0+1$ and $|\Psi|=|\Psi|_0+n$, and this degree shift affects the equation $d_\dg\Psi=0$. 

However, the situation is fixed by replacing $\Psi\in \underline{\Hom}(V^{\otimes n},V)$ by 
\begin{equation}\label{eqapp13}
\Psi_{[n-1]}(a_1,\dots,a_n):=(-1)^{(n-1)|a_1|_0+(n-2)|a_2|_0+\dots+|a_{n-1}|_0}\Psi(a_1,\dots,a_n)
\end{equation}
In fact, the correspondence $\Psi\mapsto\Psi_{[n-1]}$ intertwines $d_\dg$ and $d_\dg[1]$.

Formula \eqref{eqapp13} is well-known, it appears (at least) in [GJ] in many places. 
}
\end{remark}
\endcomment

\comment
\subsection{\sc The localization}\label{app12bis}
One can localize a $\Br$-algebra by a Maurer-Cartan element, as follows. 

There is a map of operads $\Lie\to \Br$; the Lie bracket $[f,g]$ is
$$
[f,g]=f\{g\}-g\{f\}\circ \tau
$$
where $\tau=(1,2)$ is the permutation of the arguments. (It gives a sign such as $(-1)^{|f|_0|g|_0}$ for the case of $X(V)$).

Let $m\in X(V)^1$ satisfies the Maurer-Cartan equation 
$$
d m+\frac12[m,m]=0
$$

Let $X$ be an algebra over $\Br$. Define a new algebra over $\Br$ replacing $d\mapsto d+[m,-]$, 
replacing $f\cdot g$ by $f\cdot g+m\{f,g\}$
and maintaining the brace operations  $f\{g_1,\dots,g_n\}, n\ge 0$, without changes. 

One checks that in this way we get a new brace algebra. 

Now we switch to the ``real'' example of $C^\udot(A,A)[1]$ where $A$ is an associative dg algebra. 

Consider the underlying complex $A_f$, and take $X(A_f[1])$ as in Section \ref{app11}. It is a $\Br$-algebra. 

Take the Maurer-Cartan element $m(a_1,a_2)=(-1)^{|a_1|_0}a_1a_2$, see Lemma \ref{lemmaapp1} for the sign correction.

Finally, localize $X(A_f[1])$ by $m$. 

In this way, we get the well-known formulas for the shifted Hochschild cochain complex $C^\udot(A,A)[1]$.
We write them down below for the reader's reference. 

\begin{equation}
\begin{aligned}
\ &(d_\Hoch[1]\Psi)(a_1,\dots,a_{n+1})=(-1)^{|a_1||\Psi|+|\Psi|+1}a_1\Psi(a_2,\dots,a_{n+1})+\\
&\sum_{j=1}^n(-1)^{|\Psi|+1+\sum_{s=1}^j(|a_s|+1)}\Psi(a_1,\dots,a_ja_{j+1},\dots,a_{n+1})+(-1)^{|\Psi|+\sum_{i=1}^n(|a_i|+1)}
\Psi(a_1,\dots,a_n)a_{n+1}
\end{aligned}
\end{equation}
\begin{equation}
(\Psi_1\cdot \Psi_2)(a_1,\dots,a_{m+n})=(-1)^{|\Psi_1|-1+(|\Psi_2|-1)(|a_1|+\dots+|a_m|+m)}\Psi_1(a_1,\dots,a_m)\Psi_2(a_{m+1},\dots,a_{m+n})
\end{equation}
\begin{equation}
\begin{aligned}
\ &\Theta\{\Psi_1,\dots,\Psi_m\}(a_1,\dots,a_N)=\sum_{i_1,\dots,i_m}(-1)^{ \sum_{\ell=1}^m(|\Psi_\ell|-1)\cdot \sum_{s\le i_\ell}(|a_s|+1)}   \\
&\Theta(a_1,\dots,\Psi_1(a_{i_1+1},\dots,a_{j_1}),a_{j_1+1},\dots,\Psi_2(a_{i_2+1},\dots,a_{j_2}),a_{j_2+1},\dots,\Psi_m(a_{i_m+1},\dots,a_{j_m}),\dots,a_N)
\end{aligned}
\end{equation}
\begin{equation}\label{eqpropmtbis}
\begin{aligned}
\ &\partial(\Theta\{\Psi_1,\dots,\Psi_m\})=(-1)^{(|\Theta|-1)(|\Psi_1|-1)}\Psi_1\cdot \Theta\{\Psi_2,\dots,\Psi_m\}-\\
&\sum_{j=1}^{m-1}(-1)^{|\Theta|-1+\sum_{s< j}(|\Psi_s|-1) }\Theta\{\Psi_1,\dots,\Psi_j\cdot\Psi_{j+1},\dots,\Psi_m\}+
\Theta\{\Psi_1,\dots,\Psi_{m-1}\}\cdot\Psi_m
\end{aligned}
\end{equation}
\begin{equation}\label{eqabcbr1}
(\Theta_1\cdot\Theta_2)\{\Psi_1,\dots,\Psi_m\}=\sum_{\ell=0}^m(-1)^{(|\Theta_2|-1)(\sum_{s=1}^\ell(|\Psi_s|-1))}\Theta_1\{\Psi_1,\dots,\Psi_\ell\}\cdot \Theta_2\{\Psi_{\ell+1},\dots,\Psi_m\}
\end{equation}
\begin{equation}\label{eqabcbr2}
\begin{aligned}
\ &(\Theta\{\Psi_1,\dots,\Psi_m\})\{\Gamma_1,\dots,\Gamma_N\}=\sum_{i_1,\dots,i_m}(-1)^{\sum_{\ell=1}^m(|\Psi_\ell|-1)(\sum_{s\le i_\ell}(|\Gamma_s|-1)}\\
&\Theta\{\Gamma_1,\dots,\Gamma_{i_1},\Psi_1\{\Gamma_{i_1+1},\dots,\Gamma_{j_1}\},\Gamma_{j_1+1},\dots,\Gamma_{i_2},\Psi_2\{\Gamma_{i_2+1},\dots,\Gamma_{j_2}\},\dots\}
\end{aligned}
\end{equation}
These operations define an action of the operad $\Br$ on $C^\udot(A,A)[1]$.

\subsection{\sc From $C(A,A)[1]$ to $C(A,A)$}\label{app13}
In this paper, we work with the unshifted Hochschild cochain complex $C^\udot(A,A)$, as well as with the ``categorified'' versions of it $\Coh_\dg(C,D)(F,G)$. We would like to write down some formulas for $C^\udot(A,A)$ with correct signs we use in the main body of the paper. 

\begin{remark}\label{rembraceshift}{\rm
We draw reader's attention that we do {\it not } deal with the shifted brace operad $\Br\{-1\}$. We produce from $\Br$ a coloured operad $\Br_\mathrm{col}$ with 2 colors, such that it acts on $(C^\udot(A,A),C^\udot(A,A)[1])$. Indeed, in the brace operation $\Theta\{\Psi_1,\dots,\Psi_k\}$ only $\Theta$ is being shifted by $[-1]$, but $\Psi_1,\dots,\Psi_k$ remain the same. 
The correct notation should be $\Theta[-1]\{\Psi_1,\dots,\Psi_k\}$. However, to avoid any confusion, we use the notation $\Theta\{\Psi_1,\dots,\Psi_k\}_{[-1]}$. 

In these notations, the Tsygan map is
\begin{equation}
C^\udot(A,A)\to C^\udot(C^\udot(A,A),C^\udot(A,A))
\end{equation}
$$
\Theta\mapsto \Theta\{\Psi_1,\dots,\Psi_k\}_{[-1]}
$$
It has degree 0 and preserves the operations, see [Ts, Prop.3 and 4]. The original map in loc.cit. was seemingly a map
$$
C^\udot(A,A)[1]\to C^\udot(C^\udot(A,A),C^\udot(A,A))[1]
$$
so the changes we made are minor. 
}
\end{remark}

We write down all ``shifted'' formulas explicitly.

\begin{equation}
\begin{aligned}
\ &(d_\Hoch\Psi)(a_1,\dots,a_{n+1})=(-1)^{|a_1||\Psi|+|\Psi|}a_1\Psi(a_2,\dots,a_{n+1})+\\
&\sum_{j=1}^n(-1)^{|\Psi|+\sum_{s=1}^j(|a_s|+1)}\Psi(a_1,\dots,a_ja_{j+1},\dots,a_{n+1})+(-1)^{|\Psi|+1+\sum_{i=1}^n(|a_i|+1)}
\Psi(a_1,\dots,a_n)a_{n+1}
\end{aligned}
\end{equation}
\begin{equation}
(\Psi_1\cup \Psi_2)(a_1,\dots,a_{m+n})=(-1)^{|\Psi_2|(|a_1|+\dots+|a_m|+m)}\Psi_1(a_1,\dots,a_m)\Psi_2(a_{m+1},\dots,a_{m+n})
\end{equation}
(where $\psi_1\cup \psi_2=(-1)^{|\psi_1|-1}\psi_1\cdot \psi_2$).
\begin{equation}
\begin{aligned}
\ &\Theta\{\Psi_1,\dots,\Psi_m\}_{[-1]}(a_1,\dots,a_N)=\sum_{i_1,\dots,i_m}(-1)^{ \sum_{\ell=1}^m(|\Psi_\ell|-1)\cdot \sum_{s\le i_\ell}(|a_s|+1)}   \\
&\Theta\Big(a_1,\dots,\Psi_1(a_{i_1+1},\dots,a_{j_1}),a_{j_1+1},\dots,\Psi_2(a_{i_2+1},\dots,a_{j_2}),a_{j_2+1},\dots,\Psi_m(a_{i_m+1},\dots,a_{j_m}),\dots,a_N\Big)
\end{aligned}
\end{equation}
(see Remark \ref{rembraceshift}).

This operations fulfil the following identities.

\begin{equation}\label{eqapp6x}
[d_\tot,f\cup g]=0
\end{equation}

\begin{equation}\label{eqapp7x}
\begin{aligned}
\ &[d_\tot,f\{g_1,\dots,g_n\}_{[-1]}=(-1)^{(|g_1|-1)|f|}g_1\cup f\{g_2,\dots,g_n\}_{[-1]}\\
&-\sum_{j=1}^{n-1}(-1)^{|f|+\sum_{s=1}^{j}(|g_j|-1)}f\{g_1,\dots,g_j\cup g_{j+1},\dots,g_n\}_{[-1]}
+(-1)^{|f|+|g_1|+\dots+|g_{n-1}|-n}f\{g_1,\dots,g_{n-1}\}_{[-1]}\cup g_n
\end{aligned}
\end{equation}
(where $d_\tot$ is the total differential).

\begin{equation}\label{eqapp3x}
(f\cup g)\cup h=f\cup (g\cup  h)
\end{equation}
\begin{equation}\label{eqapp4x}
\begin{aligned}
\ &(f_1\cup f_2)\{g_1,\dots,g_n\}_{[-1]}=
\sum_{k=0}^n(-1)^{\sum_{i=1}^k|f_2|(|g_k|-1)}f_1\{g_1,\dots,g_k\}_{[-1]}\cup f_2\{g_{k+1},\dots,g_n\}_{[-1]}
\end{aligned}
\end{equation}
\begin{equation}\label{eqapp5x}
\begin{aligned}
\ &f\{g_1,\dots,g_m\}_{[-1]}\{h_1,\dots,h_N\}_{[-1]}=
\sum_{i_1,\dots,i_m} (-1)^{\sum_{\ell=1}^m|g_m|\cdot (\sum_{s\le i_\ell}(|h_s|-1))} \\
&f\big\{h_1,\dots,h_{i_1}, g_1\{h_{i_1+1},\dots,h_{j_1}\}_+,h_{j_1+1},\dots,h_{i_m},g_m\{h_{i_m+1},\dots,h_{j_m}\}_+,\dots,h_N\big\}_{[-1]}
\end{aligned}
\end{equation}

\endcomment

\comment

In this paper, we deal with operations on $C^\udot(A,A)$ itself rather than on $C^\udot(A,A)[1]$.
Apparently, there are two sources of additional signs which emerge in the passage from $C^\udot(A,A)[1]$ to $C^\udot(A,A)$. 

In general, if an operad $\mathcal{O}$ in $\Vect_\dg(\k)$ acts on $X\in \Vect_\dg(k)$, the operad $\mathcal{O}\{1\}$ acts on $X[1]$, where $\mathcal{O}\{1\}$ is the {\it operadic shift}.

Recall that $\mathcal{O}\{1\}(n)=\mathcal{O}(n)[-n+1]$, with the action of the symmetric group $\Sigma_n$ on $\mathcal{O}\{1\}(n)$ twisted by the sign representation $sgn_n$. One views $\mathcal{O}\{1\}$ as the (level) tensor product with the $\End$-operad
\begin{equation}\label{opshift0}
\mathcal{O}\{1\}=\underline{\End}(\k[1])\otimes \mathcal{O}
\end{equation}

The operadic compositions should be also suitably sign-corrected:
\begin{equation}\label{opshift1}
(O_m\{1\})\circ_i (O_{n}\{1\})=(-1)^{(i-1)(n-1)+|O_m|(n-1)}(O_m\circ_i O_n)\{1\}
\end{equation}
where $O_\ell\in O(\ell)$. The latter formula is used to describe in terms of generators and relations the shifted operad of an operad given in terms of generators and relations.

The sign \eqref{opshift1} can be deduced from \eqref{opshift0} as follows. First of all, 
\begin{equation}
\underline{\End}(\k[1])_m\circ_i \underline{\End}(\k[1])_n=(-1)^{(i-1)(n-1)}\underline{\End}(\k[1])_{m+n-1}
\end{equation}
where $\underline{\End}(\k)_\ell\in\underline{\End}(\k)(\ell)$ is the generator obtained from the generator in $\k$.

Next, for any two operads $\mathcal{O}$ and $\mathcal{O}^\prime$, one has by definition:
$$
(\mathcal{O}_m\otimes\mathcal{O}^\prime_m)\circ_i(\mathcal{O}_n\otimes\mathcal{O}^\prime_n)=(-1)^{|\mathcal{O}^\prime_m||\mathcal{O}_n|}(\mathcal{O}_m\circ_i\mathcal{O}_n)\otimes (\mathcal{O}^\prime_m\circ_i\mathcal{O}^\prime_n)
$$

For general compositions, one deduces:
\begin{equation}
\mathcal{O}_m\{1\}\circ (\mathcal{O}_{n_1}\{1\}\otimes\dots\otimes\mathcal{O}_{n_m}\{1\})=(-1)^{\sum_{\ell=1}^m(n_\ell-1)(\ell-1)+\sum_{\ell=1}^m(n_\ell-1)\cdot( \sum_{s<\ell}\mathcal|{O}_s|)}(\mathcal{O}_m\circ(\mathcal{O}_{n_1}\otimes\dots\mathcal{O}_{n_m}))\{1\}
\end{equation}

We want to make precise the way how the shifted operad $\mathcal{O}\{1\}$ acts on $X[1]$. It is {\it not} true that the action $O_n\{1\}(f_1[1],\dots,f_n[1])$ of an operation $O_n\{1\}\in \mathcal{O}\{1\}(n)$ is just $O_n(f_1,\dots,f_n)[-n+1]$. 

The correct formula is:
\begin{equation}\label{opshift2}
O_n\{1\}(f_1[1],\dots,f_n[1])=(-1)^{(n-1)|f_1|+(n-2)|f_2|+\dots+|f_{n-1}|+n|O_n|}O_n(f_1,\dots, f_n)[1]
\end{equation}
(The reader can check that \eqref{opshift1} and \eqref{opshift2} are consistent).

Consider the case of the operad $\Br$ acting on $C^\udot(A,A)[1]$, we want the operad $\Br\{-1\}$ to act on $C^\udot(A,A)$. There is no problem with the differential $d_\dg$, as we are now in the situation of the {\it r.h.s.} of \eqref{eqapp10}; we only replace $d_\dg$ by $d_\dg[1]=-d_\dg$.

(To avoid any confusion, we denote the brace operation on the complex $C^\udot(A,A)$ by $\{\}_+$, and the cup-product by $\cup$, instead of $\cdot$).

These operations obey the relations in the operad $\Br\{-1\}$:
\begin{equation}\label{eqapp3x}
(f\cdot g)\cdot h=f\cdot(g\cdot h)
\end{equation}
\begin{equation}\label{eqapp4x}
\begin{aligned}
\ &(f_1\cup f_2)\{g_1,\dots,g_n\}_+=\\
&\sum_{k=0}^n(-1)^{(n-k)|f_1|+k|f_2|+(n-k)(|g_1|+\dots+|g_k|-k)+k+\sum_{i=1}^k|f_2||g_k|}f_1\{g_1,\dots,g_k\}_+\cup f_2\{g_{k+1},\dots,g_n\}_+
\end{aligned}
\end{equation}
\begin{equation}\label{eqapp5x}
\begin{aligned}
\ &f\{g_1,\dots,g_m\}_+\{h_1,\dots,h_N\}_+=
\sum_{i_1,\dots,i_m} (-1)^{\sum_{\ell=1}^m(|g_m|\cdot \sum_{s\le i_\ell}|h_s|+i_\ell)+\epsilon(i_1,\dots,i_m,j_1,\dots,j_m)} \\
&f\big\{h_1,\dots,h_{i_1}, g_1\{h_{i_1+1},\dots,h_{j_1}\}_+,h_{j_1+1},\dots,h_{i_m},g_m\{h_{i_m+1},\dots,h_{j_m}\}_+,\dots,h_N\big\}_+
\end{aligned}
\end{equation}
with 
\begin{equation}\label{sterrible}
\begin{aligned}
\ &\epsilon(i_1,\dots,i_m,j_1,\dots,j_m)=\sum_{s=1}^m(j_s-i_s)(|f|-1)+\sum_{\ell=1}^m(\ell-i_\ell+\sum_{s=\ell+1}^m(j_s-i_s))(|g_\ell|-1)+\\
&\sum_{\ell=0}^m\sum_{s=j_\ell+1}^{i_{\ell+1}}(m-\ell+\sum_{t=\ell+1}^m(j_t-i_t))(|h_s|-1)+
\sum_{\ell=1}^m\sum_{s=i_\ell+1}^{j_\ell}(m-\ell-s+\sum_{t=\ell+1}^m(j_t-i_t))(|h_s|-1)
\end{aligned}
\end{equation}
where we set $j_0=0$, $i_{m+1}=N$.

where
$N^\prime=N-J+I+m$, $J=j_1+\dots+j_m$, $I=i_1+\dots+i_m$. 
\begin{equation}
\begin{aligned}
\ &
N^\prime(|f|-1)+\\
&\sum_{s=1}^{i_1}(N^\prime-s)(|h_s|-1)+\\
&(N^\prime-i_1-1)(|g_1|+|h_{i_1+1}|+\dots+|h_{j_1}|-j_1+i_1-1)+\\
&\sum_{s=j_1+1}^{i_2}(N^\prime+j_1-i_1-1-s)(|h_s|-1)+\\
&(N^\prime+j_1-i_1-i_2-2)(|g_2|+|h_{i_2+1}|+\dots+|h_{j_2}|-j_2+i_2-1)+\\
&\sum_{s=j_2+1}^{i_3}(N^\prime+(j_1-i_1)+(j_2-i_2)-2-s)(|h_s|-1)+\dots\\
&\sum_{s=j_\ell+1}^{i_{\ell+1}}(N^\prime+(j_1-i_1)+\dots+(j_\ell-i_\ell)-\ell-s)(|h_s|-1)+\\
&(N^\prime+(j_1-i_1)+\dots+(j_\ell-i_\ell)-\ell-i_{\ell+1}-1)(|g_{\ell+1}|+|h_{i_{\ell+1}+1}|+\dots+|h_{j_{\ell+1}}|-j_{\ell+1}+i_{\ell+1}-1)+\dots
\end{aligned}
\end{equation}

\begin{equation}
\begin{aligned}
r.h.s.=&(N^\prime-1)(|f|-1)+\\
&\sum_{\ell=1}^m(N^\prime+(j_1-i_1)+\dots+(j_{\ell-1}-i_{\ell-1})-j_\ell-\ell)(|g_\ell|-1)+\\
&\sum_{\ell=0}^m\sum_{s=j_{\ell}+1}^{i_{\ell+1}}(N^\prime+(j_1-i_1)+\dots+(j_\ell-i_\ell)-\ell-s)(|h_s|-1)+\\
&\sum_{\ell=1}^m\sum_{s=i_\ell+1}^{j_{\ell}}(N^\prime+(j_1-i_1)+\dots+(j_{\ell-1}-i_{\ell-1})-i_\ell-\ell)(|h_s|-1)
\end{aligned}
\end{equation}
where we set $j_0=0$, $i_{m+1}=N$.

\begin{equation}\label{eqapp6x}
[d_\tot,f\cup g]=0
\end{equation}

\begin{equation}\label{eqapp7x}
\begin{aligned}
\ &[d_\tot,(-1)^{|f|-1}f\{g_1,\dots,g_n\}_+]=(-1)^{|g_1||f|+n}g_1\cup f\{g_2,\dots,g_n\}_+\\
&-\sum_{j=1}^{n-1}(-1)^{1+\sum_{s=1}^{j-1}|g_j|}f\{g_1,\dots,g_j\cup g_{j+1},\dots,g_n\}_+
+(-1)^{|g_1|+\dots+|g_{n-1}|-n+1}f\{g_1,\dots,g_{n-1}\}_+\cup g_n
\end{aligned}
\end{equation}
(where $d_\tot$ is the total differential).

\endcomment

 \comment
\section{\sc Proof of Theorem \ref{theorassoc}}\label{sectionproofta}
Here we complete the proof of Theorem \ref{theorassoc} started in Section \ref{sectionassoc}. We check the compatibility with the differential (s1) in Section \ref{sectionta1}, check statements (s2)-(s3) in Section \ref{sectionta2}, and check statement (s4) in Section \ref{sectionta3}.

\subsection{\sc Compatibility with the differential (s1)}\label{sectionta1}
We check the compatibility the differential for the hardest case of \eqref{eqassoc4bis}, the cases of \eqref{eqassoc1}, \eqref{eqassoc1bis}, \eqref{eqassoc1bisbis} are substantially easier and are left to the reader.

At first, we check this compatibility up to the signs, writing them as $\pm$; after that, we compare the signs with the corresponding terms. 

We need to prove:
\begin{equation}\label{check1}
\alpha_{CDE}\Big(d\varepsilon\big(\varepsilon(f;g_1,\dots,g_k);h_1,\dots,h_N\big)\Big)=
d\alpha_{CDE}\Big(\varepsilon\big(\varepsilon(g;g_1,\dots,g_k);h_1,\dots,h_N\big)\Big)
\end{equation}

The action of the differential on $\varepsilon\big(\varepsilon(f;g_1,\dots,g_k);h_1,\dots,h_N\big)$  is
\begin{equation}\label{check2}
\begin{aligned}
\ & d\Big(\varepsilon\big(\varepsilon(f; g_1,\dots,g_k), h_1,\dots,h_N\big)\Big)=\\
&\pm \varepsilon\big(\varepsilon(f;g_1,\dots,g_k); h_2,\dots,h_N\big)\circ  (\id\star h_1)+\sum_{j=1}^{N-1}\pm \varepsilon\big(\varepsilon(f;g_1,\dots,g_k);h_1,\dots, h_{j+1}\circ h_j,\dots, h_N\big)+\\
&\pm (\id\star h_N)\circ \varepsilon\big(\varepsilon(f; g_1,\dots,g_k); h_1,\dots,h_{N-1}\big)+\\
&\varepsilon\big(d(\varepsilon(f;g_1,\dots,g_k));h_1,\dots,h_N\big)+\sum_{j=1}^N\pm \varepsilon\big(\varepsilon(f;g_1,\dots,g_k);h_1,\dots, d(h_j),\dots,h_N\big)
\end{aligned}
\end{equation}
The term $d\varepsilon(f;g_1,\dots,g_k)$ is further expressed by ???. We get
\begin{equation}\label{check3}
\begin{aligned}
\ & d\Big(\varepsilon\big(\varepsilon(f; g_1,\dots,g_k), h_1,\dots,h_N\big)\Big)=\\
&\pm \varepsilon\big(\varepsilon(f;g_1,\dots,g_k); h_2,\dots,h_N\big)\circ  (\id\star h_1)+\sum_{j=1}^{N-1}\pm \varepsilon\big(\varepsilon(f;g_1,\dots,g_k);h_1,\dots, h_{j+1}\circ h_j,\dots, h_N\big)+\\
&\pm (\id\star h_N)\circ \varepsilon\big(\varepsilon(f; g_1,\dots,g_k); h_1,\dots,h_{N-1}\big)+\\
&\pm \underline{\varepsilon\big((\id\star g_k)\circ \varepsilon(f;g_1,\dots,g_{k-1}); h_1,\dots,h_N\big)}+\sum_{j=1}^{k-1}\pm\varepsilon\big(\varepsilon(f; g_1,\dots, g_{j+1}\circ g_j,\dots,g_k);h_1,\dots,h_N\big)+\\
&\underline{\varepsilon\big(\varepsilon(f;g_2,\dots,g_k)\circ (\id\star g_1);h_1,\dots,h_N\big)}+\\
&\varepsilon\big(\varepsilon(df;g_1,\dots,g_k);h_1,\dots,h_N\Big)+\sum_{j=1}^k\pm \varepsilon\big(\varepsilon(f;g_1,\dots, dg_j,\dots,g_k);h_1,\dots,h_N\big)+\\
&\sum_{j=1}^n\pm \varepsilon\big(\varepsilon(f;g_1,\dots,g_k);h_1,\dots, dh_j,\dots, h_N\big)
\end{aligned}
\end{equation}
Finally, the two underlined terms are expressed by ???. We get:
\begin{equation}\label{check4}
\begin{aligned}
\ & d\Big(\varepsilon\big(\varepsilon(f; g_1,\dots,g_k), h_1,\dots,h_N\big)\Big)=\\
&\pm \varepsilon\big(\varepsilon(f;g_1,\dots,g_k); h_2,\dots,h_N\big)\circ  (\id\star h_1)+\sum_{j=1}^{N-1}\pm \varepsilon\big(\varepsilon(f;g_1,\dots,g_k);h_1,\dots, h_{j+1}\circ h_j,\dots, h_N\big)+\\
&\pm (\id\star h_N)\circ \varepsilon\big(\varepsilon(f; g_1,\dots,g_k); h_1,\dots,h_{N-1}\big)+
\sum_{s=0}^N \pm \varepsilon\big((\id\star g_k);h_{s+1},\dots,h_N\big)\circ \varepsilon\big((f;g_1,\dots,g_{k-1}); h_1,\dots,h_s\big)+\\
&\sum_{j=1}^{k-1}\pm\varepsilon\big(\varepsilon(f; g_1,\dots, g_{j+1}\circ g_j,\dots);h_1,\dots,h_N\big)+
\sum_{s=0}^N\pm\varepsilon\big(\varepsilon(f;g_2,\dots,g_k);h_{s+1},\dots,h_N\big)\circ\varepsilon\big((\id\star g_1);h_1,\dots,h_s\big)+\\
&\varepsilon\big(\varepsilon(df;g_1,\dots,g_k);h_1,\dots,h_N\Big)+\sum_{j=1}^k\pm \varepsilon\big(\varepsilon(f;g_1,\dots, dg_j,\dots,g_k);h_1,\dots,h_N\big)+\\
&\sum_{j=1}^n\pm \varepsilon\big(\varepsilon(f;g_1,\dots,g_k);h_1,\dots, dh_j,\dots, h_N\big)
\end{aligned}
\end{equation}

To compute the l.h.s. of \eqref{check1}, we apply $\alpha_{CDE}$ to the r.h.s. of \eqref{check4}, making use that $\alpha_{CDE}$ maps compositions to compositions. Introduce notation:
$$
\Theta(f; g_1,\dots,g_k; h_1,\dots,h_N):=\text{the r.h.s. of \eqref{eqassoc4bis}}
$$
One has:
\begin{equation}\label{check5}
\begin{aligned}
\ &\alpha_{CDE}\Big(d\varepsilon\big(\varepsilon(f;g_1,\dots,g_k);h_1,\dots,h_N\big)\Big)=\\
&\pm \Theta(f;g_1,\dots,g_k;h_2,\dots,h_N)\circ (\id\star h_1)+\sum_{j=1}^{N-1}\pm\Theta(f;g_1,\dots,g_k;h_1,\dots,h_{j+1}\circ h_j,\dots,h_N)+\\
&\pm(\id\star h_N)\circ \Theta(f;g_1,\dots,g_k;h_1,\dots,h_{N-1})+\sum_{s=0}^N\pm (\id\star\varepsilon(g_k;h_{s+1},\dots,h_N))\circ \Theta(f;g_1,\dots,g_{k-1};h_1,\dots,h_s)+\\
&\sum_{j=1}^{k-1}\pm\Theta(f;g_1,\dots, g_{j+1}\circ g_j,\dots,g_k;h_1,\dots,h_N)+\sum_{s=0}^N\pm\Theta(f;g_2,\dots,g_k;h_{s+1},\dots,h_N)\circ (\id\star\varepsilon(g_1;h_1,\dots,h_s))+\\
&\Theta(df;g_1,\dots,g_k;h_1,\dots,h_N)+\sum_{j=1}^k\pm\Theta(f;g_1,\dots,dg_j,\dots,g_k;h_1,\dots,h_N)+\\
&\sum_{j=1}^N\pm\Theta(f;g_1,\dots,g_k;h_1,\dots,dh_j,\dots,h_N)
\end{aligned}
\end{equation}
Now we compute the r.h.s. of \eqref{check1}.
\begin{equation}\label{check6}
\begin{aligned}
\ &d\Big(\alpha_{C,D,E}\big(\varepsilon(\varepsilon(f; g_1,\dots,g_k);h_1,\dots,h_N)\big)\Big)=\\
&\sum_S\pm d\Big(\varepsilon\big(f;\   h_1,\dots, h_{i_1}, \varepsilon(g_1;h_{i_1+1},\dots,h_{i_2}), h_{i_2+1},\dots, h_{i_3},\varepsilon(g_2;h_{i_3+1},\dots,h_{i_4}), h_{i_4+1},\dots,h_{i_5},\dots \big)\Big)
\end{aligned}
\end{equation}
We compute the differential in the second line of \eqref{check6} by successive use of\eqref{eqd1}. We distinguish the following cases:
\begin{itemize}
\item[(a1)] $i_1=0$,
\item[(a2)] $i_1>0$,
\item[(b1)] $i_{2k}=N$,
\item[(b2)] $i_{2k}<N$.
\end{itemize}
For the case $(a1)$, the leftmost extreme boundary term in \eqref{eqd1} is 
\begin{equation}\label{check7}
\pm\Theta(f;g_2,\dots,g_k;h_{i_2+1},\dots,h_N)\circ (\id\star\varepsilon(g_1;h_1,\dots,h_{i_2}))
\end{equation}
while in the case $(a2)$ this term is
\begin{equation}\label{check8}
\pm \Theta(f;g_1,\dots,g_k;h_2,\dots,h_N)\circ (\id\star h_1)
\end{equation}
and similarly for the rightmost extreme terms in \eqref{eqd1} in $(b1)$ and $(b2)$ cases. 

The terms such as 
\begin{equation}\label{check9}
\varepsilon\big(f; h_1,\dots, h_{i_1}, \underline{\varepsilon(g_1;h_{i_1+1},\dots,h_{i_2-1})\circ (\id\star h_{i_2})}, h_{i_2+1},\dots,\big)
\end{equation}
(obtained from the extreme boundary terms of plugged-in $\varepsilon(g_j; h_{i_{2j-1}+1},\dots, h_{i_{2j}})$, such as the underlined fragment in \eqref{check9}), are cancelled, as follows.

The summand $S$ (in notations of \eqref{eqassoc4bis}) which is displayed in \eqref{check9} corresponds to $m_1=i_1,\ell_1=i_2-i_1,\dots$ (in the notations introduced below \eqref{eqassoc4bis}).
Consider the summand $S^\prime$, corresponded to $m_1^\prime=i_1,\ell^\prime_1=i_2-i_1-1,\dots$. Then the summand in \eqref{check6}, corresponded to $S^\prime$, has the summand equal to the one displayed in \eqref{check9}, which is obtained from a regular (not an extreme) component in the differential \eqref{eqd1}.

After these remarks we see that the second line of \eqref{check6} has exactly the same summands as the r.h.s. of \eqref{check5}.

The equality of signs is proven by a straightforward but cumbersome check. ???

\subsection{\sc Check of (s2)-(s3)}\label{sectionta2}

\subsection{\sc Check of (s4)}\label{sectionta3}

\endcomment

\bigskip

\noindent{\small
 {\sc Euler International Mathematical Institute\\
10 Pesochnaya Embankment, St. Petersburg, 197376 Russia }}

\bigskip

\noindent{{\it e-mail}: {\tt shoikhet@pdmi.ras.ru}}

\end{document}